\documentclass{article}
\pdfoutput=1
\usepackage[utf8]{inputenc}
\usepackage{fancyhdr, hyperref, amsmath, amssymb, graphicx, wrapfig, amsthm, xcolor, bbold, enumitem, tcolorbox, wasysym, mathtools, tikz, scalerel, paracol, shortcuts}
\usepackage[margin=1in]{geometry}
\usepackage{fourier}
\usepackage{amsmath}
\usepackage{float}
\usepackage{anyfontsize}

\title{Skein Theory for Affine A Subfactor Planar Algebras}
\author{Melody Molander}
\date{}

\begin{document}

\maketitle
\begin{abstract}
 The Kuperberg Program asks to find presentations of planar algebras and use these presentations to prove results about their corresponding categories purely diagrammatically. This program has been completed for index less than 4 and is ongoing research for index greater than 4. We give generators-and-relations presentations for the affine $A$ subfactor planar algebras of index 4. Exclusively using the planar algebra language, we give new proofs to how many such planar algebras exist. Categories corresponding to these planar algebras are monoidally equivalent to cyclic pointed fusion categories. We give a proof of this by defining a functor yielding a monoidal equivalence between the two categories. The categories are also monoidally equivalent to a representation category of a cyclic subgroup of $SU(2)$. We give a new proof of this fact, explicitly using the diagrammatic presentations found. This gives novel diagrammatics for these representation categories. 
\end{abstract}
{
\tableofcontents}
\section{Introduction}
 
A \textit{von Neumann algebra} is a unital algebra of bounded operators on a Hilbert space closed under a specific topology and possessing an \textit{adjoint} that can be thought of as a generalized matrix transpose. When the Hilbert space is finite-dimensional, von Neumann algebras are necessarily direct sums of the algebras of $n\times n$ matrices over $\mathbb{C}$ (with varying $n$). \textit{Factors}, which are von Neumann algebras with center isomorphic to $\mbb{C}$, are building blocks of von Neumann algebras. A unital inclusion of factors, $N\subseteq M$, is called a \textit{subfactor}. The \textit{standard invariant} of a subfactor is the 2-category capturing bimodule data coming from $M$ and $N$. When $N\subseteq M$ is \textit{finite depth} and both $N$ and $M$ are the \textit{hyperfinite $II_1$ factor}, the standard invariant is a complete invariant of subfactors \cite{Pop94}. 

Subfactors have another invariant, a real number called the \textit{index}. Jones \cite{Jon83} found that the set of all indices of subfactors is the set $\{4\cos^2\left(\frac{\pi}{n}\right)|n\geq 3\}\cup [4,\infty]$. Amazingly, techniques used in his proof of this led to the \textit{Jones polynomial} link invariant \cite{Jon85} and spawned the field of quantum topology.

Subfactor theory can be thought of as a noncommutative version of Galois theory. Whereas Galois theory studies inclusions of fields $F \subseteq E$, subfactors are unital inclusions, $N\subseteq M$, of factors, which are highly noncommutative algebras. Both fields and factors satisfy that any map between them are unital inclusion or zero, so analyzing maps between fields or factors is the same as studying field extensions or subfactors. The index $[M:N]$ of a subfactor measures the size of these extensions in a similar notion to the degree $[E:F]$ of a field extension. 

The standard invariant is analogous to the Galois group. Let $\mc{R}$ be the hyperfinite $II_1$ factor and $G$ be a finite group. There is exactly one way, up to conjugacy, in which $G$ can act on $\mc{R}$ by outer automorphisms \cite{Jon80, Con77}. Let $N=\mc{R}^G$, the fixed points of the action. $N$ can be shown to be a factor. There will be a Galois correspondence between intermediary subfactors $N\subseteq M \subseteq \mc{R}$ and subgroups $H\leq G$ \cite{NT60a, NT60b}. Just like in Galois theory, the automorphisms fixing $N$ of $M$, $\te{Aut}_N(M)$, equals $G$. Additionally, $|G|=[\mc{R}:N]$. However, not all examples of finite index subfactors come from groups.

Decades of work have been done to classify subfactors of small index. For an overview see \cite{JMS14, AMP23}. A \textit{subfactor planar algebra}, introduced by Jones \cite{Jon99}, is a diagrammatic way to view the standard invariant of a subfactor. 2-morphisms can be drawn as pictures in the plane, reminiscent to braid diagrams. Multiplication of 2-morphisms is given by vertical stacking. Tensor product is given by horizontal concatenation. Existence of some subfactors were proved using subfactor planar algebras \cite{BPMS12, MP15}. The planar algebra framework has opened up a new point of view for the study of subfactors.

Kuperberg posed a program to see how far the planar algebra language can be pushed in the understanding of the standard invariants of subfactors. 

\textit{The Kuperberg Program:} Give a presentation by generators-and-relations for every subfactor planar algebra and prove as much as possible about the planar algebra using only this presentation.

For subfactors of index less than 4 there is an ADE classification by their \textit{principal graphs}. See \cite{Ocn88} for an outline of the proof. The Kuperberg program has been completed for all possible indices less than 4. Type $A$ subfactor planar algebras are well-understood as the \textit{Temperley-Lieb} planar algebras. Type $D_{2n}$ was completed by Morrison, Peters, and Snyder \cite{MPS10} in 2010. Type $E_6$ and $E_8$ presentations were given shortly after by Bigelow \cite{Big10}. 

For index larger than 4, constructing presentations of known subfactors is an ongoing effort. A presentation for the subfactor planar algebras with principal graphs \textit{Haagerup} and its dual was constructed by Peters \cite{Pet10} and \textit{Extended Haagerup} and its dual was constructed by Bigelow, Morrison, Peters, and Snyder \cite{BPMS12}. Type \textit{2221} was constructed by Han \cite{Han10} and types \textit{4442, 3333, 3311, and 2221} were constructed by Morrison and Penneys \cite{MP15}.

The aim of this paper is to address some of the index 4 subfactor planar algebras. Popa \cite{Pop94} found that the principal graphs of index 4 subfactor planar algebras are exactly the simply-laced affine ADE Dynkin diagrams. We give the presentations of all affine $A$ subfactor planar algebras of index 4.

To satisfy the Kuperberg program, we then use these presentations to prove we have found all planar algebras of type $A$ as well as show equivalences to other interesting categories. After finding the planar algebra presentations given in this paper, the author saw similarities to the presentation given in \cite{Cze24}. In the paper, Czenky gives diagrammatics for cyclic pointed fusion categories, $\te{Vec}_{\mbb{Z}_m}^\omega$. The fusion categories $\te{Vec}_{\mbb{Z}_m}^\omega$ are not strict when the associator $\omega$ is not the identity. However, every monoidal category is monoidally equivalent to a strict monoidal category. From the planar algebra presentations we construct in this paper, we show the monoidal equivalence of the nonstrict $\te{Vec}_{\mbb{Z}_m}^\omega$ to the strict categories coming from the planar algebras explicitly. As a result, this shows how an associator of a fusion category gets concealed in the relations of the planar algebras.

While this paper was in preparation, we discovered that Reynolds' thesis \cite{Rey23} also gives presentations of diagrammatic categories with affine $A$ fusion rules. However, in \cite{Rey23}, Reynolds approaches these diagrammatics through the representation theory of a finite subgroup of $SU(2)$. In this paper, we define a functor yielding a monoidal equivalence of these two categories. In particular, we show purely through the defined diagrammatic categories that the category arising from affine $A_{m}$ planar algebras, where $m$ is finite, is monoidally equivalent to the representation category $\te{Rep}(C_m^\zeta)$, where $C_m^\zeta$ is a cyclic subgroup of $SU(2)$. Both this equivalence and the equivalence to $\te{Vec}_{\mbb{Z}_m}^\omega$ are known isomorphic categories (see \cite{GHJ89}), however, the diagrammatic proofs of their monoidal equivalences are novel.

\subsection{Acknowledgements}
The author would like to thank her advisor Stephen Bigelow as well as Dave Penneys for presenting this problem to her, their guidance, and many enlightening conversations.
\section{Background}

\subsection{Subfactors}

We begin our background with a discussion on subfactors. We refer the interested reader to \cite{JS97} for further details. Let $\mc{H}$ be a Hilbert space. A \textit{von Neumann algebra} is a unital *-subalgebra of $B(\mc{H})$, $M$, that equals its double commutant, $M''$, i.e., $M=M''$. When the center of $M$ is the smallest possible, that is, isomorphic to $\mbb{C}$, we say that $M$ is a \textit{factor}. When $\mc{H}$ is separable, von Neumann \cite{vN49} showed every von Neumann algebra, $M \subseteq B(\mc{H})$ has an essentially unique decomposition as a direct integral of factors. Understanding von Neumann algebras thus boils down to analyzing factors. 

There are three types of factors: type $I,II,$ and $III$.  For this paper, we consider a class of type $II$ factors called \textit{type $II_{1}$}. A factor of \textit{type $II_1$}, $M$, in $B(\mc{H})$, is an infinite-dimensional factor that admits a normalized trace function. From now on, we use \emph{factor} to mean a type $II_1$ factor.

To study factors, we consider maps between factors. A factor has no non-trivial closed two-sided ideal, so any map between factors is a unital inclusion or the zero map. Therefore, understanding maps between factors is equivalent to studying a unital inclusion of factors, $N\subseteq M$, which is called a \textit{subfactor}. 

There is a special type $II_1$ factor called the \textit{hyperfinite $II_1$ factor}, denoted $\mc{R}$. Connes \cite{Con76} proved that if $N\subseteq \mc{R}$ is a subfactor then $N$ is isomorphic to $\mc{R}$ or is finite-dimensional. In the infinite-dimensional case, since $N\cong \mc{R}$, the only information important from the subfactor $N\subseteq \mc{R}$ is not the von Neumann algebras $N$ and $\mc{R}$ themselves, but their inclusion. 

There is a measure of size of a subfactor. The \textit{index} of a subfactor $N\subseteq M$ is the von Neumann dimension of $L^2(M)$ as a left $N$-module \cite{Jon83}. Remarkably, Jones \cite{Jon83} found that the possible indices of subfactors lies in the set

\begin{align*}
    \bigg\{4\cos^2\left(\f{\pi}{n}\right):n\geq 3\bigg\}\cup [4,\infty]
\end{align*}

\noindent and further that every value in this set corresponds to an index of a subfactor. When $M$ is irreducible as an $N-M$ bimodule, we say the subfactor is \textit{irreducible}.

\subsection{Category Theory}

       For the rest of this paper, we work over the field of complex numbers, $\mbb{C}$, and assume all subfactors are finite-index and irreducible,  unless stated otherwise. We refer the reader to \cite{EGNO15} for a more detailed study of these topics. 
       
\subsubsection{2-Categories and The Standard Invariant}       
       A \textit{2-category}, $\mc{C}$, contains objects, morphisms (also called 1-morphisms), and 2-morphisms, which are morphisms between 1-morphisms. See \cite{EGNO15} section 2.12 for a formal definition. In particular, there are tensor products for 1-morphisms and 2-morphisms under certain compatible criteria. For 1-morphisms, the tensor product is composition. We can tensor two 2-morphisms $f \in \te{Hom}(X,Y)$, $g \in \te{Hom}(Z,W)$ as $f\otimes g\in \te{Hom}(X \otimes Z, Y \otimes W)$ if $X,Z$ and $Y,W$ are composable ($X \otimes Z$ and $Y\otimes W$ exists as 1-morphisms). Diagrammatically, the tensor product is denoted by placing the 1-morphisms or 2-morphisms side-by-side. We can also multiply compatible 1-morphisms or 2-morphisms. In the graphical calculus language, this is vertical stacking: $f\cdot g$ is stacking $f$ on top of $g$. Additionally, there are associators, and left and right unitors satisfying pentagon and triangle axioms.

    Given a subfactor $N\subseteq M$, we can analyze its $N-M$, $M-N$, $N-N$, and $M-M$ bimodules. As an example, $N$ can be considered as an $N-N$ bimodule, $\leftindex_N{N}_N$, and $M$ can be considered as an $M-M$, $N-M$, $M-N$, and $N-N$ bimodule: $\leftindex_M{M}_M$, $\leftindex_N{M}_M$, $\leftindex_M{M}_N$, and $\leftindex_N{M}_N$ respectively. We can tensor bimodules over the ring between them. For example, $\leftindex_N{M}_M\otimes_M \leftindex_M{M}_N\cong \leftindex_N{M}_N$.
    
    From this data we can form a powerful subfactor invariant. The \textit{standard invariant} of a subfactor $N\subseteq M$ is a unitary (strictly) pivotal 2-category, $\mc{C}$, with
     two objects: $N$ and $M$.  The 1-morphisms are the bimodules tensor-generated by $\leftindex_N{M}_M$ and $\leftindex_M{M}_N$. The 2-morphisms are the bimodule intertwiners. Further, the identities $\te{id}_N$ and $\te{id}_M$ are \textit{simple}. (i.e., $\te{End}(\te{id}_N)\cong \te{End}(\te{id}_M)\cong \mbb{C}$). As the category is strictly pivotal, every 1-morphism $X$ has a dual which we will denote $\ol{X}$ satisfying that $\ol{\ol{X}}=X$ and the properties of duals found in section 2.10 of \cite{EGNO15}. Since the standard invariant is unitary, this means there is an anti-linear map, $*$, that takes any 2-morphism, say $f\in \te{Hom}(X,Y)$ to a 2-morphism $f^* \in \te{Hom}(Y,X)$ such that the induced inner product $\langle f,g \rangle = \te{tr}(f^*g)$ is positive definite. 
    
    The standard invariant 2-category has diagrammatics that will be used to build the planar algebra. Define $X=\leftindex_N{M}_M$ and $Y=\ol{X}=\leftindex_M{M}_N$. We can denote the object $N$ by ``unshaded" and the object $M$ by ``shaded". The generators $X$ and $Y$ are denoted: \begin{align*}
   X= \Xinbox = \hspace{0.2cm} \X  \te{ and } Y = \Yinbox = \Y
\end{align*}
(we often drop the outside box and star and assume the location of the star is the center left of the diagram).

\noindent Tensoring 1-morphisms (and 2-morphisms) together corresponds to horizontal concatenation. For example,
\begin{align*}
    Y \otimes X = \Y \hspace{0.2cm} \X 
\end{align*}
We see then that the compatibility criteria for the tensor product will be fulfilled when the diagrams are placed side-by-side and the interior shading agrees.  The evaluation and coevaluation maps of $X$ and its dual are well-known maps in subfactor theory. The evaluation and coevaluation maps of $\ol{X}$ are denoted 

    \begin{align*}
        \evaluationmap: X \otimes \ol{X}= \leftindex_N{M}_M\otimes_M \leftindex_M{M}_N \to \leftindex_N{N}_N \te{ and }  \coevaluationmap: \leftindex_N{N}_N \to \leftindex_N{M}_N=\leftindex_N{M}\otimes_M M_N, 
    \end{align*}

    \noindent respectively. The evaluation map is called \textit{conditional expectation}  and the coevaluation map is the inclusion map of the subfactor $N\subseteq M$. The evaluation and coevaluation maps of $X$: 

    \begin{align*}
        \evaluationmapX: \leftindex_M{M}_N\otimes_N\leftindex_N{M}_M\to \leftindex_M{M}_M \te{ and }\coevaluationmapX: \leftindex_M{M}_M \to \leftindex_M{M}_N\otimes_N \leftindex_N{M}_M
    \end{align*}

    \noindent arises from the multiplication map and Jones' basic construction \cite{Jon83}, respectively. 
    
    A 2-morphism $f$ from $A$ to $B$ is a diagram with a distinguised star with bottom boundary data matching $A$ and top boundary data matching $B$. For example,

    \begin{align}\label{temperleyliebdiagrams}
      f=\fmorph \te{ and } g= \gmorph
    \end{align}

    \noindent are 2-morphisms from $X\otimes \ol{X} \otimes X$ to itself. We say a morphism is of \textit{type} $(k,+)$ (respectively $(k,-)$) if there are $k$ points on the top and bottom and the region to the right of the star is unshaded (respectively shaded). When the context is clear, we drop the star and the outside box and potentially the boundary of the box itself. Vertical composition (or multiplication) of 2-morphisms is given by vertical stacking. For example,

    \begin{align}\label{temperleyliebmultiplication}
        fg=\fgmorph=\gmorph, \quad \te{ and } \quad g^2=\gsquared=\gbubble
    \end{align}

    \noindent In the latter case, we have a bubble, which can be ``popped" for a factor of $\delta>0$, where $\delta^2$ is the index of the subfactor. That is,

    \begin{align*}
        \gbubble = \delta \te{ }\gmorph.
    \end{align*}

    \noindent Diagrams are considered equal if they are isotopic.

    \subsubsection{Fusion Categories}
    The $M-M$ bimodules and $N-N$ bimodules give a pair of pivotal categories called the \textit{even parts}. When these categories have finitely many simple objects we say that the subfactor is \textit{finite depth}. In this case, the even parts are fusion categories.  

    A \textit{monoidal category}, $(\mc{C},\otimes, \alpha, \mbb{1}, \lambda, \rho)$ is a category $\mc{C}$, a bifunctor, $\otimes: \mc{C}\times \mc{C} \to \mc{C}$ called the \textit{tensor product}, $\mbb{1}\in \mc{C}$ the \textit{unit}, and $\alpha_{X,Y,Z}$, $\lambda_X$, and $\rho_X$ the associators, and left and right unitors satisfying pentagon and triangle axioms. A monoidal category is \textit{rigid} if every object has a right and left dual. We say $\mc{C}$ is \textit{$\mbb{C}$-linear} if for all objects $X,Y \in \mc{C}$, $\te{Hom}(X,Y)$ has a $\mbb{C}$-module structure such that composition of morphisms is $\mbb{C}$-bilinear. 
    
    Let $X_1, X_2$ be objects in a category, $\mc{C}$. If there exists an object $Y$ in $\mc{C}$, with morphisms $\rho_1:Y \to X_1$, $\rho_2:Y\to X_2$, $\iota_1: X_1 \to Y$, $\iota_2:X_2\to Y$ such that $\rho_1\iota_1=\te{id}_{X_1}$, $\rho_2\iota_2=\te{id}_{X_2}$ and $\iota_1\rho_1+\iota_2\rho_2=\te{id}_Y$, we call $Y$ the \textit{direct sum} of $X_1$ and $X_2$ and denote $Y=X_1 \oplus X_2$. We call $\mc{C}$ \textit{additive} if it is $\mbb{C}$-linear and for every two objects $X_1, X_2$ in $\mc{C}$, there exists their direct sum $X_1\oplus X_2$ in $\mc{C}$. The  category is said to be \textit{semisimple} if every object is a direct sum of finitely many simple objects. 

    The \textit{kernel}, $\te{Ker}(f)$, of a morphism $f:X \to Y$ of an additive category $\mc{C}$ is a pair, $(K,k)$, where $K$ is an object in $\mc{C}$ and $k:K \to K$ is a morphism such that $fk=0$ and for any other pair $(K',k')$ satisfying $k':K'\to X$ $fk'=0$, there exists a unique morphism $\ell: K'\to K$ such that $k\ell=k'$. The \textit{cokernel}, $\te{Coker(f)}$, of a morphism $f:X \to Y$ is a pair, $(C,c)$, where $C$ is an object and $c:Y \to C$ is a morphism such that $cf=0$ and for any other pair $(C',c')$ satisfying this criteria, there exists a unique morphism $\ell:C\to C'$ where $\ell c =c'$. An \textit{abelian} category $\mc{C}$ is an additive category in which for every morphism $\varphi: X \to Y$ there exists a sequence 
    \begin{align*}
        K \os{k}{\to} X \os{i}{\to} I \os{j}{\to} Y \os{c}{\to} C
    \end{align*}
    such that
    $ji=\varphi$, $(K,k)=\te{Ker}(\varphi)$, $(C,c)=\te{Coker}(\varphi)$, $(I,i)=\te{Coker}(k)$, and $(I,j)=\te{Ker}(c)$. 

    A \textit{fusion category}, $\mc{C}$, over $\mbb{C}$ is a rigid, semisimple, $\mbb{C}$-linear, abelian category with only finitely many isomorphism classes of simple objects and the unit object $\mbb{1}$ is simple \cite{ENO05}. 

    Given a monoidal category $\mc{C}$, we can construct its \textit{additive envelope}, $\mc{A}(\mc{C})$, which has objects finite formal sums $X_1 \oplus ... \oplus X_n$ of objects in $\mc{C}$ and morphisms defined in the following way: for every $X_1 \oplus ... \oplus X_m$, $Y_1 \oplus ... \oplus Y_n$ in $\mc{A}(\mc{C})$, 
    \begin{align*}
        \te{Hom}_{\mc{A}(\mc{C})}(X_1 \oplus ... \oplus X_m, Y_1 \oplus ... \oplus Y_n)=\oplus_{i,j} \te{Hom}(X_i,Y_j)
    \end{align*}
    and composition of morphisms is given by matrix multiplication. 
    
    A \textit{monoidal subcategory} of a monoidal category $\mc{C}$ is a monoidal category $(\mc{D}, \otimes, \alpha, \mbb{1}, \lambda, \rho)$ where $\mc{D}\subset \mc{C}$ is a full unital subcategory closed under tensor product. For any object $X \in \mc{C}$, the \textit{monoidal subcategory generated by $X$} is the category with objects $X^{\otimes n}$ for all $n\in \mbb{N}$ and is denoted $\langle X \rangle$. 

    Let $(\mc{C},\otimes, \alpha, \mbb{1}, \lambda, \rho)$ and $(\mc{C}',\otimes', \alpha', \mbb{1}', \lambda', \rho')$ be two monoidal categories. A \textit{monoidal functor} $\mc{C} \to \mc{C}'$ is a pair $(F,J)$ where $F:\mc{C}\to \mc{C}'$ is a functor and $J_{X,Y}:F(X)\otimes'F(Y)\to F(X\otimes Y)$ is a natural isomorphism such that $F(\mbb{1})$ is isomorphic to $\mbb{1}'$ and that the following diagram commutes:

\begin{center}
    \begin{tikzcd}[row sep=large, column sep = large]
    (F(X)\otimes'F(Y))\otimes'F(Z) \arrow[r, "\alpha'_{F(X), F(Y),F(Z)}"] \arrow[d, "J_{X,Y}\otimes'\te{id}_{F(Z)}"]
     & F(X)\otimes'(F(Y)\otimes'F(Z)) \arrow[d, "\te{id}_{F(X)}\otimes' J_{Y,Z}"] \\
    F(X\otimes Y)\otimes'F(Z) \arrow[d, "J_{X\otimes Y,Z}"]
    &  F(X)\otimes'F(Y\otimes Z) \arrow[d, "J_{X,Y\otimes Z}"]\\
    F((X\otimes Y)\otimes Z)\arrow[r, "F(\alpha_{X,Y,Z})"]& F(X\otimes (Y \otimes Z)) 
\end{tikzcd}
\end{center}

\noindent A monoidal functor $F:\mc{C}\to \mc{C'}$ is \textit{full} if the induced map $\til{F}:\te{Hom}(X,Y)\to \te{Hom}(F(X),F(Y))$ is surjective and is \textit{faithful} if $\til{F}$ is injective. We call $F$ \textit{essentially surjective} if for all objects $Y$ in $\mc{C}'$ there is an object $X$ in $\mc{C}$ where $F(X) \cong Y$. The functor $F$ is an \textit{equivalence} if it is full, faithful, and essentially surjective. If further $F$ is a monoidal functor, we say $F$ is a \textit{monoidal equivalence}.

 
\subsection{Planar Algebras}

 The data of the standard invariant of a subfactor can be equivalently viewed as a \textit{planar algebra}, first defined by Jones \cite{Jon99}. When $N$ and $M$ are both the hyperfinite $II_1$ subfactor and $N\subseteq M$ is finite depth, then the standard invariant, and thus its planar algebra, is a complete invariant of the subfactor \cite{Pop94}. For a more detailed introduction, see \cite{Pet10}. 

 A \textit{planar algebra} is a collection of vector spaces with an action of a \textit{shaded planar operad}.  The \textit{shaded planar operad} has checkerboard shaded elements which look like:
 \begin{itemize}
     \item an outer square, $D$,
     \item $k\in \mbb{Z}_{\geq 0}$ empty inner squares inside of $D$,
     \item a possibly empty collection of nonintersecting strands between the squares such that every boundary has the same number of strands on its top and bottom,
     \item a distinguished star on the outside of each square.
 \end{itemize}
The below figure is an example of an element of the shaded planar operad. 

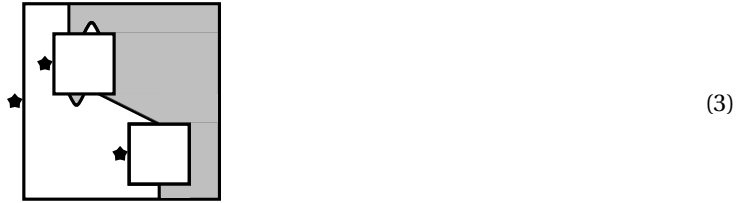
\begin{align}\label{shadedtangleexample}
    \shadedtangleexample
\end{align}

\noindent Two shaded tangles are equal if they are isotopic. Planar tangles can be composed by insertion under the appropriate conditions. That is, if the shading and number of strands on the boundary agree when the stars match up, we can compose two tangles. As there may be multiple squares in which the composition can be done, we label the squares along with a subscript in the composition indicating which square to insert into. Below we give an example:

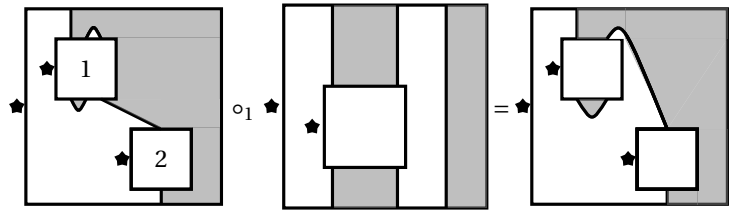
\begin{align*}
    \shadedtangleexamplelabelled \text{ }\circ_1 \text{ }\shadedtanglecompositiona=\shadedtanglecompositionb
\end{align*}

Notice that the star of each square in a tangle can be in an unshaded or shaded region. We say a square in a tangle has \textit{type} $(k,+)$ (respectively $(k,-)$) if its star is in an unshaded (respectively shaded) region and it has $k$ strands on its top and bottom. For example, the top square in figure (\ref{shadedtangleexample}) has type $(3,+)$ and the bottom square has type $(1,+)$. The outside square of the tangle has its type defined as done for 2-morphisms in the previous section. So the tangle in figure (\ref{shadedtangleexample}) has type $(1,+)$. Composition then is only defined when squares have the same type. 

Let $\{\mc{P}_{k,\pm}\}_{k\in \mbb{Z}_{\geq 0}}$ be a collection of vector spaces. Each tangle can be associated to a multilinear map between these vector spaces. For example, the tangle, $T$, in figure (\ref{shadedtangleexample}) corresponds to a map $Z_T:\mc{P}_{3,+}\otimes \mc{P}_{1,+}\to \mc{P}_{1,+}$.

Let $T$ be a tangle. Then denote $\mc{D}_T$ as the collection of its inner boxes. Define $\partial(D)$ to be the type of $D\in \mc{D}_T$. A \textit{planar algebra}, $\mc{P}$, is a collection of vector spaces $\{\mc{P}_{k,\pm}\}_{k \in \mbb{Z}_{\geq 0}}$ together with multilinear maps $Z_T:\otimes_{D \in \mc{D}_T} \mc{P}_{\partial(D)} \to \mc{P}_{\partial(T)}$ satisfying
\begin{enumerate}
    \item isotopic tangles produce the same maps,
    \item the tangle with just straight strands and no boxes is the identity,
    \item the composition of tangle diagrams corresponds to the composition of their multilinear maps.
\end{enumerate}

    Define a \textit{homomorphism} $\theta:\mc{P}\to \mc{Q}$ between planar algebras $\mc{P}$ and $\mc{Q}$ to be a family of linear maps, $\theta_{n,\pm}: \mc{P}_{n,\pm} \to \mc{Q}_{n,\pm}$ that preserve shading and such that $\theta(Z_T(f))=Z_T(\theta\circ f)$ for every tangle $T$. If each $\theta_{n,\pm}$ is bijective then $\theta$ is said to be an \textit{isomorphism}.
    
 We define the vector space $\mc{P}_{k,\pm}$ of a planar algebra $\mc{P}$ to be its \textit{$k$th-box spaces}. When considering the standard invariant, we choose $\mc{P}_{1,+}=\te{End}(X)$, $\mc{P}_{1,-}=\te{End}(\ol{X})$, $\mc{P}_{2,+}=\te{End}(X\otimes \ol{X})$, $\mc{P}_{2,-}=\te{End}(\ol{X}\otimes X)$ and so on. 


 \subsubsection{The Temperley-Lieb Planar Algebra}

The most popular example of a planar algebra is the \textit{Temperley-Lieb} planar algebra, $\mc{TL}$. This planar algebra was introduced in 1971 \cite{TL71} and then diagrammatically described by Kauffman \cite{Kau87} in 1987. 

The vector spaces, which we will denote denote $\mc{TL}_{k,\pm}$, consist of diagrams of non-intersecting strands with $k$ points on the bottom and top and the appropriate shading. Examples of elements in $\mc{TL}_{3,+}$ were seen in figure (\ref{temperleyliebdiagrams}). We also have two elements 
\begin{align*}
    \emptyset_+=\wsquare,\te{ and }\emptyset_-=\bsquare
\end{align*}
in $\mc{TL}_{0,+}$ and $\mc{TL}_{0,-}$ respectively. 

The generators $X$ and $\ol{X}$ of 1-morphisms are the same diagrammatic diagrams as the 2-morphisms $\te{id}_X$ and $\te{id}_{\ol{X}}$, in $\mc{TL}_{1,+}$ and $\mc{TL}_{1,-}$, respectively. In fact, when $P$ is a 1-morphism that satisfies $P^2=P$, then $\te{id}_P=P$. In this case, we often abuse notation and call the 2-morphism $\te{id}_P$, just $P$.

The planar operad action on the vector spaces in $\mc{TL}$ can be seen through diagrammatics. The figure below corresponds to a linear map $Z_T: \mc{TL}_{3,+} \otimes \mc{TL}_{3,+} \to \mc{TL}_{3,+}$. Then plugging in diagrams $f$ and $g$ from (\ref{temperleyliebdiagrams}) into the bottom and top squares respectively gives $g \in \mc{TL}_{3,+}$, just as shown in (\ref{temperleyliebmultiplication}).

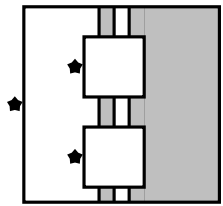
\begin{align*}
    \multiplicationtangle
\end{align*}

\noindent This tangle is called the \textit{multiplication tangle (of type $(3,+)$)}. Analogous tangles of all types give that vector spaces in planar algebras can actually be viewed as algebras. 

Each algebra $\mc{TL}_{k,\pm}$ is generated (as an algebra) by a basis denoted $\te{id}_{k, \pm}, e_{1,\pm},...,e_{{k-1},\pm}$, where $e_{i,\pm}$ has the appropriate shading and is straight strands in all but the $i$ and $i+1$ position, where there is a cup and cap. As an example, $g$ from figure (\ref{temperleyliebdiagrams}) is the generator $e_{1,+}$ for the algebra $\mc{TL}_{3,+}$. The element $\te{id}_{k,\pm}$ is $k$ straight strands with the appropriate shading.

Another commonly used tangle is the \textit{(right) trace tangle (of type $(k,\pm)$)}, $\te{tr}_{k,\pm}^{R}$, for which we have drawn the $(2,-)$ trace tangle below.

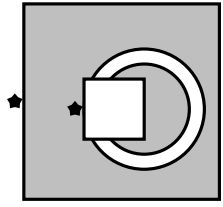
\begin{align*}
    \tracetangle
\end{align*}

\noindent Analogous tangles can be drawn for the left trace tangle, $\te{tr}_{k,\pm}^{L}$. We commonly say \textit{the trace tangle} and denote $\te{tr}$ which means the right trace tangle of the appropriate type. Notice that plugging in any diagram from $\mc{TL}$ to the trace tangle will result in a diagram in $\mc{TL}_{0,\pm}$ with no strands on the boundary. We call such a diagram a \textit{closed diagram}.

Two other commonly used tangles are given below.

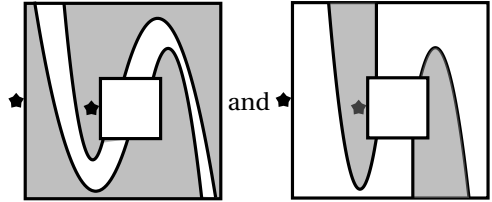
\begin{align*}
    \dualtangle \te{ and }\clicktangle
\end{align*}

The left-side is the \textit{the dual tangle (of type $(2, -)$)}. We denote the image of a diagram $D$ under the tangle by $\ol{D}$. This tangle rotates a diagram by $\pi$. The right-hand side is \textit{the click tangle (of type $(3,+)$).}, denoted $\mc{F}$, also referred to as the \textit{Fourier transform}. This tangle ``clicks" the star clockwise to the next region, i.e., the star now lies between the top two strands going to the top of the outer box. Analogous tangles of all types exist.


\subsubsection{Subfactor Planar Algebras and Their Principal Graphs}

A planar algebra, $\mc{P}$, arising from a subfactor is called a \textit{subfactor planar algebra}. These planar algebras satisfy a few additional properties:
\begin{enumerate}
    \item $\mc{P}$ is \textit{evaluable}: The space of closed diagrams, $\mc{P}_{0,\pm}$ is one-dimensional. (This means every closed diagram evaluates to a number.)
    \item $\mc{P}$ is \textit{spherical}: the left and right traces of the 1-box spaces are equal ($\te{tr}_{1,\pm}^R=\te{tr}_{1,\pm}^{L}$).
    \item $\mc{P}$ is \textit{positive definite}: there exists an antilinear \textit{adjoint map} on each box-space, $*:\mc{P}_{k,\pm}\to \mc{P}_{k,\pm}$ which is compatible with vertical reflection (reflection over a horizontal line) on planar tangles and the bilinear form on $\mc{P}_{k,\pm}$ induced by the adjoint, $\langle f,g\rangle =\te{tr}(f^*g)$ is positive definite. 
\end{enumerate}

Subfactor planar algebras are a diagrammatic axiomatization of the standard invariant of a subfactor. Jones \cite{Jon99} proved that given a finite index subfactor, its standard invariant forms a subfactor planar algebra. Using a different axiomatization of the standard invariant called \textit{$\lambda$-lattices}, Popa \cite{Pop95} proved that given a subfactor planar algebra, $\mc{P}$, there is a subfactor whose standard invariant forms $\mc{P}$. This was then proved using the planar algebra language by Guionnet, Jones, and Shlyakhtenko \cite{GHJ10}. 

Subfactor planar algebras encode the index of a subfactor. The value of a closed circle, $\delta>0$, satisfies that $\delta^2$ is the index of the subfactor. The Temperley-Lieb planar algebra is indeed a subfactor planar algebra as long as $\delta\geq 2$. Any subfactor planar algebra of index greater than or equal to 4 contains a copy of Temperley-Lieb.

By considering a single object in a 2-category, we can arrive at a notion of an \textit{unshaded planar algebra} and \textit{unshaded subfactor planar algebras}. These satisfy all the same properties as planar algebras and subfactor planar algebras, however they are not axiomatizing a standard invariant of a subfactor.

An element, $P$, in any box-space of a subfactor planar algebra satisfying that $P^2=P^*=P$ is a \textit{projection}. From a planar algebra $\mc{P}$, we can create a create a 2-category, $\mc{C}_{\mc{P}}$. This category is the additive envelope of the category where 
\begin{enumerate}
    \item Objects are projections in all box spaces
    \item A morphism in $\te{Hom}(P,Q)$ where $P$ and $Q$ are two projections is any diagram in the planar algebra with bottom boundary $Q$ and top boundary $P$. 
    \item The tensor product is horizontal concatenation
    \item The dual is rotation by $\pi$ (or equivalently $-\pi$)
    \item The unit objects are $\emptyset_+=\wsquare$ and $\emptyset_-=\bsquare$.
\end{enumerate}

A projection is \textit{minimal} if $\te{Hom}(P,P)$ is one-dimensional.  As a consequence, using that $P^2=P$,  every element of $\te{Hom}(P,P)$ is a multiple of $\te{id}_P=P$. We say two projections, $P$ and $Q$ are \textit{isomorphic} if and only if there exists an $f \in \te{Hom}(P,Q)$ (giving $f^*\in \te{Hom}(Q,P)$) such that $f^*f=P$ and $ff^*=Q$. For any planar algebra, $\mc{P}$, the simples of $\mc{C}_\mc{P}$ are precisely the minimal projections of $\mc{P}$. Further, $\mc{C}_\mc{P}$ is semisimple. This means that the tensor product of any two objects can be decomposed as a direct sum of minimal projections.

In Temperley-Lieb, in each box space, $\mc{TL}_{k,\pm}$, there is a unique minimal projection, denoted $f^{(k,\pm)}$, called the \textit{Jones-Wenzl projection}. These projections satisfy that for all the multiplicative generators other than the identity, $e_{i,\pm}$, $f^{(k,\pm)}e_{i,\pm}=e_{i,\pm}f^{(k,\pm)}=0$. These projections have a recursive relation, called \textit{Wenzl's relation} \cite{Wen87}: 
\begin{align*}
    &f^{(k,\pm)}\otimes X^{\pm} \cong f^{(k+1, \pm)} \oplus f^{(k-1,\pm)}, \te{ where,}\\
    &f^{(0,+)}=\emptyset_+=\wsquare, f^{(0,-)}=\emptyset_-=\bsquare, 
    f^{(1,+)}=X=\X, \te{ and } f^{(1,-)}=\ol{X}=\Y
\end{align*}

The \textit{$n$th quantum number} is defined as 
\begin{align*}
    [n]_q=\frac{q^n-q^{-n}}{q-q^{-1}}=q^{n-1}+q^{n-3}+...+q^{-(n-3)}+q^{-(n-1)}
\end{align*}
Then we can define $q$ such that $\delta=[2]_q=q+q^{-1}$. The Jones-Wenzl projections satisfy that $\te{tr}(f^{(k,\pm)})=[k+1]_q$. In particular, $\te{tr}(f^{(1,\pm)})=\te{tr}(X)=\te{tr}(\ol{X})=[2]_q=\delta$, and $\te{tr}(f^{(0,\pm)})=\te{tr}(\emptyset_+)=\te{tr}(\emptyset_-)=[1]_q=1$.

Through the category $\mc{C}_\mc{P}$ of a subfactor planar algebra, we can view another invariant of subfactors called the \textit{principal and dual principal graphs}. The principal graphs can be though of as a part of a ``multiplication table"-like graph that encodes the decomposition of simple objects when tensored by $X$ or $\ol{X}$. 

Vertices of the principal graphs are isomorphism classes of minimal projections. To find the vertices connected to a vertex labelled by a minimal projection, $P$, we do the following. By the semisimplicity of $\mc{C}_\mc{P}$, $P\otimes X^{\pm}$ is isomorphic to a direct sum of minimal projections:

\begin{align}\label{formulaforprincipalgraph}
    P \otimes X^{\pm} \cong \oplus_{Q \in N(P)} Q
\end{align}

Where $N(P)$ is some finite set of minimal projections. Then $N(P)$ is the set of neighbors of $P$ in the principal graph. The number of edges connecting $P$ and some $Q\in N(P)$ is the number of times $Q$ appears in the summand of (\ref{formulaforprincipalgraph}). This means that if $P$ and $Q$ are minimal projections, there are $\te{dim}(\te{Hom}(P\otimes X^{\pm}, Q))=\te{dim}(\te{Hom}(P,Q\otimes X^{\pm})$ edges between $P$ and $Q$. What this boils down to is that $P\otimes X^{\pm}$ is isomorphic to the direct sum of its neighbors in the prinicipal graph.

For Temperley-Lieb, Wenzl's relation gives that the principal graphs will be the $A_\infty$ Dynkin diagram:

\begin{align*}
    \temperleyliebprincipalgraph \quad \te{ and } \quad \temperleyliebdualprincipalgraph
\end{align*}

This paper deals exclusively with index 4 subfactor planar algebras (which are necessarily of the hyperfinite $II_1$ factor). In this case, Popa \cite{Pop94} gives that the principal graphs are the simply-laced affine Dynkin diagrams. Specifically, there are $n$ of type $\til{A}_{2n-1}$ for $n\geq 1$, 1 of type $\til{A}_\infty$, 1 of type $A_\infty$ (Temperley-Lieb), $n-2$ of type $\til{D}_{n}$ for $n\geq 4$, and 1 each of  $\til{E}_6$, $\til{E}_7$, and $\til{E}_8$. For this index, the principal and dual principal graphs are always identical.

To conclude, the decomposition of the planar algebra into minimal projections can be viewed as a \textit{Bratteli diagram}. This is an infinite graph, where row $0$ starts at the top. 

The vertices in row $k$ will correspond to the minimal projections in the decomposition of $X\otimes \ol{X} \otimes \os{(k-1)}{...}\otimes X^{\pm}$ ($k$ projections tensored together alternating between $X$ and its dual). The labels of the vertices will correspond to the number of times they appear in the decomposition.

 The number of edges from the vertex corresponding to $f^{(\lambda, \pm)}$ in  row $k$ to the vertex corresponding to $f^{(\mu,\pm)}$ in row $k+1$ is $N_{\lambda,\mu}$ where $f^{(\lambda, \pm)} \otimes X^{\pm} \cong \oplus_{\mu} N_{\lambda, \mu}f^{(\mu)}$. 

We can make the Bratteli diagram for the Temperley-Lieb planar algebra (which we can also call the \textit{$A_\infty$ planar algebra}). Since the principal graph and dual principal graph are the same, we will drop the $\pm$. Using Wenzl's relation we get the vertices and their labels:
\begin{align*}
\text{Row 0: } & \emptyset \cong f^{(0)}\\
\text{Row 1: } & X \cong f^{(1)}\\
   \text{Row 2: } & X \otimes X \cong f^{(1)} \otimes X \cong f^{(0)} \oplus f^{(2)} \\
    \text{Row 3: } & X \otimes X \otimes X \cong f^{(1)}\otimes (f^{(1)}\otimes f^{(1)}) \cong (\emptyset \oplus f^{(2)})\otimes X \cong f^{(1)}\oplus f^{(1)} \oplus f^{(3)} \\
       \text{Row 4: }  & X \otimes X \otimes X \otimes X \cong f^{(0)} \oplus f^{(0)} \oplus f^{(2)} \oplus f^{(2)} \oplus f^{(2)}\oplus f^{(4)}
\end{align*}
\noindent as well as that there will be exactly 1 edge between a vertex in row $k$ and $k+1$ when the vertices differ by 1. This gives that the first five rows of its Bratteli diagram will look like

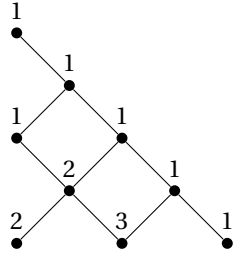
\begin{align*}
    \bratteli
\end{align*}

\noindent From the Bratteli diagram, we get that the dimension of the $k$th box space, $\mc{P}_{k,\pm}$, is the sum of the squares of the $k$th row.

\section{Necessary Relations for Affine \texorpdfstring{$A$}{A} Finite Subfactor Planar Algebras}\label{necessaryshadedsection}

    Our first goal is to find all diagrammatic presentations of subfactor planar algebras whose principal graph and dual principal graph are $\til{A}_{2n-1}$ for all $n\in \mathbb{N}$. Fix $n\in \mbb{N}$ and define $\mc{P}$ to be a subfactor planar algebra with principal and dual principal graphs the $\til{A}_{2n-1}$ Dynkin diagram. Let $\mc{C}_{\mc{P}}$ be the 2-category created from $\mc{P}$. Define $X=\X$, $Y=\Y$, $\emptyset_+=\wsquare$, and $\emptyset_-=\bsquare$. We indicate $\strand \os{(n)}{...} \strand$ to denote $n$ parallel strands, with color, shading and orientation assumed from context. 

    Recall from the background discussion that the vertices of principal and dual principal graphs of a subfactor planar algebra $\mc{P}$ correspond to equivalence classes of minimal projections in the category $\mc{C}_\mc{P}$. For $n\geq 2$, let $\mc{P}$ have principal graphs with the following labelling:

     \begin{equation}\label{eq: Afinshadearbgraphs}
        \Gamma_{+}: \arbAfingraphpos, \te{ and } \Gamma_{-}: \arbAfingraphneg
    \end{equation}

    and when $n=1$, let $\mc{P}$ have principals graphs:

    \begin{equation}\label{eq: Afinshadearbgraphsone}
        \Gamma_{+}: \arbAfingraphposone \te{ and } \Gamma_{-}: \arbAfingraphnegone
    \end{equation}

    \noindent where each $P_i, P_i', Q_j, Q_j'$ for $1\leq i \leq n$ and $1\leq j \leq n-1$ are representative of their respective equivalence class and the 2 above an edge indicates that there are two edges. 

    Let $P$ be any representative for a minimal projection on a principal graph for some subfactor planar algebra with index $\delta^2$. Let $N(P)$ denote all the neighbors of $P$ in the principal graph. Then taking the trace of both sides of the formula (\ref{formulaforprincipalgraph}) gives:
        \begin{equation}\label{eq: thetraceformula}
             \delta \te{tr}(P)=\sum_{Q \in N(P)}\te{tr}(Q)
        \end{equation}
        where $\delta >0$. We call (\ref{eq: thetraceformula}) \textit{the trace formula}.
        
        By Perron-Frobenius, if $\Gamma$ is a finite connected graph with vertex set $\Gamma_V$, then there exists a function $t:\Gamma_V\to (0,\infty)$ and a value $d\in (0,\infty)$ such that for every vertex $v\in \Gamma_V$, $dt(v)=\sum_{w\in N(v)} t(w)$. Furthermore, $d$ and $t$ are unique up to multiplying $t$ by a positive real number. For the trace formula, since $\te{tr}(\emptyset)=1$, the value of $d$ is unique, so $d=\delta$, and $t$ must equal $\te{tr}$ in order for $dt(v)=\sum_{w\in N(v)} t(w)$ to hold. 
        
\begin{lem}\label{lem: a-fin-graph-index}
    $\mc{P}$ has index 4 and all minimal projections have trace 1. 
\end{lem}

    \begin{proof} All minimal projections have exactly two neighbors in their principal graph $\Gamma_+$ or $\Gamma_-$. Therefore, if the trace of all the minimal projections are 1 and $\delta=2$, the trace formula holds. By Perron-Frobenius, these values are unique, so $\mc{P}$ has index 4 and all minimal projections have trace 1.  \end{proof}

\begin{lem}\label{lem: Afinshadeuniquerep}
        If $\mc{P}$ is a subfactor planar algebra with labelling of its principal and dual principal graphs given in (\ref{eq: Afinshadearbgraphs}) with $n\geq 2$, then there exists a unique representative of  $[P_1], [Q_1]$ in $\mc{P}_{1,+}$ and $[P_1'], [Q_1']$ in $\mc{P}_{1,-}$.
\end{lem}

    \begin{proof} The proof will be identical for $\mc{P}_{1,+}$ and $\mc{P}_{1,-}$, so we will just prove the result for $\mc{P}_{1,+}$. First we prove uniqueness. Suppose $P_1^{\dagger}\cong P_1^{\dagger \dagger}\cong P_1$ and $P_1^{\dagger}, P_1^{\dagger \dagger}\in \mc{P}_{1,+}$. By the principal graph, $\smallwsquare \otimes X \cong X \cong P_1 \oplus Q_1$. By definition of direct sum, there exists morphisms in $\mc{C}_{\mc{P}}$, $\rho_1:P_1 \oplus Q_1 \ra P_1$, $\rho_2: P_1 \oplus Q_1 \ra Q_1$, $\iota_1: P_1 \ra P_1\oplus Q_1$, and $\iota_2: Q_1 \ra P_1 \oplus Q_1$, satisfying $\rho_1\iota_1=\te{id}_{P_1}$, $\rho_2\iota_2=\text{id}_{Q_1}$, $\rho_2\iota_1=\rho_1\iota_2=0$ and $\iota_1\rho_1+\iota_2\rho_2=\te{id}_{P_1\oplus Q_1}$. By definition of isomorphism there exists morphisms in $\mc{C}_{\mc{P}}$, $f:P_1 \oplus Q_1 \ra X$, $g:X \ra P_1 \oplus Q_1$ such that $f g = \te{id}_{X}$ and $g f=\te{id}_{P_1 \oplus Q_1}$. Thus 

    \begin{equation} \label{eq: Afinshadep1q1unique}
        f \te{id}_{P_1\oplus Q_1}  g =f g = \te{id}_{X}=f\iota_1 \rho_1  g + f \iota_2 \rho_2  g
    \end{equation}

    \noindent Call the map $R=f\iota_1 \rho_1 g \in \te{Hom}\left(X, X\right)$ and $S=f \iota_2\rho_2 g \in \te{Hom}\left(X, X\right)$, so we can rewrite (\ref{eq: Afinshadep1q1unique}) as $\te{id}_X=R+S$. Since $P_1^{\dagger}\in \mc{P}_{1,+}$ and is a projection, $\te{id}_X P_1^{\dagger} \te{id}_X = P_1^{\dagger} \in \te{Hom}(X,X)$. On the other hand, using that $\te{id}_{X}=R+S$ we get 

        \begin{equation}\label{eq: Afinshadep1dagger}
            P_1^{\dagger}=(R+S)P_1^{\dagger}(R+S)=R P_1^{\dagger} R + R P_1^{\dagger}  S + S  P_1^{\dagger} R + S P_1^{\dagger} S
        \end{equation}

    \noindent $S$ is a map that factors through $Q_1$ and not $P_1$ and as $Q_1$ and $P_1^{\dagger}$ are nonisomorphic minimal projections $R P_1^{\dagger} S = S  P_1^{\dagger} R=S P_1^{\dagger} S=0$. What's left to compute is $R P_1^{\dagger} R=f \iota_1 \rho_1 g P_1^{\dagger} f \iota_1 \rho_1 g$. Notice that $\rho_1 g P_1^{\dagger} f \iota_1\in \te{Hom}(P_1, P_1)$ and as this hom-space is one-dimensional, this gives that $\rho_1 g P_1^{\dagger} f \iota_1=\lambda P_1$ for some $\lambda \in \mathbb{C}$. Equation (\ref{eq: Afinshadep1dagger}) becomes $RP_1^{\dagger}R=\lambda (f \iota_1 \rho_1 g)=\lambda R=P_1^{\dagger}$. We can deduce that $\lambda$ is nonzero since $P_1^{\dagger}$ has nonzero trace. We can do the same thing for $P_1^{\dagger \dagger}$ and get that there exists a $\mu \in \mathbb{C}\backslash\{0\}$ such that $\mu R=P_1^{\dagger \dagger}$. Thus $\f{1}{\lambda}P_1^{\dagger}=\f{1}{\mu}P_1^{\dagger \dagger}$. $P_1^{\dagger}$ and $P_1^{\dagger\dagger}$ have the same trace, so $\f{1}{\lambda}\te{tr}(P_1^{\dagger})=\f{1}{\mu}\te{tr}(P_1^{\dagger\dagger})$, giving that $\lambda=\mu$ and $P_1^{\dagger}=P_1^{\dagger \dagger}$, proving uniqueness. A similar proof will show uniqueness for the representative of $[Q_1]$ in $\mc{P}_{1,+}$. 

    Next, we prove existence. We claim that $R$ is the correct choice for the unique representative $[P_1]$ in $\mc{P}_{1,+}$ and $S$ is the correct choice for the unique representative of $[Q_1]$ in $\mc{P}_{1,+}$. It easy to see that $R^2=R$. As $R \in \te{Hom}\left(X, X\right)$, $R \in \mc{P}_{1,+}$. What's left to show is that $R\cong P_1$. Defining $h_1=\rho_1 g\in \te{Hom}\left(X,P_1\right)$ and $k_1=f \iota_1 \in \te{Hom}\left(P_1, X\right)$ gives $R\cong P_1$. So indeed $R$ is the unique representative of $[P_1]$ in $\mc{P}_{1,+}$. A similar proof shows $S$ is the unique representative of $[Q_1]$ in $\mc{P}_{1,+}$. \end{proof}

\begin{lem}\label{lem: nonep1q1reps}
    When $n=1$, there exists $R_1, Q_1 \in \mc{P}_{1,+}$ isomorphic to $P_1$ and $R_1', Q_1'\in \mc{P}_{1,-}$ isomorphic to $P_1'$ with $R_1Q_1=Q_1R_1=0$, $\ol{R}_1\ol{Q}_1=\ol{Q}_1\ol{R}_1=0$, $X=R_1+Q_1$, $Y=R_1'+Q_1'$ and $\ol{R}_1=R_1'$ and $\ol{Q}_1=Q_1'$. 
\end{lem}

    \begin{proof} When $n=1$, the principal graph gives that $X \cong P_1 \oplus P_1$. By definition of direct sum, this means that there exists some $R_1, Q_1\in \mc{P}_{1,+}$ with $R_1Q_1=Q_1R_1=0$ and $X=R_1+Q_1$. Taking the duals of $R_1$ and $Q_1$ we get $\ol{R}_1\ol{Q}_1=\ol{Q}_1\ol{R}_1= 0$ and $Y=\ol{R}_1+\ol{Q}_1$ and $\ol{R}_1, \ol{Q}_1\in \mc{P}_{1,-}$. Since $R_1$ and $Q_1$ are minimal projections isomorphic to $P_1$ this gives that $\ol{R}_1$ and $\ol{Q}_1$ are minimal projections. Since $Y=\ol{R}_1+\ol{Q}_1$ we obtain that $\ol{R}_1\cong \ol{Q}_1 \cong P_1'$. \end{proof}
    
    From now on, for $n\geq 2$, choose $P_1$ and $Q_1$ to be the unique representatives of $[P_1]$ and $[Q_1]$ in $\mc{P}_{1,+}$ and let $P_1'$, $Q_1'$ be the unique representatives of $[P_1']$ and $[Q_1']$ in $\mc{P}_{1,-}$ respectively. For $n=1$ choose $P_1=R_1$, $P_1'=\ol{R}_1$, $Q_1$, and $Q_1'=\ol{Q}_1$ from Lemma \ref{lem: nonep1q1reps}.  From Lemmas \ref{lem: nonep1q1reps} and \ref{lem: Afinshadeuniquerep} , $X= P_1 + Q_1$ and $Y=P_1'+Q_1'$. We already know when $n=1$ that we can asssume $P_1'=\ol{P}_1$ and $Q_1'=\ol{Q}_1$. We prove this result for $n\geq 2$ by first finding bases for $\mc{P}_{1,+}$ and $\mc{P}_{1,-}$.

\begin{lem} \label{lem: Afinshade-basis}
        For $n\geq 2$, $\{P_1, Q_1\}$ forms a basis of $\mc{P}_{1,+}$ and $\{P_1', Q_1'\}$ forms a basis of $\mc{P}_{1,-}$.
\end{lem}

    \begin{proof} Without loss of generality, we will show $\{P_1, Q_1\}$ forms a basis of $\mc{P}_{1,+}$. Recall from the background discussion that the dimension of $\mc{P}_{1,+}$ is the sums of the squares in row 1 of the Brattelli diagram. Since $X\cong P_1 \oplus Q_1$, the dimension of $\mc{P}_{1,+}$ is 2. 

    What's left to check is that $P_1$ and $Q_1$ are linearly independent.  Since $Q_1$ and $P_1$ are nonisomorphic minimal projections, neither $Q_1$ or $P_1$ are zero and $P_1Q_1=Q_1P_1=0$, $P_1^2=P_1$, and $Q_1^2=Q_1$. Therefore if there exists $\lambda, \mu \in \mbb{C}$ where $\lambda P_1+\mu Q_1=0$, multiplying by $P_1$ gives $\lambda=0$ and multiplying by $Q_1$ gives $\mu=0$. \end{proof}

\begin{lem}
        Either $\ol{P}_1=P_1'$ and $\ol{Q}_1=Q_1'$ or $\ol{P}_1=Q_1'$ and $\ol{Q}_1=P_1'$.
\end{lem}

    \begin{proof} The case $n=1$ was done in Lemma \ref{lem: nonep1q1reps}. In Lemma \ref{lem: Afinshadeuniquerep} we found that $P_1$ and $Q_1$ satisfy $X=P_1+Q_1$ and that $Y=P_1'+Q_1'$. By definition of projection, $P_1^2=P_1=P_1^*$. For any diagram $A\in \mc{P}$ it is clear that $\ol{A}^2=\ol{A^2}$ and $\ol{A}^*=\ol{A^*}$, thus
    \begin{align*}
        \ol{P_1^2}=\ol{P}_1=\ol{P}_1^2=\ol{P_1^*}=\ol{P}_1^*
    \end{align*}
    which gives that $\ol{P}_1\in \mc{P}_{1,-}$ is a projection. Similarly, $\ol{Q}_1 \in \mc{P}_{1,-}$ is a projection. 
    
    Let $R$ be any minimal projection. Clearly, $f\in \te{Hom}(P_1, R)$ if and only if $\ol{f}\in \te{Hom}(\ol{R}, \ol{P}_1)$, which gives that the dimension of the hom-space from $\ol{P}_1$ to itself is 1 and from $\ol{P}_1$ to any nonisomorphic minimal projection is 0. Thus, $\ol{P}_1$ is a minimal projection in $\mc{P}_{1,-}$. By Lemma \ref{lem: Afinshadeuniquerep} this gives that $\ol{P}_1=P_1'$ or $\ol{P}_1=Q_1'$. Likewise, $\ol{Q}_1$ is a minimal projection. The dimension of $\te{Hom}(\ol{P}_1, \ol{Q}_1)$ is 0 since $\te{Hom}(P_1, Q_1)$ has dimension 0, so $\ol{P}_1$ and $\ol{Q}_1$ are nonisomorphic minimal projections. Therefore, either $\ol{P}_1=P_1'$ and $\ol{Q}_1=Q_1'$ or $\ol{P}_1=Q_1'$ and $\ol{Q}_1=P_1'$. \end{proof}

    Due to the principal graph's symmetry we can choose that $\ol{P}_1=P_1'$ and $\ol{Q}_1=Q_1'$. Recall that in a planar algebra, the dual of a diagram is rotation by $\pi$. For our diagrammatic descriptions of subfactor planar algebras with principal graph $\til{A}_{2n-1}$, we will want to have diagrams $P_1$, $P_1'$, $Q_1$, $Q_1'$ such that $P_1, Q_1 \in \mc{P}_{1,+}$, $P_1', Q_1'\in \mc{P}_{1,-}$, and rotating $P_1$ and $Q_1$ $\pi$ gives $P_1'$ and $Q_1'$ respectively. We will introduce diagrammatic notation that aids the visualization of this phenomena by denoting

    \begin{equation} \label{eq: AfinshadelabelsP1Q1}
        P_1=\redr, Q_1=\bluer, P_1'=\ol{P}_1=\redl, \te{ and } Q_1'=\ol{Q}_1=\bluel.
    \end{equation}

    \begin{lem} \label{lem: Afinshadesaddlerels}
        For $\mc{P}$ we have the following:
        \begin{enumerate}[label=(\roman*)]
            \item $\redsaddlesin=P_1 \otimes \ol{P}_1$,   $\bluesaddlesin=Q_1\otimes \ol{Q}_1$, $\redsaddlesout=\ol{P}_1\otimes P_1$, and $\bluesaddlesout=\ol{Q}_1\otimes Q_1$,
            \item $\redsaddlesin \cong \emptyset_{+}$, $\redsaddlesout \cong \emptyset_{-}$, $\bluesaddlesin \cong \emptyset_{+}$, and $\bluesaddlesout \cong \emptyset_{-}$
        \end{enumerate}
        Thus we can conclude $P_1 \otimes \ol{P}_1 \cong \emptyset_{+}\cong Q_1 \otimes \ol{Q}_1 \cong \emptyset_{+}$ and $\ol{P}_1\otimes P_1 \cong \emptyset_{-} \cong \ol{Q}_1\otimes Q_1$.
    \end{lem}

    \begin{proof} We will prove the first equality or isomorphism of each part. The rest will follow analogously. Define $D$ to be the lefthand side of the first equality of part (i). For part (i), using that the trace of $P_1$ and $\ol{P}_1$ are 1, we can see that
   \begin{align*}
          \Biggl<D-(P_1 \otimes \ol{P}_1), D-(P_1 \otimes \ol{P}_1)\Biggr> = \te{tr}\left(-D+(P_1 \otimes \ol{P}_1)\right) = 0
      \end{align*}
    Since $\mc{P}$ is a subfactor planar algebra, the inner product $<,>$ is positive-definite, which gives that $D-(P_1 \otimes \ol{P}_1)=0$ and thus $D=P_1\otimes \ol{P}_1$. 

    For part (ii), in the category $\mc{C}_\mc{P}$, define $f=\redcupin$, which is a morphism from $\emptyset_{-}$ to $P_1 \otimes \ol{P}_1$. Clearly, $f^*$ is then $\redcapin$. Multiplying these we get $ff^*=D$ and $f^*f=\emptyset_{+}$. Therefore $D \cong \emptyset_{+}$. \end{proof}

    Using the previous lemma and its proof we can see that for any diagram $T$ in the planar algebra, $T^*$ is vertical flip and star the inside of any box, then extend this definition anti-linearly and on diagrams. 

\begin{lem}
        In $\mc{P}$, $\te{tr}(X)=2\emptyset_{+}$ and $\te{tr}(Y)=2\emptyset_{-}$. 
\end{lem}

    \begin{proof} This follows immediately from taking the traces of the identities: $X=P_1+Q_1$ and $Y=\ol{P}_1+\ol{Q}_1$. \end{proof}

\begin{lem} \label{lem: Afinshadedtensordecomp}
        Let $n\geq 2$. In $\mc{P}$ with principal graphs given by (\ref{eq: Afinshadearbgraphs}) we get:
        \begin{enumerate}[label=(\roman*)]
            \item $P_k \cong P_1 \otimes \ol{Q}_1 \otimes P_1 \otimes \os{(k-2)}{...} \otimes R_1$,
            \item $P_k'\cong \ol{P}_1 \otimes Q_1 \otimes \ol{P}_1 \os{(k-2)}{...} \otimes \ol{R}_1$,
            \item $Q_\ell \cong Q_1 \otimes \ol{P}_1 \otimes Q_1 \otimes \os{(\ell-2)}{...} \otimes \ol{R}_2$, and
            \item $Q_\ell'\cong \ol{Q}_1 \otimes P_1 \otimes \ol{Q}_1 \os{(\ell-2)}{...} \otimes R_2$
        \end{enumerate} 
        for $1\leq k \leq n$ and $1\leq \ell \leq n-1$ where $R_1$ and $R_2$ are either $P_1$ or $\ol{Q}_1$, depending on the parity of $n$, and $R_1 \neq R_2$.
\end{lem}

    \begin{proof} Proving each part is nearly identical, so we will just prove (i). We will prove (i) by induction. When $k=1$, the claim is clear. By the principal graph, $P_1 \otimes Y \cong \emptyset_{+} \oplus P_2$. In Lemma \ref{lem: Afinshadeuniquerep} we showed $X=P_1+Q_1$, so we get $Y=\ol{P}_1+\ol{Q}_1$. Then
    \begin{align*}
        P_1 \otimes Y = P_1 \otimes (\ol{P}_1+\ol{Q}_1)= P_1 \otimes \ol{P}_1 + P_1 \otimes \ol{Q}_1.
    \end{align*}
    By Lemma \ref{lem: Afinshadesaddlerels} we know $P_1 \otimes \ol{P}_1 \cong \emptyset_{+}$, which gives $P_1 \otimes \ol{Q}_1 \cong P_2$.    

    Next, assume the claim is true up to some $2\leq k\leq n-1$. When $k$ is even, $R_1=\ol{Q}_1$, that is, $P_k \cong P_1 \otimes \ol{Q}_1 \otimes \os{(k-1)}{...} \otimes \ol{Q}_1$. Then, by the principal graph, $P_k \otimes X \cong P_{k-1} \oplus P_{k+1}$. Since $X=P_1 + Q_1$, we get that 
    \begin{align*}
        P_k \otimes X = P_k \otimes (P_1 + Q_1) = P_k \otimes P_1 + P_k \otimes Q_1 \cong  P_1 \otimes \ol{Q}_1 \otimes \os{(k-1)}{...} \otimes \ol{Q}_1 \otimes P_1 +  P_1 \otimes \ol{Q}_1 \otimes \os{(k-1)}{...} \otimes \ol{Q}_1 \otimes Q_1
    \end{align*}
    and again, from Lemma \ref{lem: Afinshadesaddlerels}, we know $\ol{Q}_1 \otimes Q_1 \cong \emptyset_{-}$, which gives $P_1 \otimes \ol{Q}_1 \otimes \os{(k-2)}{...}\otimes P_1 \cong P_{k-1}$. Therefore, $P_1 \otimes \ol{Q}_1 \otimes \os{(k-1)}{...} \otimes \ol{Q}_1 \otimes P_1 \cong P_{k+1}$, as desired.

    When $k$ is odd, $R_1=P_1$, so $P_k \cong P_1 \otimes \ol{Q}_1 \otimes \os{(k-1)}{...} \otimes P_1$. By the principal graph, $P_k \otimes Y \cong P_{k-1}\oplus P_{k+1}$. Since $Y=\ol{P}_1+\ol{Q}_1$, we get
    \begin{align*}
        P_k \otimes Y \cong P_k \otimes (\ol{P}_1+\ol{Q}_1) = P_k \otimes \ol{P}_1 + P_k \otimes \ol{Q}_1 \cong P_1 \otimes \ol{Q}_1 \otimes \os{(k-1)}{...} \otimes P_1 \otimes \ol{P}_1 + P_1 \otimes \ol{Q}_1 \otimes \os{(k-1)}{...} \otimes P_1 \otimes \ol{Q}_1
    \end{align*}
    As $P_1 \otimes \ol{P}_1 \cong \emptyset_{+}$, we get $P_1 \otimes \ol{Q}_1 \otimes \os{(k-1)}{...} \otimes P_1 \otimes \ol{P}_1 \cong P_1 \otimes \ol{Q}_1 \otimes \os{(k-2)}{...} \otimes \ol{Q}_1\cong P_{k-1}$, we get $P_{k+1}\cong P_1 \otimes \ol{Q}_1 \otimes \os{(k-1)}{...} \otimes P_1 \otimes \ol{Q}_1$. Hence, we have the desired claim. \end{proof}

\begin{lem}\label{lem: AfinshadePnisotoQn}
         Let $n\geq 2$. In $\mc{P}$ with principal graphs given by (\ref{eq: Afinshadearbgraphs}) we get:
         \begin{enumerate}[label=(\roman*)]
             \item $P_n \cong Q_{n-1} \otimes \ol{R}_2$
             \item $P_n'\cong Q_{n-1} \otimes R_2$
         \end{enumerate}
         where $R_2$ is $P_1$ or $\ol{Q}_1$, depending on the parity of $n$. 
\end{lem}

    \begin{proof} Both parts will be similar, so we just prove part (i). First assume $n$ is even. So $\ol{R}_2=\ol{P}_1$. By the principal graph, $Q_{n-1}\otimes Y\cong Q_{n-2} \oplus P_n$. On the other hand, from Lemma \ref{lem: Afinshadedtensordecomp} and the fact that $Y=\ol{P}_1+\ol{Q}_1$ we obtain,
    \begin{align*}
        Q_{n-1}\otimes Y \cong Q_1 \otimes \ol{P}_1 \otimes Q_1 \otimes \os{(n-3)}{...} \otimes Q_1 \otimes Y \cong Q_1 \otimes \ol{P}_1 \otimes Q_1 \otimes \os{(n-3)}{...} \otimes Q_1 \otimes \ol{P}_1 + Q_1 \otimes \ol{P}_1 \otimes Q_1 \otimes \os{(n-3)}{...} \otimes Q_1 \otimes \ol{Q}_1
    \end{align*}
    Since $Q_1 \otimes \ol{Q}_1\cong \emptyset_{+}$, we have that 
    \begin{align*}
        Q_1 \otimes \os{(n-3)}{...} \otimes Q_1 \otimes \ol{Q}_1 \cong Q_1 \otimes \os{(n-4)}{...} \otimes \ol{P}_1 \cong Q_{n-2}
    \end{align*}
    which gives then that $P_n \cong Q_1 \otimes \ol{P}_1 \otimes Q_1 \otimes \os{(n-3)}{...} \otimes Q_1 \otimes \ol{P}_1$, as we wished. 

    When $n$ is odd, $\ol{R}_2=Q_1$. The principal graph nows gives $Q_{n-1}\otimes X\cong Q_{n-2}\oplus P_n$. Similar to the even case, using that $X=P_1+Q_1$, we get
    \begin{align*}
        Q_{n-1}\otimes X \cong Q_1 \otimes \ol{P}_1\otimes Q_1 \otimes \os{(n-3)}{...}\otimes \ol{P}_1 \otimes X\cong Q_1 \otimes \ol{P}_1\otimes Q_1 \otimes \os{(n-3)}{...}\otimes \ol{P}_1 \otimes P_1 + Q_1\otimes \ol{P}_1\otimes  Q_1 \otimes \os{(n-3)}{...}\otimes \ol{P}_1 \otimes Q_1. 
    \end{align*}
    Since $\ol{P}_1\otimes P_1 \cong \emptyset_{-}$, we get $Q_1 \otimes \ol{P}_1\otimes Q_1 \otimes \os{(n-4)}{...} \otimes Q_1 \cong Q_{n-2}$
      which then gives that $P_n\cong Q_1 \otimes \ol{P}_1\otimes Q_1 \otimes \os{(n-3)}{...} \otimes \ol{P}_1\otimes Q_1 $, completing the proof. \end{proof}
    
    We are now ready to prove the necessary relations for any subfactor planar algebra with principal graph $\til{A}_{2n-1}$. For diagrams $U$, $U^*$, $V$, and $V^*$ we intend to indicate that the strands are alternating in color. A purple strand indicates a strand of color blue or red. We also intend for these diagrams to have checkerboard shading. We shade the regions that must be shaded and use a ``checkerboard" to indicate a region that should have the appropriate shading, which will depend on $n$. 

    \begin{defn}[Elements in Affine $A$ Subfactor Planar Algebras] \label{eq: elementsofaffineAfiniteshaded}
        We will show in the following theorem that the $\til{A}_{2n-1}$ subfactor planar algebras have elements 
          \begin{equation}
        P_1=\redr, Q_1=\bluer, U=\Ushade, U^*=\Uadshade, V=\Vshade, V^*=\Vadshade.
    \end{equation}
    \end{defn}

    \begin{defn}[The Relations of Affine $A$ Finite Subfactor Planar Algebras] \label{relationsofAffineAfinite}
           We will show in the following theorem that the relations of the $\til{A}_{2n-1}$ finite subfactor planar algebras are:  
    \begin{enumerate}[label=(\roman*)]
        \item (the bubble relations) $\rshadeinbubble =\bshadeinbubble =\emptyset_{+}$, and $\rshadeoutbubble=\bshadeoutbubble=\emptyset_{-}$
        \item (strand decompositions) $X=\X=\redr+\bluer$ and $Y=\Y=\redl+\bluel$,
        \item (color disagreements) $\rbshader=\brshader=0$
        \item (the saddle relations) $\redsaddlesin= \redr \otimes \redl=\rrshademid$, $\redsaddlesout =\redl \otimes \redr=\rrshadeout$, $\bluesaddlesin=\bluer \otimes \bluel=\bbshademid$, and $\bluesaddlesout= \bluel \otimes \bluer = \bbshadeout$,
        
        \item (the unitary relations) $UU^*= \UUadshade = \redr \otimes \bluel \otimes \os{(n-1)}{...} \otimes \purpleshade$ = \rbpshadein, \\ $U^*U= \UadUshade =\bluer \otimes \redl \otimes \os{(n-1)}{...} \otimes \purpleshade=\brpshadein$, \\
        $V^*V=\VadVshade = \bluel\otimes \redr  \otimes \os{(n-1)}{...} \purpleshade=\brpshadeout$, and \\
        $VV^*=\VVadshade = \redl \otimes \bluer \otimes \os{(n-1)}{...} \otimes \purpleshade= \rbpshadeout$, and
        \item (the click relations) for some $n$th root of unity $\sigma_n$, $\mc{F}(U)=\sigma_n V^*=\sigma_n \mc{F}^{-1}(U)$, and $\mc{F}(U^*)=V=\sigma_n \mc{F}^{-1}(U^*)$, or diagrammatically, $\Ubclickshade =\sigma_n \Vadshade =\sigma_n \Urclickshade $, and \\
        $\Uadrclickshade=\Vshade=\sigma_n \Uadbclickshade$.
    \end{enumerate}
    \end{defn}

    We now show that the elements and relations defined above are necessarily elements and relations in an $\til{A}_{2n-1}$ subfactor planar algebra.

\begin{thm} \label{thm: Afinshadenecessary} (Necessary relations for $\til{A}_{2n-1}$). Fix $n\in \mathbb{N}$. If $\mc{P}$ is a subfactor planar algebra with principal graph and dual principal graph $\til{A}_{2n-1}$, then $\mc{P}$ has the elements from Definition (\ref{eq: elementsofaffineAfiniteshaded}) and relations from Definition (\ref{relationsofAffineAfinite}). Further, we get the equivalence classes of minimial projections are when $n\geq 2$:
    \begin{equation} \label{eq: Afinshadegraphs}
        \Gamma_{+}:\Afinshadegraphpos  \te{ and } \Gamma_{-} \Afinshadegraphneg
    \end{equation}
    and when $n=1$,
    \begin{equation}\label{eq: Afinshadegraphsone}
        \Gamma_{+}:\Afinshadegraphposone  \te{ and } \Gamma_{-} \Afinshadegraphnegone
    \end{equation}
    \end{thm}

    \begin{proof} By Lemmas \ref{lem: Afinshadeuniquerep} and  \ref{lem: nonep1q1reps} we know that there exists elements $P_1$ and $Q_1$ in $\mc{P}$ that can be denoted $P_1=\redr$ and $Q_1=\bluer$, $X=P_1+Q_1$, and $Y=\ol{P}_1+\ol{Q}_1$,  which gives relation (ii). Further by Lemma \ref{lem: a-fin-graph-index}, we get relation (i). When $n\geq 2$, $P_1$ and $Q_1$ are nonisomorphic minimal projections so $P_1Q_1=Q_1P_1=0$, which is also true for $n=1$. This gives relation (iii). Relation (iv) was given in part (i) of Lemma \ref{lem: Afinshadesaddlerels}. 

    Before showing the existence of elements $U$, $U^*$, $V$, and $V^*$, notice that Lemma \ref{lem: Afinshadedtensordecomp} gave us the equivalence classes of minimal projections shown in (\ref{eq: Afinshadegraphs}) for $n\geq 2$. Choose the representatives of each of the minimal projections to be this choice. Thus $P_n= P_1 \otimes \ol{Q}_1 \otimes P_1 \os{(n-2)}{...}\otimes R_1$, and $P_n'=\ol{P}_1\otimes Q_1 \otimes \ol{P}_1\otimes \os{(n-2)}{...}\otimes \ol{R}_1$ where $R_1$ is either $P_1$ or $\ol{Q}_1$ depending on the parity of $n$. Define $Q_n=Q_1 \otimes \ol{P}_1\otimes Q_1 \otimes \os{(n-2)}{...} \otimes \ol{R}_2$ and $Q_n'=\ol{Q}_1\otimes P_1 \otimes \ol{Q}_1 \otimes \os{(n-2)}{...}\otimes R_2$ where $R_2$ is either $P_1$ or $\ol{Q}_1$ and $R_2\neq R_1$. Then notice that Lemma \ref{lem: AfinshadePnisotoQn} gives that
    \begin{align*}
        P_n=\rbpshadein \cong \brpshadein=Q_n \text{ and } P_n'=\rbpshadeout \cong \brpshadeout=Q_n'
    \end{align*}
    Therefore there exists morphisms $U$, $U^*$, $W$, and $W^*$ in $\mc{P}$ such that $U\in \te{Hom}(Q_n,P_n)$, $U^*\in \te{Hom}(P_n,Q_n), W \in \te{Hom}(Q_n',P_n')$, and $W^*\in \te{Hom}(P_n',Q_n')$, such that $UU^*=\te{id}_{P_n}=P_n$, $U^*U=\te{id}_{Q_n}=Q_n$, $WW^*=\te{id}_{P_n'}=P_n'$, and $W^*W= \te{id}_{Q_n'}=Q_n'$, where the second equalities of all of these is due to $P_n$, $P_n'$ (and thus $Q_n$ and $Q_n'$) being projections. Diagrammatically, this means there are elements
    \begin{align*}
        \Ushade, \te{ } \Uadshade, \te{ } \Wshade, \te{ }\Wadshade
    \end{align*}
    in $\mc{P}$ and further, they satisfy relation (v). 

    Next, note that $\mc{F}(U)\in \te{Hom}(P_n',Q_n')$, and since $Q_n'\cong P_n'$, this hom-space is one-dimensional. Therefore there is an $\alpha \in \mathbb{C}$ such that $\mc{F}(U)=\alpha W^*$. Similarly, there exists a $\beta \in \mbb{C}$ where $\mc{F}(W^*)=\beta U$. Taking $\mc{F}^{-1}$ of both sides of each equation and moving the scalar the other side gives $\alpha^{-1} U =\mc{F}^{-1}(W^*)$ and $\beta^{-1}W^*=\mc{F}^{-1}(U)$. Taking the adjoint of both sides, which is vertically flipping and interchanging $U$ and $U^*$ as well as $W$ and $W^*$, gives $\mc{F}^{-1}(U^*)=\ol{\alpha}W$ and $\mc{F}^{-1}(W)=\ol{\beta}U^*$. Applying the click again, we obtain $\mc{F}(W)=\ol{\alpha}^{-1}U^*$ and $\mc{F}(U^*)=\ol{\beta}^{-1}W$. Since $\mc{F}(U)=\alpha W^*$, $\mc{F}^2(U)=\alpha \mc{F}(W^*)=\alpha\beta U$. Then because $\mc{F}^{2n}$ is the identity map, $\mc{F}^{2n}(U)=(\alpha\beta)^n(U)$. So $\alpha \beta$ is an $n$th root of unity. Define $\sigma_n=\alpha \beta$ $V=\beta W$, and $V^*=\ol{\beta}W^*$, and diagrammatically denote these as
    \begin{align*}
        \Vshade \te{ and } \Vadshade
    \end{align*}

    Next we show that $V$ and $V^*$ satisfy the requirements of relation (iv). To do so, we first show that $\ol{\beta}=\beta^{-1}$. This is true since, using sphericality, we can move the rightmost strand of $\te{tr}(UU^*)$ to the left, which is equivalent to $\te{tr}(\mc{F}^{-1}(U)\mc{F}(U^*))$. Therefore,
    \begin{align*}
        \te{tr}(Q_n)=\te{tr}(UU^*)=\te{tr}(\mc{F}^{-1}(U)\mc{F}(U^*))=\te{tr}(\beta^{-1}W^*\ol{\beta}^{-1}W)=\beta^{-1}\ol{\beta}^{-1}\te{tr}(W^*W)=\beta^{-1}\ol{\beta^{-1}}\te{tr}(Q_n')
    \end{align*}
    By relation (i), we get that $\te{tr}(Q_n')=1=\te{tr}(Q_n)$, and thus $\beta^{-1}\ol{\beta}^{-1}=1$. Hence, $\beta^{-1}=\ol{\beta}$. By instead taking trace of $U^*U$, we can do the same process as before and get $\ol{\alpha}=\alpha^{-1}$. Now we can see that $VV^*=WW^*=P_n'$ and $V^*V=W^*W=Q_n'$ giving that relation (iv) is satisfied for $V$ and $V^*$ as well. 

    Finally, we prove relation (v). Using the equalities found earlier we can see that 
    \begin{align*}
        \mc{F}(U)&=\alpha W^*=\alpha \beta\beta^{-1} W^*=\alpha \beta V^*=\sigma_n V^*, \te{ and}\\
        \mc{F}(U^*)&=\ol{\beta}^{-1}W=\beta W =V
    \end{align*}
    which give the first equalities of each of the equations in relation (v). Further,
    \begin{align*}
        \mc{F}(V)&=\beta\mc{F}(W)=\beta \ol{\alpha}^{-1}U^*=\beta \alpha U^*=\sigma_n U^*, \te{ and}\\
        \mc{F}(V^*)&=\beta^{-1}\mc{F}(W^*)=\beta^{-1}\beta U=U
    \end{align*}
    Taking $\mc{F}^{-1}$ of these last two equalities we get $V=\sigma_n \mc{F}^{-1}(U^*)$ and $V^*=\mc{F}^{-1}(U)$. This completes showing relation (vi) and hence the entire theorem. \end{proof}

    \begin{rem}
        We can always color a strand in its entirety by red or blue. Suppose a strand is colored more than one color. Then look at where the colors switch. If the switch is between red and blue then use color disagreement to make the diagram 0. If the switch is between black and red (or blue) then multiplying $\redr$ by $\X=\redr+\bluer$ (or the opposite shading) then using the color disagreement gives we can color the entire portion of strand red. Doing this process at every color change for every strand along with using strand decomposition to ensure there are no strictly black strands will give a diagram whose strands are either all red or all blue in their entirety. 
    \end{rem}
    
    In Theorem \ref{thm: Afinshade-generators} we will prove that the elements $P_1$, $Q_1$, $U$, $U^*$, $V$, and $V^*$ actually generate the planar algebra $\mc{P}$. Next, we will know show that if we have these elements as generators and the relations from Theorem \ref{thm: Afinshadenecessary}, we obtain the planar algebra $\mc{P}$.

\section{Sufficient Relations for Affine \texorpdfstring{$A$}{A} Finite Subfactor Planar Algebras}\label{Afinsufficientshaded}

   \begin{thm} \label{thm: Afinshadesufficient} (Sufficient relations for $\til{A}_{2n-1}$)
        Fix $n \in \mbb{N}$ and an $n$th root of unity $\sigma_n$. Let $\mc{P}$ be the planar algebra with generators $P_1$, $Q_1$, $U$, $U^*$, $V$, $V^*$ from Definition \ref{eq: elementsofaffineAfiniteshaded}
        with relations (i) through (vi) given in Definition \ref{relationsofAffineAfinite}. Define $P_1$ and $Q_1$ to be self-adjoint, $U^*$ to be the adjoint of $U$, and $V^*$ to be the adjoint of $V$, then extend the * operation anti-linearly and on diagrams. Then $\mc{P}$ is a subfactor planar algebra whose principal graph and dual principal graph are the $\til{A}_{2n-1}$ Dynkin diagrams given in Theorem \ref{thm: Afinshadenecessary}.
    \end{thm}

         Throughout the rest of this section, $\mc{P}$ will refer to the planar algebra defined through generators and relations given in the previous theorem. Fix $n\in\mbb{N}$ and a root of unity $\sigma_n$. Before proving Theorem \ref{thm: Afinshadesufficient}, we will prove a collection of facts following from this generators and relations presentation. It is clear that for any diagram $T$, $T^*$ is obtained by vertical flipping and taking adjoints of boxes.

    \begin{lem}
        The adjoint is well-defined.
    \end{lem}

    \begin{proof} Relations (i), (ii), (iii), and (iv) satisfy that taking their adjoint gives the same equality back clearly. For relation (v), taking the adjoint of the $UU^*$ equality gives the $U^*U$ equality and vice versa, taking the adjoint of $VV^*$ gives $V^*V$ and vice versa. Finally, for relation (vi), taking the adjoint of the first equality, $\mc{F}(U)=\sigma_n V^*=\sigma_n \mc{F}^{-1}(U)$, gives $\mc{F}^{-1}(U^*)=\sigma_n^{-1}V=\sigma_n^{-1}\mc{F}(U^*)$. Multiplying this second set of equalities by a $\sigma_n$, we obtain the second set of equalities in relation (vi). Taking the adjoint again we will arrive at the first set of equalities. Therefore, the adjoint is well-defined. \end{proof}

    The next step in proving Theorem \ref{thm: Afinshadesufficient} is to show the the 0-box space, $\mc{P}_0$ is one-dimensional. We do this in two parts. First, we show $\mc{P}_0$ is at least one-dimensional by defining a surjection from $\mc{P}_0$ to $\mbb{C}$ that sends $\emptyset_{+}$ and $\emptyset_{-}$ to 1. Then, to show that $\mc{P}_0$ is at most one-dimensional, we give an evaluation algorithm for the 0-box space. 
    
    To define the function we use the fact that the strands in any diagram of a planar algebra partitions the square picture into regions. From this fact we state the following remark. This comes from the theory of van Kampen diagrams.

    \begin{rem}
        For any diagram in $\mc{P}$ there exists a unique function from the regions of the diagram to the dihedral group $D_n=\langle r,b|r^2=b^2=(rb)^n=1\rangle$ such that the region on the left-most boundary next to the star is assigned 1 and neighboring regions separated by a red (respectively blue) strand are assigned elements of the form $x$ and $xr$ (respectively $x$ and $xb$). 
    \end{rem}

    We now define a surjection from $\mc{P}_0$ to $\mbb{C}$. Lemma \ref{lem: Afinshade-atleast1dim} will show $f$ is a well-defined surjection. It will then follow that $\mc{P}_0$ is at least one-dimensional. 
    
    \begin{tcolorbox}[breakable, pad at break*=0mm]
    Define a function $f:\mc{P}_0\to \mbb{C}$ by the following. By strand decomposition and color disagreement it suffices to define $f$ on diagrams where every strand is either red or blue in its entirety and then extend $f$ linearly. Let $D \in \mc{P}_0$ be such a diagram.

    \begin{enumerate}
        \item If there are no $U, U^*, V,$ and $V^*$-boxes, define $f(D)=1$. 
        \item If there is a $U$, $U^*$, $V$, or $V^*$ box, enumerate them all. Say $U_1,...,U_a, U^*_1,...,U^*_b, V_1,...,V_c,V^*_1,...,V^*_d$. 
        \item Define a number, $\ell_j^J$, based on the following table where $w_j^J$ is the element of $D$ that corresponds to the region where $J_j$'s star lies.
        
        \begin{tabular}{l | l}
            When $J=U$ or $V$, &
            When $J=U^*$ or $V^*$,\\
            \hline
            & \\
            $\bullet$ if $w_j^J=(rb)^m$, then $\ell_j^J=m$, & $\bullet$ if $w_j^J=(rb)^m$, then $\ell_j^J=-m$,\\
            & \\
            $\bullet$ if $w_j^J=b(rb)^m$, then $\ell_j^J=m+1$, & $\bullet$ if $w_j^J=b(rb)^m$, then $\ell_j^J=-(m+1)$\\ 
            & \\
            $\bullet$ if $w_j^J=(br)^m$, then $\ell_j^J=-m$, and  & $\bullet$ if $w_j^J=(br)^m$, then $\ell_j^J=m$, and  \\ 
            & \\
             $\bullet$ if $w_j^J=r(br)^m$, then $\ell_j^J=-m$. & $\bullet$ if $w_j^J=r(br)^m$, then $\ell_j^J=m$.
        \end{tabular}
        
            \item Define $\ell^{D}=\left(\sum_{j=1}^a \ell_j^U+\sum_{j=1}^b \ell_j^{U^*}\right)+\left(\sum_{j=1}^c \ell_j^V + \sum_{j=1}^d \ell_j^{V^*}\right) \te{ }(\te{mod }n)$
            \item Define $f(D)=\sigma^{\ell^{D}}$
            \item Extend $f$ linearly to evaluate on all of $\mc{P}_0$.
    \end{enumerate}

    \end{tcolorbox}
       There are some things to immediately note about this function that will be said without proof.

       \begin{rem}
           The number of $U$ and $U^*$ boxes, as well as the number of $V$ and $V^*$ boxes, are equal. That is $a=b$ and $c=d$.
       \end{rem}

       \begin{rem}
           The table in step 4 gives a well-defined function from $D_n$ to the integers. 
       \end{rem}

    \begin{lem}\label{lem: Afinshade-atleast1dim}
        $f$ is a well-defined surjection.
    \end{lem}

    \begin{proof} To show that $f$ is well-defined we must show that $f$ is invariant under the relations of $\mc{P}$.

    $f$ is obviously invariant under the strand decomposition and color disagreements. 
    
    If $D$ and $D'$ are identical except in a local region where $D$ looks like the left-hand side of a saddle relation and $D'$ has the right-hand side of the saddle relation, then its clear that the labellings of the regions by elements of $D_n$ will be identical. The same is true for the bubble relations. Thus $f$ is invariant under these two relations.

    Let $D$ and $D'$ are identical except in a local region $D$ is $UU^*$ written as left-hand diagram in the unitary relations and $D'$ is the right-hand diagram with no boxes. First notice that the stars of $U$ and $U^*$ lie in the same region. Their associated $\ell$-values will cancel. What's left to check is that the regions outside of this local neighborhood are identically labelled. Any regions to the top and bottom of this local region will be labelled the same since the top and bottom boundary data is the same for both $D$ and $D'$, so $f$ is invariant under this relation. Similar arguments can be made for the rest of the unitary relations. 

    For the final relation, let $D$, $D'$, and $D''$ be diagrams in $\mc{P}_0$, identical except in a local region $D$ looks like the left hand side of the first click relation $\mathcal{F}(U)$, $D'$ looks like $V^*$, and $D''$ is $\mc{F}^{-1}(U)$. Again, we need to check that the regions outside this neighborhood are labelled the same as well as that the $\ell$-values coming from these boxes are the same for $D'$ and $D''$ and are one less than $D$. 

    The former statement is clear. The strands on the top and bottom of the diagrams are the same colors in all 3 neighborhoods. The right-most region will then also be the same label.  
    
    Now consider the latter statement. Let $x$ be the group element in the left-most region of the local neighborhood. We break this into cases.  

    If $x=(rb)^m$, then for $D$, $xb=r(br)^{(m-1)}$ Therefore, the $\ell$-value of $U$ in $D$ be $-m+1$. For $D'$, $x=(rb)^m$ is the group element labelled the region of $V^*$'s star, so the $\ell$-value for $V^*$ will be $-m$. For $D'$, $xr=r(br)^{m}$, which will give $U$ and $\ell$-value of $-m$. 

    If $x=b(rb)^m$, then for $D$, $xb=b(rb)^mb=(br)^m$, and the $\ell$-value for $U$ will be $-m$. For $D'$, $V^*$'s $\ell$-value is $-m-1$. For $D''$, $xr=(br)^{m+1}$, so $U$ has an $\ell$-value of $-m-1$. 

    If $x=(br)^m$, then for $D$, $xb=b(rb)^m$, giving $U$ an $\ell$-value of $m+1$. For $D'$, $V^*$ has an $\ell$-value of $m$. For $D''$, $xr=b(rb)^{m-1}$, which gives $U$ an $\ell$-value of $m$.

    If $x=r(br)^m$, then for $D$, the region of $U$'s star is labelled $xb=(rb)^{m+1}$, which gives $U$ an $\ell$-value of $m+1$. For $D'$, the $\ell$-value of $V^*$ is $m$. For $D''$, region of $U$'s star is labelled $xr=(rb)^m$, which gives $U$ and $\ell$-value of $m$.
    
    Thus, in all these cases, when multiplying the $D'$ and $D''$ by another value of $\sigma_n$, the resulting values of the function for $D$, $D'$, and $D''$ will be identical. A similar computation can be done for the second click relation. 

    Therefore $f$ is invariant under the relations of $\mc{P}$ and is well-defined. \end{proof} 

    \begin{cor}
        $\mc{P}_0$ is at least one-dimensional.
    \end{cor}

    \begin{proof} $f$ is a well-defined function. Further, $f(\emptyset_{+})=1$, so $f$ is a surjection. Therefore $\mc{P}_0$ is at least one-dimensional. \end{proof}

    Next, we define an algorithm to evaluate any element in $\mc{P}_0$. In doing so, we show that $\mc{P}_0$ is at most one-dimensional.
    
    \begin{tcolorbox}[breakable, pad at break*=0mm]
        Evaluation Algorithm for $\mc{P}_0$: 
        \\
    By strand decomposition and color disagreement it suffices to define the algorithm on diagrams where every strand is either red or blue in its entirety then extend linearly. Let $D\in \mc{P}_0$ be such a diagram. 
        
        \begin{enumerate}
            \item If $D$ has no $U$, $U^*$, $V$, and $V^*$ boxes, skip to step 10.
            \item If there is a $U$, $U^*$, $V$, or $V^*$-box, pick one of them and label it $W_{\te{up}}$. 
            \item Due to the checkerboard shading and the alternating colors of the strands, the strands of $W_{\te{up}}$ cannot have both its end on itself. Since $D$ is in the 0-box space, all the strands of $D$ cannot go to the boundary, so for the left-most strand of the bottom of $W_{\te{up}}$ must lie on another box. Label this box $W_{\te{down}}$.
            \item Click $W_{\te{down}}$ so that the connected strand is the top leftmost strand. This may multiply the diagram by some factor of $\sigma_n$ as well as change the type of box. We will continue to call this box $W_{\te{down}}$. 
            \item Isotope $W_{\te{down}}$ so that is is directly below $W_{\te{up}}$.

            \item The strands adjacent to the connected strands of $W_{\te{up}}$ and $W_{\te{down}}$ must be of the same color as well as agree with shading. The strand adjacent on the right of the connected strand of $W_{\te{down}}$ must be able to be isotoped so that with the strand right of the connected strand of $W_{\te{up}}$, the saddle relation can be performed so that the ends of one of the strands lies on both $W_{\te{up}}$ and $W_{\te{down}}$. 

            \item Inductively repeat the process of connecting the strands adjacent on the right of the previous strand of $W_{\te{up}}$ until $W_{\te{up}}$ and $W_{\te{down}}$ are connected by $n$ adjacent strands. 

            \item The bottom $n$ strands of $W_{\te{up}}$ are now connected to the top $n$ strands of $W_{\te{down}}$ so $W_{\te{up}}$ and $W_{\te{down}}$ must be adjoints of each other. Therefore we can use the unitary relations to rewrite $D$ as a diagram with two less boxes. Repeat steps 2-7 until we are left with 0 or 1 boxes. 

            \item If 1 box remains, its strands must connect to itself, so $D=0$. 

            \item If 0 boxes remain, $D$ is now a closed diagram consisting of only red or blue strands and is maybe multiplied by a scalar. Starting with the innermost loop, repeatedly pop the loops using the bubble relations. $D$ will then evaluate to a multiple of $\emptyset_{+}$ or $\emptyset_{-}$.
        \end{enumerate}
    \end{tcolorbox}

    With the evaluation algorithm defined, now the following lemma and its corollary becomes instant. 

    \begin{lem}\label{lem: Afinshade-suff-oned}
        $\mc{P}_0$ is at most one-dimensional.
    \end{lem}

    \begin{cor}
        $\mc{P}_0$ is one-dimensional.
    \end{cor}

    \begin{lem}\label{lem: saddles-are-minimal} All of the diagrams in the saddle relations:
        \begin{enumerate}[label=(\roman*)]
            \item $P_1 \otimes P'_1$,
            \item $Q_1 \otimes Q'_1$,
            \item $P'_1\otimes P_1$, and
            \item $Q'_1 \otimes Q_1$.
        \end{enumerate}
        are minimal projections. Further, (i) and (ii) are isomorphic to $\emptyset_{+}$ and (iii) and (iv) are isomorphic to $\emptyset_{-}$.
    \end{lem}

    \begin{proof} We will just show $P_1 \otimes P_1'$ is a minimal projection. The rest of the items have nearly identical arguments. It is clear that $P_1\otimes P_1'$ is a projection. Let $f \in \text{Hom}(P_1\otimes P_1', P_1\otimes P_1')$. If $f\neq 0$, then using the evaluation algorithm to reduce the number of boxes in $f$ connected by a strand we can assume $f$ can be written to be a collection of boxes not connected to any other box and has no closed components. Thus the strands of all the boxes in $f$ must go to the top and bottom of the diagram, as $f\neq 0$. However, there are only red strands on the boundary and all boxes for have at least one blue strand. Thus $f$ cannot contain a box. 

    Assume $f$ has no boxes. Then it is clear the only options for $f$ are scalar multiples of $P_1\otimes P_1'$ or scalar multiples of the cup/cap diagram on the left side of the saddle relation. Both of these will result in $f$ being a scalar multiple of itself, due to the saddle relations. Therefore the dimension of $\te{Hom}(P_1\otimes P_1', P_1\otimes P_1')$ is at most 1. Since the trace of $P_1\otimes P_1'$ is 1, we get that the dimension of $\te{Hom}(P_1\otimes P_1', P_1\otimes P_1')$ is 1, so $P_1\otimes P_1'$ is a minimal projection. 
    
    We want to show that $\emptyset_{+}\cong P_1 \otimes P_1'\cong Q_1 \otimes Q_1'$ and $\emptyset_{-}\cong P_1'\otimes P_1 \cong Q_1'\otimes Q_1$. Since the dimension of $\mathcal{P}_{0,\pm}$ is 1, we know that both $\emptyset_{+}$ and $\emptyset_{-}$ are minimal projections. By the saddle relations, $P_1\otimes P_1'$, $P_1'\otimes P_1$, $Q_1 \otimes Q_1'$, and $Q_1'\otimes Q_1$ are all factoring through either $\emptyset_+$ or $\emptyset_{-}$. Since all of these diagrams are nonzero, this gives the desired isomorphisms. \end{proof}

    Next we show that the $\til{A}_{2n-1}$ Dynkin diagram are the correct principal and dual principal graphs for this planar algebra. 

    \begin{lem}\label{lem: Afinshade-suff-pgraph}
        The principal graphs of $\mc{P}_0$ are the graphs given in (\ref{eq: Afinshadegraphs}) for $n\geq 2$ and (\ref{eq: Afinshadegraphsone}) for $n=1$. 
    \end{lem}

    \begin{proof} Define $P_k=P_1 \otimes \ol{Q}_1 \otimes \os{(k-1)}{...} \otimes R_1$, $P_{k}'=\ol{P}_1 \otimes Q_1 \otimes \os{(k-1)}{...} \otimes \ol{R}_1$, $Q_{\ell}=Q_1\otimes \ol{P}_1 \otimes \os{(\ell-1)}{...}\otimes \ol{R}_2$, and $Q_{\ell}'=\ol{Q}_1\otimes P_1 \otimes \os{(\ell-1)}{...} \otimes R_2$ for $1\leq k \leq n$ and $1\leq \ell \leq n-1$, where $R_1, R_2\in \{P_1,\ol{Q}_1\}$ and depend on $n$. Additionally, define $P_0=\emptyset_+$ and $P'_0=\emptyset_{-}$, which were already shown to be minimal projections since $\mc{P}_{0,\pm}$ is one-dimensional. As the adjoint distributes over $\otimes$ and $P_1$, $P_1'=\ol{P}_1$, $Q_1$, and $Q_1'=\ol{Q}_1$ are clearly projections, for all $1\leq k\leq n$ and $1 \leq\ell\leq n-1$ $P_k$, $P_k'$, $Q_\ell$, and $Q_\ell'$ are projections.
        
    To then show that these are in fact minimal projections, we want to show that for all $1\leq k \leq n$, $1\leq \ell \leq n-1$, the dimensions of $\te{Hom}(P_k,P_k)$, $\te{Hom}(P_k',P_k')$, $\te{Hom}(Q_\ell, Q_\ell)$, and $\te{Hom}(Q_\ell',Q_\ell')$ are 1. Notice that taking the trace of all these projections gives 1, so the dimension of these hom-spaces is greater than 0. We will just prove the dimension of the hom-space of $P_k$ is 1. The rest of the dimensions can be found nearly identically. 

    Fix $k$ and some $f\in \te{Hom}(P_k, P_k)$. Diagrammatically we have 

    \begin{align*}
        f= \homPkPkshade
    \end{align*}

    If $f$ has two $U$, $U^*$, $V$, and $V^*$ boxes connected by a strand, we can use the method outlined in the evaluation algorithm for $\mc{P}_0$ to write $f$ with two less boxes. Thus it is sufficient to think of $f$ as either having zero boxes, or having a collection of boxes each not connected to each other. In the latter case, this means all the strands of these boxes must be connected to the boundary or themselves. If $k<n$, then the number of points on the boundary of the $f$-box is less than $2n$, so all the boxes inside of $f$ would have to have some strands connecting to themselves, which would make $f=0$. If $k=n$ then the number of points on the boundary of the $f$-box is $2n$. If there is more than one box inside of $f$ then one of the boxes must have at least one strand connecting to itself, making $f=0$. If $f$ has exactly one box, then since the strands of any box alternates in color and the top and bottom leftmost strand on the boundary of $f$ are both red, for all the strands of the box to connect to the boundary of $f$ would require some color disagreement, making $f=0$. 
     Therefore, for $f$ to be nonzero, $f$ must have zero $U$, $U^*$, $V$, and $V^*$ boxes. 

    Assume $f\neq 0$, which by the preceding paragraph means that $f$ has zero boxes. The strands on the boundary of $f$ are alternating in color on the top and on the bottom, so $f$ cannot have any caps or cups. Therefore, $f$ must be some straight strands along with some circles, which can be popped for a scalar, say $\lambda \in \mbb{C}$. The only possibility for $f\neq 0$ with straight strands would mean the straight strands are $P_k$. Therefore $f=\lambda P_k$, giving that the dimension of $\te{Hom}(P_k,P_k)$ is 1. 

    Next we show that these are nonisomorphic minimal projections by showing that there are no nonzero morphisms between any of these distinct projections. We break this into cases depending on if the projections start with the same color strand. Define $P_0 =Q_0=\emptyset_{+}$ and $P'_0=Q'_0=\emptyset_{-}$. Fix $0\leq k\leq n$ and $0 \leq \ell \leq n-1$. 

    \textit{Case 1:} Suppose that $f$ is a nonzero morphism between the $P_k$ or $P_k'$ projection and the $Q_\ell$ or $Q_\ell'$ projection. Notice that on the boundary of $f$, the top and bottom leftmost strands will be of different colors. Without loss of generality, assume that $f$ has no closed component. Using the method outline in the evaluation algorithm for $\mc{P}_0$, write $f$ as a diagram where all of the boxes are not connected to each other. If the strands of the boxes are connected to themselves, then $f=0$. So every boxes' strands must connect to the boundary as $f\neq 0$. The number of strands on the boundary is $k+\ell$. If $k+\ell<2n$ then there are not enough points on the boundary for even one box. As $f\neq 0$, we then know $f$ contains no boxes. If $k+\ell=2n$ then $k=\ell=n$. However, $\ell\leq n-1$, so this case is eliminated. Now we are left with considering $f$ with no boxes inside. However, since the leftmost stands are of opposite colors and the colors on the top and bottom alternate, there is no such nonzero diagram $f$.

    \textit{Case 2:} Suppose that $f$ is a nonzero morphism between two minimal projections starting with the same color on the leftside. Without loss of generality, lets assume $f$ is a morphism from $P_{k_1}$ to $P_{k_2}$ where $k_1,k_2 \in \{0,...,n\}$. Again, we only need to consider the case where $f$ has no closed component and all of the strands of the boxes connect to the boundary. The number of strands on the boundary is $k_1+k_2$ which is strictly less than $2n$ unless $k_1=k_2$ which would imply the projections are equal. Thus all that is left to consider when $f$ has no boxes. There can be no caps or cups in $f$ due to the alternating colors of the strands. Therefore, the only options is for $f$ to be alternating strands and again $k_1=k_2$. 

    Naturally define $Q_n$ and $Q_n'$. Using the same methods as before we can show $Q_n$ and $Q_n'$ are minimal projection nonisomorphic to $P_k$, $P_k'$, $Q_\ell$, and $Q_\ell'$ for $0\leq k,\ell \leq n-1$. However, it is true that $Q_n \cong P_n$ and $Q_n' \cong P_n'$. To see this, notice that $U$ and $V$ are morphisms between $Q_n$ and $P_n$ and $Q_n'$ and $P_n'$ respectively. Since the trace of $UU^*$ and $VV^*$ are both nonzero, $U$ and $V$ are nonzero morphisms. Therefore there's a nonzero morphism between minimal projections giving that $Q_n \cong P_n$ and $Q_n' \cong P_n'$, as desired.

    What is left to show is that these nonisomorphic minimal projections satisfy the fusion rules for the $\til{A}_{2n-1}$ Dynkin diagram with labellings given by (\ref{eq: Afinshadearbgraphs}) or (\ref{eq: Afinshadearbgraphsone}). First, let $n\geq 2$ and define $Q_n=Q_{n-1}\otimes \ol{R}_3$ where $\ol{R}_3\neq \ol{R}_2$ and $R_3\in \{P_1,\ol{Q}_1\}$, and define $Q_n'$ similarly. Using the strand decompositions, $\emptyset_{+} \otimes X \cong P_1 \oplus Q_1$ and $\emptyset_{-} \otimes Y \cong \ol{P}_1 \oplus \ol{Q}_1$. Now we will show inductively that $P_k \otimes R$ (where $R$ is $Y$ when $k$ is odd and $X$ when $k$ is even) is isomorphic to $P_{k-1}\oplus P_{k+1}$ for $1\leq k\leq n-1$. Similar arguments will then show the necessary tensor decompositions for $P_k'$, $Q_k$, and $Q_k'$.

    For the base case, notice that $P_1 \otimes Y\cong P_1 \otimes (P_1' \oplus Q_1') \cong (P_1 \otimes P_1') \oplus (P_1 \otimes Q_1')\cong  \emptyset_{+} \oplus P_2$ using strand decomposition and Lemma \ref{lem: saddles-are-minimal}. Next assume $P_k\otimes R \cong P_{k-1}\oplus P_{k+1}$ up to some $k<n-1$. Then if $n-1$ is even,
    \begin{align*}
        P_{n-1} \otimes X \cong P_{n-1} \otimes (P_1 + Q_1)\cong (P_{n-1} \otimes P_1) \oplus (P_{n-1} \otimes Q_1)
    \end{align*}
    When $n-1$ is even, the rightmost strand of $P_{n-1}$ will be $\ol{Q}_1$. Therefore $P_{n-1}\otimes Q_1 \cong P_{n-2}$ and $P_{n-1}\otimes P_1 \cong P_n$. Thus $P_{n-1}\otimes X \cong P_{n}\oplus P_{n-2}$. If $n-1$ is odd,
    \begin{align*}
        P_{n-1}\otimes Y \cong P_{n-1} \otimes (\ol{P}_1 +\ol{Q}_1) \cong (P_{n-1} \otimes \ol{P}_1)\oplus (P_{n-1}\otimes \ol{Q}_1) 
    \end{align*}
     When $n-1$ is odd, the rightmost strand of $P_{n-1}$ is $P_1$ so $P_{n-1}\otimes \ol{P}_1\cong \emptyset_{+}$ and $P_{n-1}\otimes \ol{Q}_1 \cong P_n$. Therefore $P_{n-1}\cong P_{n-2}\oplus P_n$, concluding our inductive proof. 

     Next I show that $P_n\otimes R \cong P_{n-1}\oplus Q_{n-1}$. Following the method of the above paragraph gives $P_{n}\otimes R \cong P_{n-1} \oplus (P_{n}\otimes S)$ where $S=P_1$ if $n$ is even and $S=\ol{Q}_1$ when $n$ is odd. Further, $P_n \cong Q_n$, so $P_n \otimes S \cong Q_n \otimes S$ which will be isomorphic to $Q_{n-1}\otimes \emptyset_+$ or $Q_{n-1}\otimes \emptyset_{-}$ depending on the parity of $n$. Therefore $P_n \otimes R \cong P_{n-1} \oplus Q_{n-1}$. 

     For $n=1$, $\emptyset_{+}\otimes X\cong X=P_1+Q_1$ by definition of tensor and strand decomposition. Further, using that $P_1 \cong Q_1$ we get $X\cong P_1\oplus Q_1 \cong P_1 \oplus P_1$. Additionally, $P_1 \otimes Y\cong P_1 \otimes (\ol{P}_1 + \ol{Q}_1)$ by the strand decomposition, which is isomorphic to $P_1\otimes \ol{P}_1+P_1\otimes \ol{Q}_1\cong \emptyset_{+}\oplus \emptyset_{+}$ by Lemma \ref{lem: saddles-are-minimal}. 
     
     Similarly, for any $n\in \mbb{N}$, one can show the necessary tensor decompositions for the dual principal graph which concludes the lemma. \end{proof}

\begin{lem}\label{lem: Afinshade-spherical}
         $\mc{P}$ is spherical.
\end{lem}

     \begin{proof} To show this lemma it suffices to prove that $\mc{P}_{1,+} =\te{ span}\{P_1, Q_1\}$ and $\mc{P}_{1,-}=\{\ol{P}_1, \ol{Q}_1\}$. The inclusion of $\te{ span}\{P_1, Q_1\}$ into $\mc{P}_{1,+}$ is obvious. Let $T$ be an element in $\mc{P}_{1,+}$. Then $T$ is a linear combination of diagrams from $\te{Hom}(P_1,P_1)$, $\te{Hom}(Q_1,Q_1)$,  $\te{Hom}(P_1,Q_1)$, and  $\te{Hom}(Q_1, P_1)$. Since $P_1$ and $Q_1$ are nonisomorphic minimal projections we get that the only hom-spaces with nonzero dimension are $\te{Hom}(P_1,P_1)$ and $\te{Hom}(Q_1,Q_1)$. Further, these hom-spaces are both one-dimensional. Therefore $T$ is in $\te{ span}\{P_1, Q_1\}$, so $\mc{P}_{1,+} =\te{ span}\{P_1, Q_1\}$. The proof of $\mc{P}_{1,-}$ is the same. \end{proof}

     \begin{lem}\label{lem: Afinshade-posdef}
         $\mc{P}$ is positive definite.
     \end{lem}

     \begin{proof} In \cite{MPS10} the authors give and explicit positive definite basis of each of the box spaces for any planar algebra whose corresponding category of projections is semisimple. We can adapt this proof slightly to instead give a positive definite basis of $\mc{P}_{2m,\pm}$ which is now identified with the hom-space from $\emptyset$ to alternating tensor products of $X$ and $\overline{X}$. The trees they create would also alternate on the boundary between $X$ and $\overline{X}$.  Additionally, in our case, $X$ is not a minimal projection, however, they do not need this in their proof. The authors then show in Theorem 4.18 that their defined tree diagrams with boundary labelled entirely by $X$ give a positive orthogonal basis for the space $\te{Hom}(\emptyset, X^{\otimes k})$, which corresponds to our vector spaces $\mc{P}_{2m,\pm}$. The only thing left to check is that all of the minimal projections have positive trace. Indeed, $\te{tr}(P_1)=\te{tr}(Q_1)=\te{tr}(\ol{P}_1)=\te{tr}(\ol{Q}_1)=1$ and the rest of the minimal projections are tensor products of these elements, so the minimal projections all have trace 1, so $\mc{P}$ is positive definite. \end{proof}

     \begin{proof}[Proof of Theorem \ref{thm: Afinshadesufficient}] All of the work has already been done. From Lemma \ref{lem: Afinshade-suff-oned} we get that the space of closed diagrams is one-dimensional. Lemma \ref{lem: Afinshade-spherical} gives that $\mc{P}$ is spherical and Lemma \ref{lem: Afinshade-suff-pgraph} gives that the principal graph is the $\til{A}_{2n-1}$ Dynkin diagram as shown in Lemma \ref{lem: Afinshade-suff-pgraph}. Finally, we get that $\mc{P}$ is positive definite from Lemma \ref{lem: Afinshade-posdef}, which completes the proof. \end{proof}

     Now we can return to Theorem \ref{thm: Afinshadenecessary} and show that the elements listed are in fact generators so the Theorem is a presentation of the planar algebra. 

     \begin{thm}\label{thm: Afinshade-generators}
         The elements $P_1, Q_1, U, U^*, V,$ and $V^*$ from Theorem \ref{thm: Afinshadenecessary} generate $\mc{P}$. 
     \end{thm}

     \begin{proof} From Theorem \ref{thm: Afinshadesufficient}, we know $P_1, Q_1, U, U^*, V$, and $V^*$ generate $\mc{P}$, a subfactor planar algebra with principal graph $\til{A}_{2n-1}$. Let $\mc{Q}$ be the subfactor planar algebra from Theorem \ref{thm: Afinshadenecessary}. The elements $P_1, Q_1, U, U^*, V$, and $V^*$ will generate $\mc{Q}$ if for all $k\in \mathbb{Z}_{\geq 0}$, $\te{dim}(\mc{P}_{k,\pm})=\te{dim}(\mc{Q}_{k,\pm})$. Returning to the explicit positive definite basis for each of the box spaces created in \cite{MPS10}, these bases rely solely on the principal graph and the traces of the minimal projections. As these are equal in both $\mc{P}$ and $\mc{Q}$, we get that $\te{dim}(\mc{P}_{k,\pm})=\te{dim}(\mc{Q}_{k,\pm})$ for all $k\geq 0$. Thus $P_1$, $Q_1$, $U$, $U^*$, $V$, and $V^*$ are generators of $\mc{Q}$ since they are generators of $\mc{P}$. \end{proof}

\section{Distinctness and Fulfillment of Presentations for Affine \texorpdfstring{$A$}{A} Finite Subfactor Planar Algebras}

 From the Theorems \ref{thm: Afinshadesufficient} and \ref{thm: Afinshade-generators} we find that Theorem \ref{thm: Afinshadenecessary} with generators $P_1, Q_1, U, U^*, V,$ and $V^*$ give $n$ presentations of the subfactor planar algebra with principal graph $\til{A}_{2n-1}$. It is then quite natural to ask if any of these are isomorphic and if there are any presentations left to be found. Indeed, none of these are isomorphic and there are no more, so we can prove the below theorem, which is a known result by Popa \cite{Pop94}; however, our proof is strictly using the planar algebra presentations. 

\begin{thm}\label{thm: Afinshade-nnoniso}
    There are exactly $n$ nonisomorphic subfactor planar algebras with principal graph $\til{A}_{2n-1}$. 
\end{thm}

    First we will prove that none of the $n$ presentations given are isomorphic.

\begin{lem} \label{lem: Afinshade-noniso}
    Let $\mc{P}$ and $\mc{Q}$ be planar algebras given by generators from Definition \ref{eq: elementsofaffineAfiniteshaded} and relations from Definition \ref{relationsofAffineAfinite}   , using $n$th roots of unity $\sigma$ and $\til{\sigma}$ respectively. If $\sigma\neq \til{\sigma}$ then $\mc{P}$ and $\mc{Q}$ are nonisomorphic. 
\end{lem}

    \begin{proof} When $n=1$ there is only one 1th root of unity, so we can assume $n\geq 2$. Call the boxes of $\mc{P}$: $U$, $U^*$, $V$, and $V^*$ and the boxes of $\mc{Q}$: $\til{U}$, $\til{U}^*$, $\til{V}$, and $\til{V}^*$. Call the distinct minimal projections in $\mc{P}_{1,+}$ and $\mc{Q}_{1,+}$ $P_1$, $Q_1$, $\til{P}_1$, and $\til{Q}_1$ respectively. For sake of contradiction let $\theta: \mc{P}\to \mc{Q}$ be a planar algebra isomorphism. Then $\theta_{1,+}:\mc{P}_{1,+}\to \mc{Q}_{1,+}$ is a bijective map. 

    A planar algebra isomorphism must preserve minimal projections so either $\theta_{1,+}(P_1)=\til{P}_1$ and $\theta_{1,+}(Q_1)=\til{Q}_1$ or $\theta_{1,+}(P_1)=\til{Q}_1$ and $\theta_{1,+}(Q_1)=\til{P}_1$. From knowing $\theta_{1,+}$, we can completely determine $\theta_{1,-}$. Let $T$ be the 1-click tangle on the $(1,+)$-box space. Then $\theta_{1,-}(Z_{T}(P_1))=\theta_{1,-}(\ol{P}_1)$
    which is equivalent to $Z_{T}\left(\theta_{1,+}(P_1)\right)$.
    
    A planar algebra isomorphism also preserves tensor product. Consider when $\theta_{1,+}(P_1)=\til{P}_1$. Since $\te{Hom}(P_1\otimes \ol{Q}_1 \otimes \os{(n-1)}{...} \otimes R_1,Q_1\otimes \ol{P}_1 \otimes \os{(n-1)}{...} \otimes R_1)$ is one-dimensional, there exists a $\rho_1 \in \mbb{C}$ such that $\theta(U^*)=\rho_1 \til{U}^*$ and similarly, there is a $\rho_2 \in \mbb{C}$ with $\theta(V)=\rho_2 \til{V}$. By definition of isomorphism, $\mc{F}(\theta(U^*))=\mc{F}(\rho_1 \til{U}^*)=\rho_1 \mc{F}(\til{U}^*)=\rho_1 \til{V}$  is equivalent to $\theta(\mc{F}(U^*))=\theta(V)=\rho_2 \til{V}$ giving that $\rho_1=\rho_2$. Since $\theta$ is an isomorphism and $U^*$ is a generator of a one-dimensional hom-space, $\rho_1\neq 0$. Again, using the definition of isomorphism, $\mc{F}^{-1}(\theta(U^*))=\mc{F}^{-1}(\rho_1\til{U}^*)=\rho_1\til{\sigma}^{-1}\til{V}$ is equivalent to $\theta(\mc{F}^{-1}(U^*))=\sigma^{-1}\theta(V)=\sigma^{-1}\rho_1\til{V}$. Therefore, $\til{\sigma}^{-1}=\sigma^{-1}$, which is a contradiction. 

    Consider when $\theta_{1,+}(P_1)=\til{Q}_1$. Similar to case (ii), there exists $\rho_1, \rho_2, \rho_3,$ and $\rho_4\in \mbb{C}-\{0\}$ where $\theta(U)=\rho_1\til{U}^*$, $\theta(V^*)=\rho_2\til{V}$, $\theta(U^*)=\rho_3 \til{U}$, and $\theta(V)=\rho_4 \til{V}^*$. Using the click relations we get $\theta(\mc{F}(U))=\theta(\sigma V^*)=\sigma\rho_2 \til{V}$ is equivalent to $\mc{F}(\theta(U))=\mc{F}(\rho_1\til{U}^*)=\rho_1\til{V}$, so $\sigma=\rho_1/\rho_2$. However, $\theta(\mc{F}(V^*))=\theta(U)=\rho_1\til{U}^*$ equals $\mc{F}(\theta(V^*))=\mc{F}(\rho_2\til{V})=\rho_2\til{\sigma}\til{U}^*$ so $\til{\sigma}=\rho_1/\rho_2$, implying that $\til{\sigma}=\sigma$. Therefore if $\sigma\neq \til{\sigma}$, $\mc{P}$ and $\mc{Q}$ are nonisomorphic. \end{proof}
    
     \begin{proof}[Proof of Theorem \ref{thm: Afinshade-nnoniso}] By Lemma \ref{lem: Afinshade-noniso}, none of the $n$ presentations found are isomorphic. Thus we have at least $n$ nonisomorphic subfactor planar algebras with principal graph $\til{A}_{2n-1}$. By Theorems \ref{thm: Afinshadenecessary} and \ref{thm: Afinshade-generators} if we have a subfactor planar algebra with principal graph $\til{A}_{2n-1}$ then it must have one of the presentations listed. Thus there are exactly $n$ nonisomorphic subfactor planar algerbas with principal graph $\til{A}_{2n-1}$. \end{proof}
      
\section{Necessary Relations for Affine \texorpdfstring{$A$}{A} Finite Unshaded Planar Algebras}

    While the affine $A$ principal graphs for subfactor planar algebras must have an even number of vertices on their principal graph, we will see in this section that there is no such requirement for unshaded subfactor planar algebras. However, we will also see that depending on the parity of the number of vertices, the number of nonisomorphic planar algebras will vary. Fix $n \in \mathbb{N}$ and define $\mc{P}$ to be a unshaded subfactor planar algebra with principal graph the $\til{A}_{2n-1}$ Dynkin diagram and let $\mc{P'}$ be a unshaded subfactor planar algebra with principal graph the $\til{A}_{2n}$ Dynkin diagram. The next goal is to find all diagrammatic presentations of these planar algebras. We find that many of these results have similar proofs to the shaded case and thus will omit repetitive proofs as well as omit proofs for $\mc{P}'$ that are essentially the same as $\mc{P}$. Define
    \begin{align*}
        X= \strand.
    \end{align*}
    By Lemma \ref{lem: a-fin-graph-index} we already know $\mc{P}$ has index 4. The same proof shows that $\mc{P}'$ has index 4. 
    
    \ind Vertices of a principal graph of an unshaded subfactor planar algebra $\mc{P}$ correspond to equivalence classes of minimal projections in the category $\mc{C}_\mc{P}$. Let $\mc{P}$ have the principal graph $\til{A}_{2n-1}$ with the following labelling when $n\geq 2$
    \begin{equation}\label{eq: arbAfingraph}
             \arbAfingraph
    \end{equation}
    and when $n=1$,
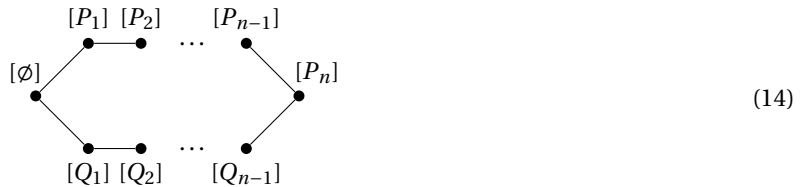
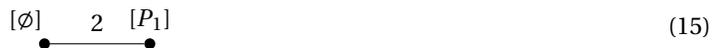
       \begin{equation}\label{eq: arbAfingraphone}
             \arbAfingraphone
    \end{equation}
    and let $\mc{P}'$ have principal graph $\til{A}_{2n}$ with the following labelling
    \begin{equation}\label{eq: arbAfingrapheven}
        \arbAfingrapheven
    \end{equation}
    where each $P_i$, $1\leq i \leq n$, and $Q_j$, $1\leq j \leq n$, are representatives of their respective equivalence class. It should be clear from context if we are referring to the projections in $\mc{P}$ or $\mc{P}'$. 
    
    By Lemma \ref{lem: a-fin-graph-index}, the trace of all the minimal projections are 1 in $\mc{P}$ (and $\mc{P}'$) and by Lemma \ref{lem: Afinshadeuniquerep}, for all but those with principal graph $\til{A}_1$, there is a unique representative of $[P_1]$ and $[Q_1]$ in $\mc{P}_1$ (respectively $\mc{P}'_1)$, which we will call $P_1$ and $Q_1$ and $X=P_1+Q_1$. Further, Lemma \ref{lem: Afinshade-basis} gives that $\{P_1, Q_1\}$ form a basis of $\mc{P}_1$ (respectively $\mc{P}_1'$) for all but the planar algebras with principal graph $\til{A}_1$. 

    \begin{lem}\label{lem: A-fin-strand-decomp}
         For all unshaded subfactor planar algebras with affine $A$ principal graphs, except $\til{A}_1$, the distinct minimal projections in the 1-box space, $P_1$ and $Q_1$ satisfy that either they are self-dual or are dual to each other. 
    \end{lem}

    \begin{proof} Without loss of generality, consider $\mc{P}$. Since $\{P_1, Q_1\}$ form a basis of $\mc{P}_1$, there exists $\lambda,\mu,\nu,\xi \in \mbb{C}$ so that $\ol{P}_1=\lambda P_1 + \mu Q_1$ and $\ol{Q}_1=\nu P_1+\xi Q_1$. Using that $P_1^2=P_1$, $Q_1^2=Q_1$, $P_1Q_1=Q_1P_1=0$, and that these properties must hold for their duals, we get that $\ol{P}_1$ and $\ol{Q}_1$ are distinct minimal projections in the 1-box space. Thus either $\ol{P}_1=P_1$ and $\ol{Q}_1=Q_1$ or $\ol{P}_1=Q_1$ and $\ol{Q}_1=P_1$. \end{proof}

\begin{lem}\label{lem: Aunshadenonep1q1}
    For unshaded subfactor planar algebras with principal graph $\til{A}_1$, there exists $R_1,Q_1 \in \mc{P}_1$ isomorphic to $P_1$ with $R_1Q_1=Q_1R_1=0$ and $X=R_1+Q_1$, where either $R_1$ and $Q_1$ are self-dual or dual to each other.
\end{lem}

    \begin{proof} From the principal graph, $X \cong P_1 \oplus P_1$ which gives that the dimension of $\mc{P}_1$ is 4. Further, $\mc{P}_1$ is an associative algebra isomorphic to $\te{Hom}(P_1 \oplus P_1, P_1 \oplus P_1)$. As $P_1$ is minimal projection, this hom-space is isomorphic to $M_2(\mbb{C})$. The dual map of is an anti-involution of $\mc{P}$. For any anti-involution of $M_2(\mbb{C})$, there exists nonzero $R_1,Q_1$ in $M_2(\mbb{C})$ that are either fixed or swapped such that $R_1^2=R_1$, $Q_1^2=Q_1$, $R_1+Q_1=I$, and $R_1Q_1=Q_1R_1=0$. As $\mc{P}_1$ is isomorphic to $M_2(\mbb{C})$ this implies that there exists elements in $\mc{P}_1$, abusing notation we will still call them $R_1$ and $Q_1$, such that $X\cong R_1 \oplus Q_1$, $R_1Q_1=Q_1R_1=0$, $R_1^2=R_1$, $Q_1^2=Q_1$. \end{proof}

    When $n=1$ let $P_1$ and $Q_1$ refer to the elements $R_1$ and $Q_1$ in the previous , respectively. 
    
    For our diagrammatic descriptions of unshaded subfactor planar algebras with principal graph of type affine $A$ we will want to have diagrams $P_1$ and $Q_1$ so that either they are dual to each other (rotation by $\pi$ changes a $P_1$ to a $Q_1$ and vice versa) or self-dual (rotation by $\pi$ of $P_1$ is $P_1$ and of $Q_1$ is $Q_1$). We will call the former case \textit{the arrow case} and denote
    \begin{align*}
        P_1=\up \text{ and } Q_1=\down.
    \end{align*} 
    We will call the latter case \textit{the color case} and denote 
    \begin{align*}
        P_1=\red \text{ and } Q_1=\blue.
    \end{align*} 
    
    A consequence of Theorem \ref{thm: A-finite-necessary-arrow-even} is that for the $\til{A}_{2n}$ planar algebra $\mc{P}'$ we cannot have the color case, whereas both cases will exist for the $\til{A}_{2n-1}$ planar algebra, $\mc{P}$. However, we will continue forward assuming the possibility for both cases for both $\mc{P}$ and $\mc{P}'$. Next we prove a collection of useful facts following from these decompositions of $X$. 

     \begin{lem}\label{lem: A-fin-gen-facts}
         For $\mc{P}$ and $\mc{P}'$ we have the following:
         \begin{paracol}{2}
            \begin{enumerate}[label=(\roman*)] 
            \item If $\strand = \up + \down$ then:
                \begin{enumerate}
                    \item $\up \raisebox{0.3cm}{{\scriptsize *}}=\up $ and $\down \raisebox{0.3cm}{{\scriptsize *}}=\down$,
                    \item $\udsaddle = \te{ }\updown$ and $\dusaddle=\te{ }\downup$
                    \item $\udsaddle \cong \dusaddle \cong \te{ }\updown\te{ } \cong\te{ } \downup \te{ } \cong\te{ } \emptyset$
                    \item $\upbubble=1, \downbubble=1,$ and $\bubble =2$
                \end{enumerate}
            \end{enumerate}
            \switchcolumn
            \begin{enumerate}[label=(\roman*)] 
                \setcounter{enumi}{1}
                \item If $\strand = \red + \blue$ then:
                \begin{enumerate}
                    \item $\red\raisebox{0.3cm}{{\scriptsize *}}=\red$ and $\blue \raisebox{0.3cm}{{\scriptsize *}}=\blue$
                    \item $\rsaddle= \te{ }\rtwo$ and $\bsaddle =\te{ }\btwo$
                    \item $\rsaddle\cong\bsaddle\cong \te{ }\rtwo \te{ }\cong \te{ } \btwo \te{ }\cong \emptyset$  
                    \item $\rbubble=1, \bbubble=1,$ and $ \bubble =2$
                \end{enumerate}
            \end{enumerate}
        \end{paracol}
     \end{lem}

     \begin{proof} As $P_1$ and $Q_1$ are minimal projections, $P_1^*=P_1$ and $Q_1^*=Q_1$, giving part (a) in both cases. For part (b), we can essentially follow the inner product computation done in Lemma \ref{lem: Afinshadesaddlerels}. For part (c) we follow a similar idea to Lemma \ref{lem: Afinshadesaddlerels} by defining $f$ to be the upper strand of the leftmost diagram of the isomorphism. Finally, part (d) follows from taking the trace of $X=P_1+Q_1$ and using that the trace of the minimal projections are 1.
    \end{proof}

        From the previous lemma we get that a for any diagram $T$, its adjoint, $T^*$, in the arrow case is vertically flip, reverse the arrows, and star the inside of any boxes, then extend this definition anti-linearly and on diagrams. For the color case, $T^*$ is vertically flip and star the inside of any boxes. 

        The proof of the below lemma is nearly identical to the proofs of Lemma \ref{lem: Afinshadedtensordecomp} and Lemma \ref{lem: AfinshadePnisotoQn} and thus will be omitted. 
\begin{lem}\label{lem: Afin-minproj}
    \begin{enumerate}[label=(\roman*)]
        \item For $\mc{P}$ with $n\geq 2$, using the labelling of the minimal projections in (\ref{eq: arbAfingraph}), we have that:
    \begin{enumerate}
        \item in the arrow case, $P_k\cong P_1^{\otimes k}$, $Q_\ell\cong Q_1^{\otimes \ell}$, and $P_1^{\otimes n} \cong Q_1^{\otimes n}$, and
        \item in the color case, $P_k \cong P_1 \otimes Q_1 \otimes \os{(k-1)}{...} \otimes R$, $Q_\ell \cong Q_1 \otimes P_1 \os{(\ell-1)}{...} \otimes S$ where $R,S$ are $P_1$ or $Q_1$ depending on the parity of $k$ and $\ell$. Further, $P_1 \otimes Q_1 \os{(n-1)}{...} \otimes R \cong Q_1 \otimes P_1 \os{(n-1)}{...} \otimes S$. 
    \end{enumerate}
    \item For $\mc{P}'$, using the labelling of the minimal projections in (\ref{eq: arbAfingrapheven}), we have that in the arrow case, $P_k\cong P_1^{\otimes k}$, $Q_\ell\cong Q_1^{\otimes \ell}$, $P_1^{\otimes (n+1)} \cong Q_1^{\otimes n}$, and $P_1^{\otimes n} \cong Q_1^{\otimes (n+1)}$.
    \end{enumerate}

\end{lem}

The above  tells us the principal graphs for $\til{A}_{2n-1}$ and $\til{A}_{2n}$. For $\til{A}_{2n-1}$ in the arrow case we have principal graph 
    \begin{equation}\label{eq: arrowprincipalgraph}
        \oAfingraph
    \end{equation}
    when $n\geq 2$ and 
    \begin{equation}\label{eq: arrowprincipalgraphone}
        \oAfingraphone
    \end{equation}
    when $n=1$. In the color case the principal graph will be
    \begin{equation}\label{eq: colorprincipalgraph}
        \rbAfingraph
    \end{equation}
    when $n\geq 2$ and
    \begin{equation}\label{eq: colorprincipalgraphone}
        \rbAfingraphone
    \end{equation}
    when $n=1$. For $\til{A}_{2n}$ the principal graph will be
    \begin{equation}\label{eq: arrowprincipalgrapheven}
        \oAfingrapheven
    \end{equation}
    for all $n\in \mbb{N}$. 

We are now ready to prove the necessary relations for any unshaded subfactor planar algebra with principal graph $\til{A}_{2n-1}$. Later in  \ref{thm: A-fin-arrow-generate} we will prove that the elements listed in Definitions \ref{Afinarrowgenerators} and \ref{Afincolorgenerators}  are indeed generators of their respective planar algebra, $\mc{P}$. For now, we want to find elements and relations for the planar algebras. 

\begin{defn}[Elements in Affine $A$ Unshaded Subfactor Planar Algebras: The Arrow Case] \label{Afinarrowgenerators}
We will show in the following  that when $\strand \cong \up \oplus \down$, the $\til{A}_{2n-1}$ unshaded subfactor planar algebras have elements
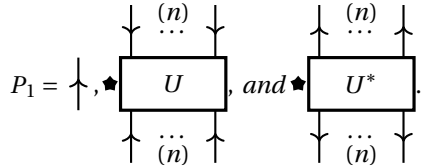
\begin{equation*}
    P_1= \up , \U , \text{ and } \Uad. 
\end{equation*}    
\end{defn}

\begin{defn}[The Relations of Affine $A$ Finite Unshaded Subfactor Planar Algebras: The Arrow Case] \label{Afinarrowrelations} We will show in the following  that when $\strand \cong \up \oplus \down$ the relations of the $\til{A}_{2n-1}$ unshaded subfactor planar algebras are:
 \begin{enumerate}[label=(\roman*)]
            \item (bubble relations) $\upbubble=\downbubble =1$ 
            \item (strand decomposition) $\strand = \up + \down$ 
            \item (orientation disagreement) $\ud=0$ and $\te{ }\du=0$ 
            \item (saddle relations) $\udsaddle = \te{ }\updown$ and $\dusaddle=\te{ }\downup$ 
            \item  ($U$ unitary relations) $UU^*=\UUad=\ndown \quad$ and
            $\quad U^*U=\UadU=\nup$ 
            \item (click relations) For some $2n$th root of unity, $\omega_{2n}$: \\$\mc{F}\left(U\right)=\Uclick=\omega_{2n}\U \quad$ and 
            $\quad \mc{F}(U^*)=\te{  }\Uadclick=\omega_{2n}\Uad$.
        \end{enumerate} 
\end{defn}

\begin{defn}[Elements in Affine $A$ Unshaded Subfactor Planar Algebras: The Color Case] \label{Afincolorgenerators}
We will show in the following  that when $\strand \cong \red \oplus \blue$, the $\til{A}_{2n-1}$ unshaded subfactor planar algebras have elements
\begin{equation*}
    P_1 = \red , Q_1=\blue. \V , \text{ and } \Vad 
\end{equation*}    
where a purple strand indicates using the appropriate color when alternating between red and blue strands.
\end{defn}

\begin{defn}[The Relations of Affine $A$ Finite Unshaded Subfactor Planar Algebras: The Color Case] \label{Afincolorrelations} We will show in the following  that when $\strand \cong \red \oplus \blue$ the relations of the $\til{A}_{2n-1}$ unshaded subfactor planar algebras are given below. Note that we use purple strands to indicate the strand is either red or blue, depending on parity in a way that should be clear from context.
     \begin{enumerate}[label=(\roman*)]
            \item  (bubble relations) $\rbubble=\bbubble=1$ 
            \item (strand decomposition) $\strand = \red + \blue$ 
            \item (color disagreement) $\br=0$ 
            \item (saddle relations) $\rsaddle= \te{ }\rtwo$ and $\bsaddle =\te{ }\btwo$ 
            \item ($V$ unitary relations) $ VV^*= \VVad=\nbr\quad $ and $\quad V^*V=\VadV=\nrb$      
            \item (click relation) For some $n$th root of unity, $\tau_{n}$:
            \begin{align*}
                \mc{F}(V)=\Vclickr=\tau_n \te{ }\Vad= \tau_n \te{ } \Vclickb=\tau_n\mc{F}^{-1}(V).
            \end{align*}
            \end{enumerate}
\end{defn}

We show in the following  that the elements and relations defined above are necessarily elements and relations in an $\til{A}_{2n-1}$ unshaded subfactor planar algebra.

\begin{thm}\label{thm: A-finite-necessary-arrow} (Necessary relations for $\til{A}_{2n-1}$)
        Fix $n \in \mbb{N}$. If $\mc{P}$ is an unshaded subfactor planar algebra with principal graph $\til{A}_{2n-1}$, then one of two cases hold, which we call the arrow case and the color case. 
        \begin{enumerate}
            \item In the arrow case, $\mc{P}$ has elements given in Definition \ref{Afinarrowgenerators} with relations given in Definition \ref{Afinarrowrelations}. Further, we get the equivalence classes of minimal projections are for $n\geq 2$ is (\ref{eq: arrowprincipalgraph}) and for $n=1$ is (\ref{eq: arrowprincipalgraphone}).
            \item In the color case, $\mc{P}$ has elements given in Definition \ref{Afincolorgenerators} with relations given in Definition \ref{Afincolorrelations}. Further, we get the equivalence classes of minimal projections are for $n\geq 2$, (\ref{eq: colorprincipalgraph}), and for $n=1$, is (\ref{eq: colorprincipalgraphone}).
        \end{enumerate}
                
    \end{thm}
      \begin{proof} By Lemmas \ref{lem: A-fin-strand-decomp} and \ref{lem: Aunshadenonep1q1}, we know there are two cases of decomposing $X$: either
      \begin{align*}
          \strand\cong \up \oplus \down \te{ or } \strand \cong \red \oplus \blue.
      \end{align*}
     The bubble relations follows from all minimal projections having trace 1. The orientation/color disagreement is true because $P_1$ and $Q_1$ are nonisomorphic minimal projections so $P_1Q_1=Q_1P_1=0$. The strand decompositions follow from the equality $X=P_1+Q_1$, and saddle relations were proven in Lemma \ref{lem: A-fin-gen-facts}. 
     
      For convenience, let us distinguish the minimal projections in the arrow and color case, by reserving $P_1$ and $Q_1$ to be the minimal projections in the 1-box space for the arrow case and $\til{P}_1$, $\til{Q}_1$ to be the minimal projections in the 1-box space for the color case. Before proving the final two relations, notice that Lemma \ref{lem: Afin-minproj} gives us the equivalence classes of minimal projections shown in each of the principal graphs shown in the statement of this . By the same lemma, we have that in the arrow case
      \begin{align*}
          \up \os{(n)}{...}\up \cong \down \os{(n)}{...}\down,
      \end{align*}
      and by the definition of isomorphic elements, there must exist morphisms, say $U \in \te{Hom}\left(P_1^{\otimes n}, Q_1^{\otimes n}\right)$ and $U^*\in \te{Hom} \left(Q_1^{\otimes n}, P_1^{\otimes n} \right)$ such that $UU^*=Q_1^{\otimes n}$ and $U^*U=P_1^{\otimes n}$. Therefore there exists elements
      \begin{align*}
          U=\U \te{ and } U^*=\Uad,
      \end{align*} 
      in $\mc{P}$ where $UU^*=Q_1^{\otimes n}$ and $U^*U=P_1^{\otimes n}$, which gives our U unitary relations.
      
        For the final relation in the arrow case, notice that $\mc{F}(U) \in \te{Hom}\left(P_1^{\otimes n}, Q_1^{\otimes n}\right)$. However, as $P_1^{\otimes n} \cong Q_1^{\otimes n}$ and these elements are minimal projections, we get that $\te{dim}\left(\te{Hom}\left(P_1^{\otimes n}, Q_1^{\otimes n}\right)\right)=1$. Thus, $\mc{F}(U)=\lambda U$ for some $\lambda \in \mbb{C}$. Applying $\mc{F}$ $2n$ times we must circle back to $U$, so $\mc{F}^{2n}\left(U\right)=\lambda^{2n} U = U$. So $\lambda = \omega_{2n}$, some $2n$th root of unity. Noticing that $\mc{F}^*=\mc{F}^{-1}$ gives that taking the star of $\mc{F}(U)=\lambda U$ with $\lambda=\omega_{2n}$ is $\mc{F}^{-1}(U^*)=\omega_{2n}^{-1}U^*$. Applying $\mc{F}$ to both sides gives $U^*=\omega_{2n}^{-1} \mc{F}(U^*)$, which is the second click relation. 
     
        In the color case, there must exist morphisms, say $W \in \te{Hom}\left(\til{Q}_1 \otimes \til{P}_1 \os{(n-1)}{...}\otimes S, \til{P}_1 \otimes \til{Q}_1 \os{(n-1)}{...}\otimes R \right)$ along with $W^*\in \te{Hom}\left(\til{P}_1 \otimes \til{Q}_1 \os{(n-1)}{...}\otimes R, \til{Q}_1 \otimes \til{P}_1 \os{(n-1)}{...}\otimes S\right)$, where $R,S \in \{\til{P}_1,\til{Q}_1\}$, $R=P_1$ only if $n$ is odd, and $R\neq S$ such that $W^*W= \til{Q}_1 \otimes \til{P}_1 \os{(n-1)}{...}\otimes S$ and $WW^*=\til{P}_1 \otimes \til{Q}_1 \os{(n-1)}{...}\otimes R$. Since $\te{Hom}\left(\til{Q}_1 \otimes \til{P}_1 \os{(n-1)}{...}\otimes S, \til{P}_1 \otimes \til{Q}_1 \os{(n-1)}{...}\otimes R\right)$ is one-dimensional, there exists $\alpha,\beta  \in \mathbb{C}$ such that $\mc{F}(W)=\alpha W^*$ and $\mc{F}(W^*)=\beta W$. Taking $\mc{F}^{-1}$ of both sides and moving the scalar over, we then get $\mc{F}^{-1}(W^*)=\alpha^{-1}W$ and $\mc{F}^{-1}(W)=\beta^{-1} W^*$. On the other hand, taking the adjoint of both sides of the equalities: $\mc{F}(W)=\alpha W^*$ and $\mc{F}(W^*)=\beta W$ and noticing $\mc{F}(W)^*=\mc{F}^{-1}(W^*)$ and $\mc{F}(W^*)^*=\mc{F}^{-1}(W)$ gives that $\mc{F}^{-1}(W^*)=\overline{\alpha} W$ and $\mc{F}^{-1}(W)=\ol{\beta}W^*$. Therefore, $\alpha^{-1}=\ol{\alpha}$ and $\beta^{-1}=\ol{\beta}$. 

        We can further show that $\alpha\beta=\tau_n$, some $n$th root of unity. Notice that $\mc{F}^2(W)=(\alpha\beta) W$ and applying $\mc{F}^2$ $n$ times we must circle back to $W$, so $\mc{F}^{2n}(W)=(\alpha\beta)^nW=W$ and $\alpha\beta$ must be an $n$th root of unity. 

        Define the morphism $V$, shown below, by $V=\sqrt{\alpha} W^*$, making $V^*=\ol{\sqrt{\alpha}}W=\sqrt{\alpha^{-1}}W$, which is also shown below
        \begin{align*}
            V=\V \te{ and }
            V^*=\Vad.
        \end{align*}
     Thus $VV^*= \til{Q}_1 \otimes \til{P}_1 \os{(n-1)}{...}\otimes S$ and $V^*V=\til{P}_1 \otimes \til{Q}_1 \os{(n-1)}{...}\otimes R$, giving relation (v). Further, for the $n$th root of unity $\tau_n=\alpha\beta$, 
        \begin{align*}
            \mc{F}(V)=\sqrt{\alpha}\beta W=\alpha \beta \sqrt{\alpha^{-1}}W=\alpha \beta V^* = \tau_n V^*, \text{ and }
            \tau_n \mc{F}^{-1}(V)=\beta \sqrt{\alpha}W=\beta \alpha \sqrt{\alpha^{-1}}W=\alpha \beta V^*
        \end{align*}
        proving relation (vi). Thus we have shown that $\mc{P}$ must have all of the given elements and relations. \end{proof}

        We will show in  \ref{thm: A-fin-arrow-generate} that the elements $P_1$, $U$, and $U^*$ from the preceding  are indeed generators of $\mc{P}$. Next we tackle the $\til{A}_{2n}$. The following  is nearly identical to the arrow case of  \ref{thm: A-finite-necessary-arrow} so will be omitted. 

\begin{thm}\label{thm: A-finite-necessary-arrow-even} (Necessary relations for $\til{A}_{2n}$). Fix $n\in \mbb{N}$. If $\mc{P}'$ is an unshaded subfactor planar algebra with principal graph $\til{A}_{2n}$, then $\mc{P}'$ has elements given in Definition \ref{Afinarrowgenerators} and relations given in Definition \ref{Afinarrowrelations} except with $U$ having $n+1$ strands on the bottom and $U^*$ having $n+1$ strands on top. Further we get the equivalence classes of minimal projections are given in (\ref{eq: arrowprincipalgrapheven}).
\end{thm}

From this, the following remark is immediate. 

\begin{rem}\label{thm: Afinevennocolorcase}
    In $\mc{P}'$, we cannot have that $P_1$ and $Q_1$ are self-dual.
\end{rem}

    \begin{rem}
        Just as in the shaded case, we can always oriented a strand in its entirety using strand decomposition and orientation disagreement.
    \end{rem}
    
   We will show in  \ref{thm: A-fin-arrow-generate} that elements $P_1$, $U$, and $U^*$ in the previous  generate $\mc{P}'$.
\section{Sufficient Relations for Affine \texorpdfstring{$A$}{A} Finite Unshaded Planar Algebras: The Arrow Case}

In this section we start with generators and relations for a planar algebra and prove that it generates an unshaded subfactor planar algebra and further this planar algebra has an affine $A$ principal graph. Depending on the choice of root of unity and the number of strands around the generator boxes, these planar algebras could either have $\til{A}_{2n-1}$ or $\til{A}_{2n}$ Dynkin diagram as its principal graph.

\begin{thm}\label{thm: arrow-A-finite-suff} (Sufficient relations for $\til{A}_{2n-1}$ and $\til{A}_{2n}$: Arrow case) Fix $n \in \mathbb{N}$ and a $2n$th root of unity $\omega_{2n}$ (respectively, a $2n+1$ root of unity $\omega_{2n+1}$ for $\til{A}_{2n}$). Let $\mc{P}$ be the planar algebra generated by $P_1$, $U$, $U^*$ from Definition \ref{Afinarrowgenerators} with relations from Definition \ref{Afinarrowrelations}. Let $\mc{P}'$ be defined the same as $\mc{P}$ except with $U$ having $n+1$ strands on the bottom and $U^*$ having $n+1$ strands on top. Define $P_1$ to be self-adjoint, $U^*$ to be the adjoint of $U$, and extend * anti-linearly and on diagrams. Then $\mc{P}$ (respectively, $\mc{P}'$) is an unshaded subfactor planar algebra whose principal graph is the $\til{A}_{2n-1}$ (respectively, $\til{A}_{2n}$) Dynkin diagram, given by, for $n\geq 2$ by (\ref{eq: arrowprincipalgraph}) and when $n=1$ by (\ref{eq: arrowprincipalgraphone}) (respectively, (\ref{eq: arrowprincipalgrapheven}) for $\til{A}_{2n}$.)
\end{thm}

Fix $n\in \mbb{N}$, a $2n$th root of unity $\omega_{2n}$, and a $2n+1$th root of unity $\omega_{2n+1}$. Sometimes we will use $\omega$ to mean either $\omega_{2n}$ or $\omega_{2n+1}$ when the context is clear. Throughout the rest of this section, $\mc{P}$ and $\mc{P}'$ will refer to the planar algebra defined through generators and relations given in the previous theorem . We will denote the $k$-th box spaces of $\mc{P}$ and $\mc{P}'$ as $\mc{P}_k$ and $\mc{P}'_k$ respectively. While dealing with the arrow case, define
    \begin{align*}
        Q_1 = \down.
    \end{align*}

Before proving Theorem \ref{thm: arrow-A-finite-suff}, we will prove some claims about $\mc{P}$ and $\mc{P}'$. Most of the proofs we will just prove for $\mc{P}$. Adapting the proofs for $\mc{P}'$ mostly involves adjusting some $n$ to be $n+1$ but are otherwise identical. If there is more work than that to adjust, we will include it. 

As $*$ is defined to be extended anti-linearly and on diagrams, it is clear that for any diagram $T$ in the planar algebra, $T^*$ is vertically flipping, reversing the arrows, and switching $U$ and $U^*$. It is left to check that the * is well-defined by checking it respects the defining relations of the planar algebra. 
    
    \begin{lem}
        The adjoint is well-defined.
    \end{lem}
    
    \begin{proof} From the preceding paragraph, $*$ respects the bubble relations, strand decomposition, orientation disagreement, saddle relations, and the $U$ unitary relation clearly. For the click relations, notice that taking the adjoint of the first relation gives $\mc{F}(U)^*=\mc{F}^{-1}(U^*)=\omega^{-1} U^*$. Then taking $\mc{F}$ of second equality gives $U^*=\omega^{-1}\mc{F}(U)$, which gives the second click relation. Similarly, taking the adjoint of the second click relation will give the first click relation. \end{proof}

 As done for shaded planar algebras, we start by showing that the 0-box spaces are one-dimensional. Again, we use a standard fact from the theory of van Kampen diagrams which we state in the following remark.

 \begin{rem}
     For any diagram in $\mc{P}$ (respectively $\mc{P}'$) there exists a unique function from the regions of the diagram to the cyclic group $\mathbb{Z}_{2n}=\langle u | u^{2n}=1\rangle$ (respectively $\mathbb{Z}_{2n+1}$) such that the region on the left-most boundary next to the star is assigned 1 and neighboring regions are assigened elements from the group in the following way. If on the left-side from the point of view of the arrow the region is labelled $x$, then the region on the right is labelled $xu$.
 \end{rem}

 We first show $\mc{P}_0$ is at least one-dimensional by defining a surjection from $\mc{P}_0$ to $\mbb{C}$. 
 
 \begin{tcolorbox}[breakable, pad at break*=0mm]
        Define a function $f_U: \mc{P}_0 \ra \mbb{C}$ by the following algorithm. By strand decomposition and orientation disagreement it suffices to define $f_U$ on diagrams where every strand is oriented in on direction in its entirety and then extend $f_U$ linearly. Let $D\in \mc{P}_0$ be such a diagram.
      
        \begin{enumerate}
            
            \item If there are no $U$ and $U^*$s in $D$, define $f_U(D)=1$.  
            
            \item If there are $U$ or $U^*$s in $D$, enumerate all of the $U$ and $U^*$, say $U_1,...,U_r$, and $U^*_1,...,U^*_s$. (Note that $r=s$ but we do not need this fact.)
                      
            \item Define a number $\ell_j^J$ based on the following table where $w_j^J$ is the element of $\mbb{Z}_{2n}$ that corresponds to the region where $J_j$'s star lies.

                    \begin{tabular}{l | l}
            When $J=U$ &
            When $J=U^*$,\\
            \hline
            & \\
            $\bullet$ if $w_j^J=u^m$, then $\ell_j^J=-m$, & $\bullet$ if $w_j^J=u^m$, then $\ell_j^J=m$
        \end{tabular}
            
            \item Define $\ell^{D}= \left(\sum_{j=1}^s \ell_j^{U^*}-\sum_{i=1}^r \ell_i^{U}\right) (\te{mod } 2n)$. 
            \item Then define $f_U(D)=\omega_{2n}^{\ell^{D}}$ and extend $f_U$ linearly to define all of $f_U$.
        \end{enumerate} 
    \end{tcolorbox}

    For $\mc{P}'$ use $\omega_{2n+1}$ and take the length mod $2n+1$ instead. To show $f_U$ is well-defined we need to show that $f_U$ does not depend on the path chosen and is invariant under the defining relations of the planar algebra.

\begin{lem}
    $f_U$ is well defined. 
\end{lem}

    \begin{proof} To show the lemma, we must prove $f_U$ respects the defining relations of $\mc{P}$ (or $\mc{P}'$). Without loss of generality, consider $\mc{P}$. Let $D$ and $D'$ be two diagrams in $\mc{P}_0$ that are identical except in a local region. 

    Most of this proof is identical to the proof of Lemma \ref{lem: Afinshade-atleast1dim}. What's left to check is the unitary relations and the click relations. 

    Let $D$ and $D'$ be identical except in a local region $D$ is the middle equality of the $UU^*$ relation (i.e., written with two boxes) and $D'$ is the right-hand side (written with no boxes). The top and bottom boundary data is the same, so any region above or below the the local regions are labelled the same. Therefore the entire diagrams are labelled identically. Further, in the local region of $D$, $U$ and $U^*$'s $\ell$-values cancel, giving that $f_U$ is invariant under this unitary relation. The same argument works for $U^*U$. 

    Lastly, we check that $f_U$ is invariant under the click relations. We just check the first click relation. The other is the same process. Let $D$ have a local region that looks like the first diagram of the first click relation and $D'$ have a region that looks like $U$. It is clear that outside this region both $D$ and $D'$ are labelled identically. 
    
    Suppose the word at the left-side of the local region is $x$. The word associated to the $U$ box in the local region of $D$ is $xu^{-1}$. In $D'$ the word associated to the region is $x$. Thus the $\ell$-value associated to the box of $D$ is one more than of $D'$. Thus, when multiplying $D'$ by $\omega$, we get the same value of $f$ as $D$. 
    
    Hence we have shown that $f_U$ is invariant under the relations of $\mc{P}$. This concludes showing that $f_U$ is well-defined. \end{proof}

\begin{cor}\label{lem: arrow-0-box-at-least}
        $\mc{P}_0$ and $\mc{P}'_0$ are at least 1-dimensional. 
\end{cor}
    
  \begin{proof} We have well-defined functions from $\mc{P}_0$ to $\mathbb{C}$ and $\mc{P}'_0$ to $\mbb{C}$. These functions send $\emptyset$ to 1 and are linear. Therefore, the maps are surjections, which implies that the 0-box spaces are at least one-dimensional. \end{proof}
    
    Now I will show that $\mc{P}_0$ and $\mc{P}'_0$ are at most 1-dimensional, by defining an evaluation algorithm.  

    \begin{tcolorbox}[breakable, pad at break*=0mm]
        Evaluation Algorithm for $\mc{P}_0$:

        By strand decomposition and orientation disagreement it suffices to define the algorithm on diagrams where every strand is oriented in one direction in its entirety then extend linearly. Let $D\in \mc{P}_0$ be such a diagram.
        
        \begin{enumerate}
            \item If $D$ has no $U$ or $U^*$, skip to step 6. Now assume $D$ has some $U$ or $U^*$. Since $D$ is nontrivial, this $U$ or $U^*$ cannot have a cap or cup. The diagram is closed, so the strands of $U$ or $U^*$ cannot go to the boundary of $D$. Due to orientation, the strands of $U$ cannot connect to another $U$ and similarly, the strands of $U^*$ cannot connect to another $U^*$. Thus, if $D$ has a $U$, the strands of $U$ must connect to a $U^*$ and if $D$ has a $U^*$, the strands of $U^*$ must connect to a $U$. Thus we can assume $D$ has both a $U$ and $U^*$ and all the strands of $U$ connect to $U^*$s and the strands of $U^*$ all connect to $U$s. 
            \item Pick any $U$ and $U^*$ connected by at least one strand. We will inductively increase the number of strands connecting $U$ and $U^*$. If $U$ and $U^*$ are connected by only 1 strand, then use the click relation to have $U$ and $U^*$ connected by the top left strand of the left box is connected by the top right strand of the right box. So your diagram is multiplied by some factor of $\omega_{2n}$.
            \item For the strand of $U$ adjacent to the connected strand, isotope the strand so that it lies next to the strand of $U^*$ adjacent to the connected strand. Then use the saddle relation to make $U$ and $U^*$ connected by at least 2 adjacent strands.
            \item Inductively repeat the process until $U$ and $U^*$ are connected by $s < n$ adjacent strands. It is clear that we can then repeat the process of step 4 and get $U$ and $U^*$ connected by $n$ adjacent strands. 
            \item Use the $U$ and $U^*$ relation to reduce $D$ to have one less $U$ and $U^*$. Repeat steps 2-6 until no more $U$ and $U^*$ remain. 
            \item Now $D$ is a closed diagram consisting of only oriented loops multiplied by some scalars. Starting with an innermost loop, repeatedly pop the loops for a factor of 1 using that an oriented loop evaluates to 1. $D$ will then evaluate to a multiple of $\emptyset$ and thus we can evaluate on all of $\mc{P}_0$.
        \end{enumerate}
    \end{tcolorbox}

The evaluation algorithm of $\mc{P}'_0$ is nearly identical, just exchanging $\omega_{2n}$ with $\omega_{2n+1}$ and change appropriate $n$s with $n+1$. With the evaluation algorithm the following lemma is immediate.

\begin{lem}
    $\mc{P}_0$ and $\mc{P}'_0$ are at most one-dimensional. 
\end{lem}

With the well-defined surjection, $f_U$, the following corollary is immediate.
    
\begin{cor}\label{lem: P0-1-d}
        $\mc{P}_0$ and $\mc{P}'_0$ are one-dimensional. 
\end{cor}

\begin{lem}\label{lem: Afinarrow-P1Q1isotoempty}
    In both $\mc{P}$ and $\mc{P}'$, $P_1 \otimes Q_1 \cong \emptyset$ and $Q_1 \otimes P_1 \cong \emptyset$.
\end{lem}

    \begin{proof} We prove the first isomorphism. The other can be found similarly. Define $f=\ocupl$, making $f^*=\ocapr$. Then $ff^*=P_1 \otimes Q_1$ by the saddle relations and $f^*f=1\cdot \emptyset$ by the bubble relations. \end{proof}
    
\begin{lem}\label{lem: AfinarrowPnisotoQn}
    For all $0\leq k \leq n$, 
    \begin{enumerate}[label=(\roman*)]
        \item in $\mc{P}$, $P_1^{\otimes (2n-k)} \cong Q_1^{\otimes k}$, and 
        \item in $\mc{P}'$, $P_1^{\otimes (2n+1-k)} \cong Q_1^{\otimes k}$.
    \end{enumerate}
\end{lem}

    \begin{proof} Both planar algebras will follow the same proof, so we just prove the lemma for $\mc{P}$. Notice that $P_1^{\otimes n} \cong Q_1^{\otimes n}$, as we can let $f=U$ and $f^*=U^*$, and by the defining relations $ff^*\cong Q_1^{\otimes n}$  and $f^*f\cong P_1^{\otimes n}$. Let $k\geq 0$. Tensor $n-k$ copies of $P_1$ on the right of both sides of the  isomorphism $P_1^{\otimes n} \cong Q_1^{\otimes n}$. The left-hand side is $P_1^{\otimes (2n-k)}$ and the right-hand side, using Lemma \ref{lem: Afinarrow-P1Q1isotoempty}, is $Q_1^{\otimes k}$, which gives the desired isomorphism. \end{proof}
  
\begin{lem}\label{lem: Afin-o-principal-graph}
        The principal graph of $\mc{P}$ is (\ref{eq: arrowprincipalgraph}) when $n\geq 2$ and (\ref{eq: arrowprincipalgraphone}) when $n=1$. The principal graph of $\mc{P}'$ is (\ref{eq: arrowprincipalgrapheven}).
\end{lem}

    \begin{proof} Without loss of generality, we prove the result for $\mc{P}$. Define $P_k= P_1^{\otimes k}$ for $1\leq k \leq n$ and $Q_\ell=Q_1^{\otimes \ell}$ for $1\leq \ell \leq n-1$. First I show that for all $1\leq k \leq n$, $\te{dim}(\te{Hom}\left(P_k, P_k\right))=1$ and for all $1\leq \ell \leq n-1$, $\te{dim}\left(\te{Hom}\left(Q_\ell, Q_\ell\right)\right)=1$. For any $P_k$, pick $f \in \te{Hom}(P_k,P_k)$. Diagrammatically we have:

    
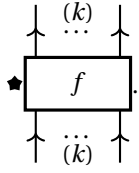
\centerline{
   \begin{tikzpicture}[baseline={([yshift=-.8ex]current bounding box.center)}, decoration={markings, mark=at position 0.6 with {\arrow{<}}}, fixedpt/.style={star, fill, minimum size=0.5mm},scale=0.7]
   \node [fixedpt, scale=0.5] (g) at (-0.2,0.5) {};
  \draw[very thick] (0,0) rectangle (2,1) node[pos=0.5] {$f$};
  \draw[thick, postaction={decorate}] (0.2,2) -- (0.2,1);
  \draw[thick, postaction={decorate}] (1.8,2) -- (1.8, 1);
  \draw[thick, postaction={decorate}](1.8,0) -- (1.8, -1);
  \draw[thick, postaction={decorate}](0.2,0) -- (0.2, -1);
    \node (a) at (0.5, -0.5) {};
    \node (b) at (1.5,-0.5) {};
    \path (a) -- node[auto=false]{\ldots} (b);
    \node (c) at (1, -0.85) {\textup{($k$)}};
    \node (d) at (0.5, 1.5) {};
    \node (e) at (1.5, 1.5) {};
    \path (d) -- node[auto=false]{\ldots} (e);
    \node (f) at (1, 1.85) {\textup{($k$)}};
    \end{tikzpicture}.}

    \noindent Then isotope a string to be adjacent to an oppositely-oriented one and use the saddle relation in the following manner:

    
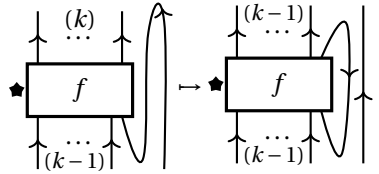
\centerline{
    \begin{tikzpicture}[baseline={([yshift=-.8ex]current bounding box.center)}, decoration={markings, mark=at position 0.6 with {\arrow{<}}}, fixedpt/.style={star, fill, minimum size=0.5mm},scale=0.7]
   \node [fixedpt, scale=0.5] (g) at (-0.2,0.5) {};
  \draw[very thick] (0,0) rectangle (2,1) node[pos=0.5] {$f$};
  \draw[thick, postaction={decorate}] (0.2,2) -- (0.2,1);
  \draw[thick, postaction={decorate}] (1.8,2) -- (1.8, 1);
  \draw[thick, postaction={decorate}](1.6,0) -- (1.6, -1);
  \draw[thick, postaction={decorate}](0.2,0) -- (0.2, -1);
  \draw[thick, postaction={decorate}] plot[smooth,tension=.8]
  coordinates{(1.8,0)  (2.2,-0.8) (2.3,1.6) (2.6,1.8) (2.6,-1)};
    \node (a) at (0.5, -0.5) {};
    \node (b) at (1.5,-0.5) {};
    \path (a) -- node[auto=false]{\ldots} (b);
    \node (c) at (0.9, -0.85) {\textup{\small{($k-1$)}}};
    \node (d) at (0.5, 1.5) {};
    \node (e) at (1.5, 1.5) {};
    \path (d) -- node[auto=false]{\ldots} (e);
    \node (f) at (1, 1.85) {\textup{($k$)}};
\end{tikzpicture}
$\mapsto$
\begin{tikzpicture}[baseline={([yshift=-.8ex]current bounding box.center)}, decoration={markings, mark=at position 0.6 with {\arrow{<}}}, fixedpt/.style={star, fill, minimum size=0.5mm},scale=0.7]
   \node [fixedpt, scale=0.5] (g) at (-0.2,0.5) {};
  \draw[very thick] (0,0) rectangle (2,1) node[pos=0.5] {$f$};
  \draw[thick, postaction={decorate}] (0.2,2) -- (0.2,1);
  \draw[thick, postaction={decorate}] (1.6,2) -- (1.6, 1);
  \draw[thick, postaction={decorate}](1.6,0) -- (1.6, -1);
  \draw[thick, postaction={decorate}](0.2,0) -- (0.2, -1);
  \draw[thick, postaction={decorate}] plot[smooth,tension=.8]
  coordinates{(1.8,0)  (2.2,-0.8) (2.3,1.6) (1.8,1)};
  \draw[thick, postaction={decorate}] (2.6,2) -- (2.6,-1);
    \node (a) at (0.5, -0.5) {};
    \node (b) at (1.5,-0.5) {};
    \path (a) -- node[auto=false]{\ldots} (b);
    \node (c) at (0.9, -0.85) {\textup{\small{($k-1$)}}};
    \node (d) at (0.5, 1.5) {};
    \node (e) at (1.5, 1.5) {};
    \path (d) -- node[auto=false]{\ldots} (e);
    \node (f) at (0.9, 1.85) {\textup{\small{($k-1$)}}};
\end{tikzpicture}\te{ }.}

    \noindent Continuing in this fashion we get $\te{tr}(f) \otimes P_k$. 

    As we can evaluate any closed diagram by Lemma \ref{lem: P0-1-d}, we get that there exists some scalar $\la \in \mbb{C}$, such that this diagram is equivalent to $\lambda P_k$. Thus $\te{dim}(\te{Hom}(P_k,P_k)\leq 1$ for all $1\leq k \leq n$. Since tr$(P_k) = 1$ this gives that $\te{dim}(\te{Hom}(P_k,P_k))=1$ for all $1\leq k \leq n$. A similar argument can be done to show $\te{dim}(\te{Hom}(Q_\ell,Q_\ell))=1$ for all $1\leq \ell \leq n-1$. 
    
    Next we want to show that $P_k$ and $Q_\ell$ are projections for all $1\leq k \leq n$, $1 \leq \ell \leq n-1$. Indeed, multiplication is vertical stacking, so $P_k^2=P_k$ for all $1\leq k \leq n$ and $Q_\ell^2=Q_\ell$ for all $1\leq \ell \leq n-1$ is clear. Recall that we define $P_1^*=P_1$ which gives that $Q_1^*=Q_1$. We then obtain that $P_k^*=P_k$ and $Q_\ell^*=Q_\ell$ for all $1\leq k \leq n$ and $1 \leq \ell \leq n-1$ and these elements are all projections. Using the previous paragraph, these projections will all be minimal. 
    
    Now we show that these projections are nonisomorphic by showing there are no nontrivial morphism between any of two distinct minimal projections we have defined. We break this into cases dependent on the orientations of the distinct projections. 
    
    \ind \textit{Case 1:} Suppose $P_k$ and $Q_\ell$ ($1\leq k\leq n$, $1\leq \ell \leq n-1$) are two projections for distinct vertices and let $X \in \te{Hom}(Q_\ell, P_k)$. Use the strand decomposition and orientation disagreement to rewrite $X$ as a summand of morphisms from $P_k$ to $Q_\ell$ with all oriented strands and without loss of generality, no closed components. Let $X'$ be some such summand. Diagrammatrically we have: 
    \begin{equation}\label{eq: Xmorph}
        \morphPxQ.
    \end{equation}
     For $X'$ suppose there are $k_1$ $U$s and $k_2$ $U^*$s. On the boundary of a $U$-box there are $2n$ sinks and on the boundary of a $U^*$-box there are $2n$ sources. Let $m$ be the number of strands in $X'$, which means there are $m$ total sinks either on the boundary of $X'$ or inside $X'$ (and thus on a $U$-box) and $m$ total sources (which must happen on a $U^*$-box). As $P_k$ has $k$ sinks on the boundary of $X'$ and $Q_\ell$ has $\ell$ sinks on the boundary of $X'$, $m=k+\ell+2k_1n=2k_2n$. Therefore, $(k+\ell) \equiv 0 (\te{mod } 2n)$. However, $1\leq k \leq n$ and $1\leq \ell \leq n-1$, which gives there is no solution and thus no nonzero $X'$. So $\te{dim}(\te{Hom}(Q_\ell, P_k))=\te{dim}(\te{Hom}(P_k, Q_\ell))=0$ for all $1 \leq k \leq n$ and $1\leq \ell \leq n-1$. 
     
    \textit{Case 2:} Suppose $P_k$ and $P_\ell$ are two projections for distinct vertices and let $1 \leq \ell < k \leq n$. Let $X \in \te{Hom}(P_\ell, P_k)$. Again, use the strand decomposition and that the 0-box space is 1-dimensional to rewrite $X$ as a summand of morphisms from $P_k$ to $P_\ell$ with all oriented strands and no closed components. Let $X'$ be some such summand. Diagrammatically we have the same as (\ref{eq: Xmorph}) except the orientation of the arrows at the top are reversed. 
    
    From $P_k$ there are $k$ sinks on the boundary of $X'$ and from $P_\ell$ there are $\ell$ sources on the boundary of $X'$. Let $m$ be the number of strands in $X'$, which is also the number of sources and sinks. Let $X'$ have $k_1$ $U$-boxes and $k_2$ $U^*$-boxes. Then $m=k+2k_1n=\ell+2k_2n$, so $k-\ell \equiv 0 (\te{ mod} 2n)$. As $1\leq \ell < k \leq n$, we again get that there is no solution and no nonzero $X'$. Therefore, $\te{dim}(\te{Hom}(P_k, P_\ell))=0$ for distinct $k$ and $\ell$ from 1 to $n$. 

    The next case is to show $\te{dim}(\te{Hom}(Q_k, Q_\ell))=0$ for distinct $k$ and $\ell$ from 1 to $n-1$. However, this argument is nearly identical to Case 2, and thus is omitted. 
    
    The final case is to show $\te{dim}(\te{Hom}(P,\emptyset))=0$ for all $P \in \{P_k, Q_\ell: 1 \leq k \leq n, 1 \leq \ell \leq n-1\}$. Without loss of generality suppose $P=P_k$ for some fixed $k\in \{1,...,n\}$ and let $f \in \te{Hom}(\emptyset, P_k)$. Suppose there are $m$ strands in $f$, $k_1$ $U$-boxes, and $k_2$ $U^*$-boxes. Then $m$ is the number of sinks in $f$ which is $2nk_1$ and it also equals the number of sources in $f$ which is $2nk_2+k$. Therefore $k \equiv 0 (\te{ mod} 2n)$. There is no such $k$ such that this is true, and so $\te{dim}(\te{Hom}(P_k,\emptyset))=0$ as desired. 
    
    Therefore, $\te{dim}(\te{Hom}(P,Q))=0$ for all distinct minimal projections $P,Q \in \{P_k, Q_\ell: 1 \leq k \leq n, 1 \leq \ell \leq n-1\}$. Thus all our minimal projections defined are nonisomorphic. 

    Finally, what's left to show the tensor decomposition. The first relation $X\cong \emptyset \otimes X \cong P_1 \oplus Q_1$ is given by the strand decomposition. Fix $k\in \mathbb{N}$ where $0<k<n-1$. Then using Lemma \ref{lem: Afinarrow-P1Q1isotoempty},
    \begin{align*}
        P_k \otimes X = P_1^{\otimes k} \otimes X \cong P_1^{\otimes k} \otimes \left(P_1 \oplus Q_1 \right) \cong \left(P_1^{\otimes (k+1)}\right) \oplus \left(P_1^{\otimes k} \otimes Q_1\right) \cong \left(P_1^{\otimes (k+1)}\right) \oplus \left(P_1^{\otimes (k-1)}\right) = P_{k+1} \oplus P_{k-1}
    \end{align*}
    Similarly, $Q_\ell \otimes X \cong Q_{\ell+1}\oplus Q_{\ell-1}$. 
    
    For the last vertex, using Lemma \ref{lem: AfinarrowPnisotoQn}, 
    \begin{align*}
        P_n \otimes X = P_1^{\otimes n} \otimes X &\cong P_1^{\otimes n} \otimes \left(P_1 \oplus Q_1\right) \cong \left(P_1^{\otimes (n+1)} \right) \oplus \left(P_1^{\otimes n} \otimes Q_1\right)
    \cong \left(Q_1^{\otimes n} \otimes P_1 \right) \oplus \left(P_1^{\otimes (n-1)}\right) \cong \left(Q_1^{\otimes (n-1)} \right) \oplus \left(P_1^{\otimes (n-1)}\right)
    \end{align*}
    which equals $P_{n-1} \oplus Q_{n-1}$. So $\til{A}_{2n-1}$ indeed is the principal graph representing the tensor decomposition. \end{proof}

Next, we need to show that our generators and relations are sufficient to ensure sphericality. However this proof is essentially done in Lemma \ref{lem: Afinshade-spherical} so will be omitted.

\begin{lem} \label{lem:A-fin-o-spherical}
    $\mc{P}$ and $\mc{P}'$ are spherical.
\end{lem}

To conclude showing that $\mc{P}$ and $\mc{P}'$ are $\til{A}_{2n-1}$ unshaded subfactor planar algebra we have one final requirement; we need to show they are positive definite.

\begin{lem} \label{lem: A-fin-o-unitary}
    $\mc{P}$ and $\mc{P}'$ are positive definite.
\end{lem}

\begin{proof} As done in Lemma \ref{lem: Afinshade-basis} we use the explicit positive definite basis given in \cite{MPS10}. The only thing left to verify is that he minimal projections have positive trace. It is easy to see that all minimal projections have trace 1, so both planar algebras are positive definite. \end{proof}

\begin{proof}[Proof of Theorem \ref{thm: arrow-A-finite-suff}] All of the work has already been done. Fix an $n\in \mbb{N}$. From Lemma \ref{lem: P0-1-d}, we get that the space of closed diagrams is 1-dimensional. Lemma \ref{lem:A-fin-o-spherical} gives that $\mc{P}$ is spherical and Lemma \ref{lem: Afin-o-principal-graph} gives that the principal graph of $\mc{P}$ is the $\til{A}_{2n-1}$ Dynkin diagram. Finally, we get that $\mc{P}$ is positive definite from Lemma \ref{lem: A-fin-o-unitary}, which completes the proof. \end{proof}

\begin{thm} \label{thm: A-fin-arrow-generate}
    The elements $P_1$, $U$, and $U^*$ from Definition \ref{Afinarrowgenerators} generate $\mc{P}$ in the arrow case of Theorem \ref{thm: A-finite-necessary-arrow}. The elements $P_1$, $U$, and $U^*$ from Definition \ref{Afinarrowgenerators} except with $U$ having $n+1$ strands on the bottom and $U^*$ have $n+1$ strands on top generate $\mc{P}'$ from Theorem \ref{thm: A-finite-necessary-arrow-even}. 
\end{thm}

\begin{proof} Nearly identical to the proof of Theorem \ref{thm: Afinshade-generators}. \end{proof}


\section{Sufficient Relations for Affine \texorpdfstring{$A$}{A} Finite Unshaded Planar Algebras: The Color Case}

As done with the arrow case, we can find that the color relations given in Theorem \ref{thm: A-finite-necessary-arrow} also are sufficient for defining the $\til{A}_{2n-1}$ unshaded subfactor planar algebra. Recall that there is no $\til{A}_{2n}$ unshaded subfactor planar algebra. 

\begin{thm} \label{thm: A-fin-c-sufficient} (Sufficient relations for $\til{A}_{2n-1}$: Color case) Fix $n \in \mbb{N}$ and an $n$th root of unity $\tau_n$. Let $\mc{P}$ be the planar algebra generated by $P_1,Q_1, V,$ and $V^*$ from Definition \ref{Afincolorgenerators} with relations from Definition \ref{Afincolorrelations}. Define $P_1$ and $Q_1$ to be self-adjoint, $V^*$ to the be adjoint of $V$ and extend * anti-linearly and on diagrams. Then $\mc{P}$ is an unshaded subfactor planar algebra whose principal graph is given by (\ref{eq: colorprincipalgraph}) when $n\geq 2$ and (\ref{eq: colorprincipalgraphone}) when $n=1$.
\end{thm}

Fix $n\in \mbb{N}$ and an $n$th root of unity, $\tau_n$. Throughout the rest of this section, $\mc{P}$ will denote the planar algebra defined by generators and relations in the above theorem. Denote the $k$-th box space as $\mc{P}_k$. Our strategy is similar as in the arrow case. Before proving Theorem \ref{thm: A-fin-c-sufficient}, we will prove some claims about $\mc{P}$. 

As the adjoint is defined to be extended anti-linearly and on diagrams, it is clear that for any diagram $T\in \mc{P}$ that for any diagram $T$, $T^*$ is obtained by vertically flipping and switching $V$ and $V^*$. 

\begin{lem}
    The adjoint is well-defined.
\end{lem}

    \begin{proof} From the preceding paragraph, * clearly respects the bubble relations, strand decomposition, color disagreement, saddle relations,  and the $V$ unitary relations. 

    Before proving that the adjoint respect the click relation we show other equalities that must hold from the click relation. For the click relation: $\mc{F}(V)=\tau_n V^*=\tau_n \mc{F}^{-1}(V)$, taking $\mc{F}^{-1}$ of the first equality gives $V=\tau_n \mc{F}^{-1}(V^*)$ and taking $\mc{F}$ of the second equality gives $\tau_n \mc{F}(V^*) = \tau_n V$. Then we can see that taking the adjoint of the click relation gives $\mc{F}^{-1}(V^*)=\tau_n^{-1}V=\tau_n^{-1}\mc{F}(V^*)$
    which is indeed true in the planar algebra. Thus the adjoint is well-defined. \end{proof} 
    
    To tackle showing that $\te{dim}(\mc{P}_0)=1$, we first show $\mc{P}_0$ is at least one-dimensional by defining a surjection from $\mc{P}_0$ into $\mbb{C}$ then show it is at most one-dimensional by giving an evaluation algorithm. Fortunately, our map is nearly identical to the function $f:\mc{P}_0\to \mbb{C}$ given in section 
    \ref{Afinsufficientshaded}. Recall that $f$ never took into consideration the shading of the diagrams. With this in mind, we define $f_V$. 
    
\begin{tcolorbox}[breakable, pad at break*=0mm]
Define $f_V:\mc{P}_0 \ra \mbb{C}$ by plugging $D\in \mc{P}_0$ into $f$ but recording red strands as $b$ and blue strands as $r$.
\end{tcolorbox}

\begin{lem}
    $f_V$ is a well-defined function. 
\end{lem}

    \begin{proof} We need to verify that $f_V$ is invariant under the relations of $\mc{P}$. The bubble relations, color disagreement, saddle relations, and unitary relations follow the same reasoning as Lemma \ref{lem: Afinshade-atleast1dim}. 

    Finally, we check that $f_V$ respects the click relation. Let $D$, $D'$, and $D''$ be identical closed diagrams except locally $D$ has a $V^*$ where $D'$ has a $\mc{F}^{-1}(V)$ and $D''$ has a $\mc{F}(V)$. It is clear that the regions outside of this local neighborhood will all have the same labellings.

    Suppose the group element on the left-hand side of the neighborhood is $x$. We break this up into cases.

    \textit{Case 1}: Suppose $x=(rb)^m$ for $m \in \mbb{Z}$. Then the corresponding $\ell$-value for the $V^*$ in $D$ will be $-m$. For $D'$, the region of $V$'s star will be the group element $xr=r(br)^m$, giving an $\ell$-value of $-m$. In $D''$, the region will associated to $V$ will be labelled $r(br)^{m-1}$ and thus the $\ell$-value of $V$ will be of $-(m-1)$.

    \textit{Case 2}: Suppose $x=(br)^m$. Then the corresponding $\ell$-value for the $V^*$ in $D$ will be $m$. In $D'$ the $V$-box's star's region will now be labelled $b(rb)^{m-1}$. So the $V$-box will have $\ell$-value of $m$. Now consider the $V$-box in $D''$. The region associated to the $V$'s star will be labelled $b(rb)^m$, so the corresponding $\ell$-value will be $m+1$. 

    \textit{Case 3}: Suppose $x=r(br)^m$. Then the corresponding $\ell$-value for $D$ will be $m$. In $D'$ the $V$-box's star's region will be labelled $(rb)^m$, giving an $\ell$-value of $m$. For $D''$, the $V$-box's star will be labelled $(rb)^{m+1}$, giving an $\ell$-value of $m+1$. 

    \textit{Case 4}: Suppose $x=b(rb)^m$. For $D$ the corresponding $\ell$-value for $V^*$ will be $-(m+1)$. In $D'$ the region is now labelled $(br)^{m+1}$, giving an $\ell$-value of $-(m+1)$.  Now for $D''$ the region for $V$'s star will be labelled $(br)^m$, which will give that $V$-box an $\ell$-value of $-m$. 

    This concludes showing that $f_V$ is well-defined. \end{proof}

\begin{lem} \label{lem: A-fin-c-at-least-1dim}
    $\mc{P}_0$ is at least one-dimensional.
\end{lem}

    \begin{proof} By the preceding lemma, $f_V$ is well-defined. Further, $f_V(\emptyset)=1$ and is a linear map, $f_V$ is surjective. So $\mc{P}_0$ is at least 1-dimensional. \end{proof}

    To show that $\mc{P}_0$ is at most one-dimensional, we can define an evaluation algorithm. Again, noticing that the evaluation algorithm in the shaded planar algebra case did not depend on shading, we can just follow the steps of the evaluation algorithm in section \ref{Afinsufficientshaded}. That is:
    
   \begin{tcolorbox}[breakable, pad at break*=0mm]
        The Evaluation Algorithm for $\mc{P}_0$:
        
        Follow the steps of the evaluation algorithm of the shaded planar algebra in section \ref{Afinsufficientshaded}, ignoring any reference to $U$ or $U^*$-boxes and shading. 
    \end{tcolorbox}

\begin{lem} \label{lem: A-fin-c-1d}
    $\mc{P}_0$ is one-dimensional.
\end{lem}

    \begin{proof} From Lemma \ref{lem: A-fin-c-at-least-1dim}, we know $\te{dim}(\mc{P}_0)$ is at least one-dimensional. As we have an evaluation algorithm, we get that every element $D \in \mc{P}_0$ can be rewritten as a scalar multiple of $\emptyset$, therefore, $\mc{P}_0$ is at most one-dimensional, giving the result. \end{proof}

    We want to then show that the $\til{A}_{2n-1}$ Dynkin diagram is the correct principal graph for this planar algebra. However this result follows from a similar process to Lemmas \ref{lem: Afinshadesaddlerels} and \ref{lem: Afinshade-suff-pgraph} which gives:

\begin{lem} \label{lem: A-fin-c-graph}
    The principal graph of $\mc{P}$ is given by the Dynkin diagrams in the right-hand column of Theorem \ref{thm: A-finite-necessary-arrow}. 
\end{lem}

Following the proof of Lemma \ref{lem: Afinshade-spherical}, we get

\begin{lem} \label{lem: A-fin-c-spherical}
    $\mc{P}$ is spherical.
\end{lem}

\begin{lem} \label{lem: A-fin-c-posdef}
    $\mc{P}$ is positive definite.
\end{lem}

\begin{proof} This is nearly the same as the proof as Lemma \ref{lem: A-fin-c-posdef}. The only thing left to check is that the minimal projections have positive trace. However, we know every minimal projection has trace equal to 1, so $\mc{P}$ is indeed positive definite. \end{proof}

\begin{proof}[Proof of Theorem \ref{thm: A-fin-c-sufficient}] All of the work has already been done. Fix $n\in \mathbb{N}$. From Lemma \ref{lem: A-fin-c-1d} we get that the space of closed diagrams is 1-dimensional. Lemma \ref{lem: A-fin-c-spherical} gives that $\mc{P}$ is spherical and Lemma \ref{lem: A-fin-c-graph} gives that the principal graph of $\mc{P}$ is the $\til{A}_{2n-1}$ Dynkin diagram. Finally, we get that $\mc{P}$ is positive definite from Lemma \ref{lem: A-fin-c-posdef}, which completes the proof. \end{proof}

\begin{thm}
    The elements $P_1, Q_1, V,$ and $V^*$ from Definition \ref{Afincolorgenerators} generate $\mc{P}$ in the color case of Theorem \ref{thm: A-finite-necessary-arrow}. 
\end{thm}

\begin{proof} Follow Theorem \ref{thm: Afinshade-generators}. \end{proof}

\section{Distinctness and Fulfillment of Presentations for Affine \texorpdfstring{$A$}{A} Finite Subfactor Planar Algebras}

The arrow case with generators $\up$, $U$, and $U^*$ with relations given in Theorem \ref{thm: A-finite-necessary-arrow} give $2n$ presentations  of $\til{A}_{2n-1}$ unshaded subfactor planar algebras as $\omega_{2n}$ is a $2n$th root of unity. The color case with generators $\red$, $\blue$, $V$, and $V^*$ with relations given in Theorem \ref{thm: A-finite-necessary-arrow} give an additional $n$ presentations as $\tau_n$ is an $n$th root of unity. It is quite natural to ask if any of these are isomorphic and if there are any more presentations left to be found. For convenience lets call 

\begin{align*}
    P_1=\up, Q_1=\down, P'_1=\red, \te{ and } Q'_1=\blue
\end{align*}

\begin{thm}\label{thm: Afinexactly3ncases}
    There are exactly $3n$ nonisomorphic unshaded subfactor planar algebras with principal graph $\til{A}_{2n-1}$.
\end{thm}

First, we will prove that none of the arrow case presentations are isomorphic to the color case presentations. Then we will prove that the $2n$ and $n$ presentations in the respective presentations are distinct. 

\begin{lem} \label{lem: A-fin-different-cases-distinct}
Fix $n \in \mathbb{N}$. All of the $2n$ planar algebras given by presentations in the arrow case for $\til{A}_{2n-1}$ are nonisomorphic to each of the $n$ planar algebras given by presentations in the color case for $\til{A}_{2n-1}$.
\end{lem}

    \begin{proof} Let $\mc{P}$ be a presentation of an arrow case with chosen $2n$th root of unity $\omega_{2n}$ for $\til{A}_{2n-1}
    $ and $\mc{Q}$  be a presentation of the color case with chosen $n$th root of unity $\tau_n$. Suppose for contradiction that there exists an isomorphism $\theta:\mc{P} \ra \mc{Q}$. Recalling the definition of planar algebra isomorphism, there exists a bijective map $\theta_1:\mc{P}_1\ra \mc{Q}_1$. A planar algebra isomorphism preserves minimal projections so either $\theta_1(P'_1)=P_1$ and $\theta_1(Q_1)=Q'_1$ or $\theta_1(P_1)=Q'_1$ and $\theta(Q_1)=P'_1$.

     Let $T$ be the click tangle in the 1-box space. Then $Z_{T}\left(\theta_1\left(P_1\right)\right)=\theta_1(P_1)$ since in $\mc{Q}$ both $P'_1$ and $Q'_1$ are self-dual. As $\theta_1$ is an isomorphism this is equal to $\theta_1\left(Z_{T}\left(P_1\right)\right)=\theta_1\left(Q_1\right)$. Therefore $\theta_1\left(P_1\right)=\theta_1\left(Q_1\right)$, which is a contradiction. Hence there is no isomorphism between $\mc{P}$ and $\mc{Q}$. \end{proof}

     \begin{lem} \label{lem: A-fin-in-case-distinct}
 Fix $n \in \mbb{N}$.
     \begin{enumerate}[label=(\roman*)]
         \item Let $\mc{P}$ and $\mc{Q}$ be planar algebras given by generators from Definition \ref{Afinarrowgenerators} and relations from Definition \ref{Afinarrowrelations} using $2n$th roots of unity $\omega_1$ and $\omega_2$ respectively. If $\omega_1\neq \omega_2$ then $\mc{P}$ and $\mc{Q}$ are nonisomorphic. 
        \item Let $\mc{P}$ and $\mc{Q}$ be planar algebras given by generators from Definition \ref{Afincolorgenerators} with relations from Definition \ref{Afinarrowrelations} using $n$th roots of unity $\tau_1$ and $\tau_2$ respectively. If $\tau_1\neq \tau_2$ then $\mc{P}$ and $\mc{Q}$ are nonisomorphic. 
    \end{enumerate}
     \end{lem}

     \begin{proof} For sake of contradiction let $\theta:\mc{P} \ra \mc{Q}$ be an isomorphism where $\mc{P}$ and $\mc{Q}$ are planar algebra with presentations given in the arrow case with chosen distinct $2n$th roots of unity $\omega_1$ and $\omega_2$ respectively. Call the $U$-box of $\mc{P}$ $U$ and the $U$-box of $\mc{Q}$ $U'$. Call the distinct minimal projections in the 1-box space of $\mc{P}$ and $\mc{Q}$ both $P_1$ and $Q_1$. However, it should be clear from context which planar algebra these projections lie in. By definition of planar algebra isomorphism, $\theta_1:\mc{P}_1 \ra \mc{Q}_1$ is a bijective map between the 1-box spaces. As planar algebra isomorphisms preserve minimal projections, either $\theta_1(P_1)=P_1$ and $\theta_1(Q_1)=Q_1$ or $\theta_1(P_1)=Q_1$ and $\theta_1(Q_1)=P_1$.

     A planar algebra isomorphism also preserves the tensor product. So when $\theta_1(P_1)=P_1$, $\theta(U)\in \te{Hom}\left(P_1^{\otimes n}, Q_1^{\otimes n}\right)$. Therefore $\theta(U)$ is a multiple of $U'$ in $\mc{Q}$. When $\theta_1(P_1)=Q_1$, we get that $\theta(U)$ is a multiple of $U'^*$.

    We then conclude that $\omega_1=\omega_2$. In the case where $\theta_1(P_1)=P_1$ we get that there exists a $\lambda \in \mathbb{C}$ such that $\theta(U)=\lambda U'$. Let $T=\mc{F}$ be the click tangle in the $n$-box space.  Then $Z_{T}(\theta(U))=Z_{T}(\lambda U')=\lambda Z_{T}(U')=\lambda \omega_2 U'$ which is the same as $\theta(Z_{T}(U))=\theta(\omega_1 U)=\omega_1\theta(U)=\omega_1 \lambda U'$. This gives that $\omega_1=\omega_2$, contradicting that $\omega_1$ and $\omega_2$ are distinct. The same process gives a contradiction when $\theta_1(P_1)=Q_1$. Thus the claim follows for part (i). 
     
     By replacing $P_1$ with $P'_1$ and $Q_1$ with $Q'_1$, the claim is nearly identical for part (ii). \end{proof}

    We can now conclude exactly how many such unshaded subfactor planar algebras there are. Recall that our definition of these planar algebra have no shading, so this result is not in disagreement to the result of Popa in \cite{Pop94}. 

     \begin{proof}[Proof of Theorem \ref{thm: Afinexactly3ncases}] By Lemma \ref{lem: A-fin-different-cases-distinct} none of the arrow cases are isomorphic to the color cases and by Lemma \ref{lem: A-fin-in-case-distinct} the $2n$ presentations found in the arrow case and the $n$ presentations found in the color case are distinct. Thus we have at least $3n$ nonisomorphic unshaded subfactor planar algebras with principal graph $\til{A}_{2n-1}$. 

     By Theorem \ref{thm: A-finite-necessary-arrow} and Theorem \ref{thm: A-fin-arrow-generate} if we have another unshaded subfactor planar algebra with principal graph $\til{A}_{2n-1}$ then it must have either the arrow case presentation or the color case presentation so there are no further presentations. Thus there are exactly $3n$ nonisomorphic unshaded subfactor planar algebras with principal graph $\til{A}_{2n-1}$. \end{proof}

By following the proof with minor adjustments of Theorem \ref{thm: Afinexactly3ncases} we get the following theorem.

\begin{thm}
    There are exactly $2n+1$ nonisomorphic unshaded subfactor planar algebras with principal graph $\til{A}_{2n}$.
\end{thm}

\section{Equivalence to Cyclic Pointed Fusion Categories}

    Adopt the notation given for the the arrow cases of the $\til{A}_{2n}$ and $\til{A}_{2n-1}$ unshaded subfactor planar algebras. These categories' objects are tensor products of $Q_1$ and $P_1$. Since, for $\mc{P}$, $Q_1^{\otimes (2n-1)}\cong P_1^{\otimes n}\otimes Q_1^{\otimes (n-1)}\cong P_1$, it is enough to say that $\mc{C}_{\mc{P}}$ has objects generated by $Q_1$ (and similarly for $\mc{P}')$. Let $\mc{C}_{m, \ol{\zeta}}$ be the corresponding category that when $m=2n$ is $\mc{P}$ with principal graph $\til{A}_{2n-1}$ and when $m=2n+1$ is $\mc{P}'$ with principal graph $\til{A}_{2n}$, with $m$th chosen root of unity $\ol{\zeta}$. In this section, we dive further into tensor categories. For more details see \cite{ENO05, EGNO15}.
    
    Recall that a \textit{fusion} category, $\mc{C}$, over $\mbb{C}$, is a rigid, semisimple, $\mbb{C}$-linear, abelian category with only finitely many isomorphism classes of simple objects, and the unit object $\mbb{1}$ is simple. An object in a monoidal category, $\mc{C}$, is \textit{invertible} if its evaluation and coevaluation maps are isomorphisms and a monoidal category $\mc{C}$ is \textit{pointed} if every simple object in $\mc{C}$ is invertible. 

\begin{ex} Let $G$ be a finite group and $\omega:G\times G\times G \ra \mbb{C}^\times$ be a 3-cocycle. That is, $\omega$ satisfies 
\begin{align*}
    \omega(g_1g_2,g_3,g_4)\omega(g_1,g_2, g_3g_4)=\omega(g_1,g_2,g_3)\omega(g_1,g_2g_3,g_4)\omega(g_2,g_3,g_4)
\end{align*}
for all $g_1,g_2,g_3, g_4 \in G$. Define $\text{Vec}_G^\omega$ as the category of finite-dimensional vector spaces over $\mbb{C}$ graded by $G$ with associator defined by $\omega$. That is, the simple objects are vector spaces labeled by elements of $G$, $V_g$, (so there is only one object in each isomorphism class) the tensor product is given by $V_g \otimes V_h=V_{gh}$, and the associativity isomorphism is determined by
\begin{align*}
    \alpha_{V_g, V_h, V_k}=\omega(g,h,k)\te{id}_{V_{ghk}}:(V_g\otimes V_h)\otimes V_k \to V_g \otimes (V_h\otimes V_k)
\end{align*}
where $g,h,k \in G$. Every pointed fusion category over $\mbb{C}$ has the form $\te{Vec}_G^\omega$ for some finite group $G$ and 3-cocycle $\omega$ \cite{ENO05}. 
\end{ex}
    
    A category is \textit{strict} if the associator and both left and right unitors are the identity. For any nontrivial 3-cocycle $\omega$, $\te{Vec}_G^\omega$ is not strict. The Mac Lane Strictness Theorem \cite{ML98} states that any monoidal category is monoidally equivalent to a strict monoidal category, so we can find some strict monoidal category equivalent to $\te{Vec}_G^\omega$. In fact, for all $m \in \mbb{N}$ and associated roots of unity $\ol{\zeta}$, $\mc{C}_{m,\ol{\zeta}}$ can easily be seen to be a strict fusion category over $\mbb{C}$. The goal of this section then is to prove the following theorem.

    \begin{thm}\label{thm: monoidallyequivtoVecGomega}
        $\mc{C}_{m, \ol{\zeta}}$ is monoidally equivalent to $\te{Vec}_{\mbb{Z}_m}^\omega$ for some 3-cocycle $\omega$. 
    \end{thm}

    Through the proof of this theorem we will see how in instead of a nontrivial associator, $\omega$, of $\te{Vec}_{\mbb{Z}_m}^\omega$, we have an on-the-nose association, at the expense of more than one object in each isomorphism class of simples. 
    
    Czenky showed in \cite{Cze24} that the pointed fusion categories for the cyclic group $G=\mbb{Z}_m$ have a nice classification based on $m$ roots of unity, which is denoted $\te{Vec}_{\mbb{Z}_m}^\zeta$ for chosen root of unity $\zeta$. She then gives a diagrammatic description of $\te{Vec}_{\mbb{Z}_m}^\zeta$. Essentially, we prove Theorem \ref{thm: monoidallyequivtoVecGomega} by finding a functor between $\te{Vec}_{\mbb{Z}_m}^\zeta$ and $\mc{C}_{m, \ol{\zeta}}$ which gives an equivalence of categories. Note that the choice of $\zeta$ for $\te{Vec}_{\mbb{Z}_m}^\zeta$ and $\ol{\zeta}$ for $\mc{C}_{m,\ol{\zeta}}$ is intentional as we will see that for any $m$th root of unity $\zeta$, $\te{Vec}_{\mbb{Z}_m}^\zeta$ to monoidally equivalent to $\mc{C}_{m,\ol{\zeta}}$ where now $\ol{\zeta}$ indicates complex conjugation of $\zeta$. 

    Fix $m\in \mbb{Z}$ and an $m$th root of unity $\zeta$. For this $\zeta$ there is a 3-cocycle $\omega_\zeta:\mbb{Z}_m\times \mbb{Z}_m \times \mbb{Z}_m \to \mbb{C}^\times$ defined by, for all $i,j,k \in \mbb{Z}_m$,
    \begin{align*}
        \omega_\zeta(i,j,k)=\zeta^{\frac{i(j+k-[j+k])}{m}}
    \end{align*}
    where $[\cdot]$ indicates the value taken modulo $m$. Not only does every $m$th root of unity determine a 3-cocycle $\omega_\zeta$, but every 3-cocycle (modulo coboundaries) is of this form, so $\te{Vec}_{\mbb{Z}_m}^\zeta$ is defined to be $\te{Vec}_{\mbb{Z}_m}^{\omega_\zeta}$\cite{Cze24}. In section 5 of the same paper, Czenky shows that $\te{Vec}_{\mbb{Z}_m}^\zeta$ can be given the diagrammatic description
    \begin{equation*}
        \textit{Objects: } \agusobj\te{ } \os{(k)}{ ... }\te{ } \agusobj \te{ ($k$ dots)}
    \end{equation*}
    and
    \textit{Morphisms: }
    \begin{equation}
        \te{id}_1=\agusmorphid, \mc{U}=\agusmorphU, \mc{U}^*=\agusmorphUad
    \end{equation}
    with  \textit{Relations: } 
    \begin{enumerate}[label=(\roman*)]
        \item $\mc{U}^*\mc{U}=\agusUadU = \te{id}_0$
        \item $\mc{U}\mc{U}^*=\agusUUad = \agusmorphid \os{(m)}{...} \agusmorphid $
        \item $\agusitensorUlhs=\zeta \agusitensorUrhs$
    \end{enumerate} 
    with tensor given by horizontal concatenation and composition of morphisms is vertical stacking. In fact $\mc{U}$, $\mc{U}^*$ form a generating set for the morphisms. When there is no confusion, we will call the object $\agusobj\te{ } \os{(k)}{ ... }\te{ } \agusobj$, $k$. Call this diagrammatic category $\mc{D}_{m,\zeta}$. We show $\mc{D}_{m,\zeta}$ is monoidally equivalent to the monoidal subcategory $\langle Q_1 \rangle$ of $\mc{C}_{m,\ol{\zeta}}$ generated by $Q_1$ whose additive envelope is equivalent to $\mc{C}_{m,\ol{\zeta}}$. 

    We are going to map $\mc{U}$ and $\mc{U}^*$ by a functor to special diagrams in our planar algebra defined below:
    \begin{equation}\label{eq: tilUandtilUaddefined}
        \til{U}=\tilU, \te{ and } \til{U}^*=\tilUad
    \end{equation}
    where these have $m$ strands on the top or bottom respectively. Notice that could have equivalently defined $\mc{P}$ and $\mc{P}'$ to have generators $Q_1$, $\til{U}$, and $\til{U}^*$ instead of $P_1$, $U$, and $U^*$. We will redefine our categories $\mc{C}_\mc{P}$ and $\mc{C}_{\mc{P}'}$ using $\til{U}$, $\til{U}^*$, and $Q_1$ as generating morphisms.
    
    \begin{tcolorbox}[breakable, pad at break*=0mm]
        Define a functor $\mc{A}:\mc{D}_{m, \zeta}\to \langle Q_1 \rangle$ on generating objects and morphisms by
    \begin{align*}
        \agusobj &\mapsto Q_1\\
        \agusmorphid &\mapsto \down\\
        \mc{U} &\mapsto \til{U}\\
        \mc{U}^* &\mapsto \til{U}^*
    \end{align*}
        then extend linearly and monoidally on morphisms. 
    \end{tcolorbox}

    Recall from the background material that $\mc{A}$ will yield an equivalence of categories if it is full, faithful, and essentially surjective. We begin by ensuring $\mc{A}$ is a well-defined functor. 

    \begin{lem}\label{lem: agusmapwelldefined}
        $\mc{A}$ is well-defined.
    \end{lem}

    \begin{proof} To show the lemma, we need to prove that $\mc{A}$ respects the defining relations of $\mc{D}_{m,\zeta}$. The left-hand side of relation (i) for $\mc{D}_{m,\zeta}$, $\mc{U}^*\mc{U}$ maps by $\mc{A}$ to $\til{U}^*\til{U}$ which equals $\te{tr}(U^*U)=\emptyset$. As $\mc{A}(\te{id}_0)=\emptyset$, $\mc{A}$ respects relation (i). $\mc{F}$ maps $\mc{U}\mc{U^*}$ to $\til{U}\til{U}^*$. In between $U$ and $U^*$ there are oppositely oriented strands. Using the saddle relations $n$ (or $n+1$, if $m$ is even) times we obtain $\til{U}\til{U}^*=UU^*\otimes Q_1 \otimes ... \otimes Q_1$, which equals $Q_1^{\otimes m}$. This is the image of $\te{id}_1 ^{\otimes m}$, so $\mc{F}$ respects relation (ii). The image of the left-hand side of relation (iii) is $Q_1\otimes \til{U}$. This equals to $\ol{\zeta}^{-1} Q_1 \otimes \mc{F}^{-1}(\til{U})$. Then on the left-side of the diagram, there are oppositely oriented strands. Using the saddle relation we get this is equal to $\ol{\zeta}^{-1}\til{U}\otimes Q_1=\zeta \til{U}\otimes Q_1$, which is the image of the right-hand side of relation (iii). Therefore $\til{A}$ is a well-defined functor. \end{proof}

    \begin{lem}\label{lem: agusmapsurj}
        $\mc{A}$ is essentially surjective on objects.
    \end{lem}

    \begin{proof} As $Q_1$ is the generating object in $\mc{C}_{m,\ol{\zeta}}$ and the image of the generating object in $\mc{D}_{m,\zeta}$, the proof is immediate. \end{proof}

    To show that $\mc{A}$ is full and faithful, we need to check that the induced map 
    \begin{align*}
        \til{A}: \te{Hom}_{\mc{D}_{m,\zeta}}(k,\ell)\to \te{Hom}_{\mc{C}_{m,\ol{\zeta}}}\left(Q_1^{\otimes k}, Q_1^{\otimes \ell}\right)
    \end{align*}
    is bijective for all $k,\ell \in \mbb{Z}_{\geq 0}$. Duality allows a diagram in $\te{Hom}_{\mc{D}_{m,\zeta}}(k,\ell)$ to be turned into a diagram in $\te{Hom}_{\mc{D}_{m,\zeta}}(k+\ell,0)$. Furthermore the bending of the diagram is linear and bijective, so we can equivalently prove that the induced map
    \begin{align*}
        \til{A}: \te{Hom}_{\mc{D}_{m,\zeta}}(k,0)\to \te{Hom}_{\mc{C}_{m,\ol{\zeta}}}\left(Q_1^{\otimes k},\emptyset\right)
    \end{align*}
    is bijective for all $k\in \mbb{N}$.

    The hom-spaces $\te{Hom}_{\mc{D}_{m,\zeta}}(k,0)$ are one-dimensional for every $k$ divisible by $m$ and is zero otherwise \cite{Cze24}. If these are the dimensions of the hom-spaces $\te{Hom}_{\mc{C}_{m,\ol{\zeta}}}\left(Q_1^{\otimes k},\emptyset\right)$ then we only need to show $\til{A}$ is full.

    \begin{lem}\label{lem: dimHomis1or0}
        The dimension of $\te{Hom}_{\mc{C}_{m,\ol{\zeta}}}\left(Q_1^{\otimes k},\emptyset\right)$ is 1 if $k$ is divisible by $m$ and is zero otherwise.
    \end{lem}

    \begin{proof} Without loss of generality, suppose $m=2n+1$. The result when $k=0$ was proven in Lemma \ref{lem: P0-1-d}. If $0<k<n$ then $Q_1^{\otimes k}$ is a minimal projection so the hom-space is zero-dimensional. Suppose $n\leq k <2n$. Lemma \ref{lem: AfinarrowPnisotoQn} gives that $Q_1^{\otimes k}\cong P_1^{\otimes (2n-k)}$, which will be a minimal projection not isomorphic to $\emptyset$, giving again that the hom-space is zero-dimensional. If $k\geq 2n$, use the isomorphism $Q_1^{\otimes (2n)}\cong \emptyset$ and the previous cases of $k$ to get the result. \end{proof}
    
    \begin{lem}\label{lem: agusmapfullyfaithful}
        $\mc{A}$ is full and faithful. 
    \end{lem}

    \begin{proof} Consider the induced map  $\til{A}: \te{Hom}_{\mc{D}_{m,\zeta}}(k,0)\to \te{Hom}_{\mc{C}_{m,\ol{\zeta}}}\left(Q_1^{\otimes k},\emptyset\right)$ for any $k\in \mbb{Z}_{\geq 0}$. By Lemma \ref{lem: dimHomis1or0} we only need to show that $\mc{A}$ is full. The dimension of both hom-spaces are zero when $k$ is not divisible by $m$, so we only need to consider when $k$ is divisble by $m$. In this case, the hom-spaces are both one-dimensional, so it is enough to check the image is nontrivial. Let $k=2n\ell$ for some $\ell \in \mbb{Z}_{\geq 0}$. When $\ell=0$, $\te{id}_0$ maps to $\emptyset$, which is nontrivial as it has nonzero trace. When $\ell>0$ then $\mc{U}^{\otimes \ell}$ maps to $\til{U}^{\otimes \ell}$. Taking the inner product of $\til{U}^{\otimes \ell}$ with itself is 1 and thus the image is nontrivial. \end{proof}

    \begin{proof}[Proof of Theorem \ref{thm: monoidallyequivtoVecGomega}] $\mc{D}_{m,\ol{\zeta}}$ is monoidally equivalent to $\langle Q_1 \rangle$ by Lemmas \ref{lem: agusmapwelldefined}, \ref{lem: agusmapsurj}, and \ref{lem: agusmapfullyfaithful}. The additive envelope of $\mc{D}_{m,\ol{\zeta}}$ is monoidally equivalent to $\te{Vec}_{\mbb{Z}_m}^\zeta$ and the additive envelope of $\langle Q_1\rangle$ is $\mc{C}_{m, \ol{\zeta}}$, giving the result. \end{proof}

\section{Equivalence to a Representation Category of a Finite Subgroup of \texorpdfstring{$SU(2)$}{SU(2)}}\label{repsu2section}

In a similar vein to the previous section, we show that our categories $\mc{C}_{m,1}$ are equivalent to some diagrammatic category associated to a finite subgroup of $SU(2)$. Recall that $SU(2)$ is the group of $2\times 2$ matrices with entries in $\mbb{C}$ of unitary matrices with determinant 1. The Temperley-Lieb planar algebra of index 4, $\mc{TL}$, is equivalent to the category of finite-dimensional representations of $SU(2)$, $\te{Rep}(SU(2))$.  Further,  let $V=\mbb{C}^2$ be the defining representation of $SU(2)$ and $X$ a strand in $\mc{TL}$. Then we get an isomorphism of the hom-space $\te{Hom}(X^{\otimes k}, X^{\otimes \ell})$ with the space of $SU(2)$-invariant homomorphisms $\te{Hom}_{SU(2)}(V^{\otimes k}, V^{\otimes \ell})$. 

Finite subgroups of $SU(2)$ were classified by Klein \cite{Kle78}. McKay \cite{McK80} showed that there is a one-to-one correspondence between the finite subgroups of $SU(2)$ and the simply-laced affine Dynkin diagrams. Fix a primitive $m$th root of unity $\zeta$. Let $C_m^\zeta$ be the subgroup of $SU(2)$ generated by the matrix
\begin{align*}
    \begin{pmatrix} \zeta & 0 \\ 0 & \zeta^{-1} \end{pmatrix}.
\end{align*}
Define $\langle V \rangle$ to be the $\mbb{C}$-linear monoidal full subcategory generated by $V$. The representation categories of the finite subgroups of $SU(2)$ are not strict, but, as done with Temperley-Lieb, our goal is to find equivalent strict monoidal categories with nice diagrammatics. To that end, we prove the following theorem:

\begin{thm}\label{thm: ryanequivalence}
    $\mc{C}_{m,1}$ is monoidally equivalent to $\te{Rep}(C_m^\zeta)$. 
\end{thm}

We proceed similarly as in the previous section. Reynolds showed in \cite{Rey23} that $\langle V \rangle$ is monoidally equivalent to the following diagrammatic category, which we will call $\mc{B}_{m,\zeta}$. It has the following diagrammatic presentation:
   \begin{equation*}
        \textit{Objects: } [k]=\ryanobj\te{ } \os{(k)}{ ... }\te{ } \ryanobj \te{ ($k$ stars)}
    \end{equation*}
    and \textit{Morphisms: }
    \begin{equation}
        \ryanmorphplus, \ryanmorphminus, \ryanmorphcapminusplus, \ryanmorphcapplusminus, \ryanmorphcupminusplus, \ryanmorphcupplusminus, \ryanmorphncapplus, \ryanmorphncapminus, \ryanmorphncupplus, \te{ and }\ryanmorphncupminus
    \end{equation}
    with composition of diagrams being vertical stacking  and tensor product is horizontal concatenation. If a $+$ and $-$ are matched anywhere, we obtain the 0 morphism and we have
    
    \textit{Relations: } 
    \begin{enumerate}[label=(\roman*)]
    \begin{paracol}{2}
        \item $\ryanmorphcupplusminus \ryanmorphplus = \ryanmorphplus \ryanmorphcupplusminus, \ryanmorphcupminusplus \ryanmorphminus= \ryanmorphminus \ryanmorphcupplusminus$
    \switchcolumn
       \setcounter{enumi}{1} \item $\ryanmorphcapplusminus\ryanmorphplus=\ryanmorphplus \ryanmorphcapminusplus, \ryanmorphcapminusplus \ryanmorphminus=\ryanmorphminus\ryanmorphcapplusminus$
    \switchcolumn
        \setcounter{enumi}{2} \item $\ryanrelcupcapplusminus=1$, $\ryanrelcupcapminusplus=1$
    \switchcolumn
        \setcounter{enumi}{3}\item  $\ryanmorphncupcapplus=1=\ryanmorphncupcapminus$
    \switchcolumn
        \setcounter{enumi}{4}\item $\ryanmorphncapminus \ryanmorphminus = \ryanmorphminus \ryanmorphncapminus$
    \switchcolumn
       \setcounter{enumi}{5} \item $\ryanmorphncapplus\ryanmorphplus=\ryanmorphplus \ryanmorphncapplus$
    \switchcolumn
        \setcounter{enumi}{6}\item $\ryanmorphncapplus\ryanmorphncapminus=\ryantwonplusminus$
        \switchcolumn
        \setcounter{enumi}{8}\item $\ryanmorphplus \ryanmorphminus =  \ryanrelsaddleplusminus$, $\ryanmorphminus \ryanmorphplus=\ryanrelsaddleminusplus$
    \switchcolumn
       \setcounter{enumi}{7} \item $\ryanmorphncapminus\ryanmorphncapplus=\ryantwonminusplus$
    \end{paracol}
\end{enumerate}
    
   We define a functor that will give the equivalence of categories in the following way. Recall the definition of $\til{U}$ and $\til{U}^*$ from (\ref{eq: tilUandtilUaddefined}). Define $\til{U}'$ and $\til{U}^{*'}$ to be flipping $\til{U}$ and $\til{U}^*$ over a horizontal axis, respectively.
     
   \begin{tcolorbox}[breakable, pad at break*=0mm]
        Define a functor $\mc{R}:\mc{B}_{m, \zeta}\to \langle X \rangle$ on generating objects by $\ryanobj \mapsto X$ and generating morphisms by
        \begin{align*}
          &\ryanmorphplus \mapsto \up  &\ryanmorphminus \mapsto \down\hspace{1.25cm} & & \\
         &\ryanmorphcapminusplus \mapsto \ocapl & \ryanmorphcapplusminus \mapsto\ocapr\hspace{0.35cm}
          &\hspace*{1cm}\ryanmorphcupminusplus\mapsto \ocupr &\ryanmorphcupplusminus \mapsto \ocupl\hspace*{0.29cm}\\
          &\ryanmorphncapplus \mapsto \til{U}^{*'} \hspace{-0.2cm}&\ryanmorphncapminus \mapsto \til{U}^*  & \hspace{1cm}\ryanmorphncupplus \mapsto \til{U}'
          & \ryanmorphncupminus \mapsto \til{U}
        \end{align*}
        then extend linearly and monoidally on morphisms. 
    \end{tcolorbox}

\begin{lem}
    $\mc{R}$ is a well-defined functor.
\end{lem}

\begin{proof} We need to check that $\mc{R}$ respects the defining relations. Using saddle relations we see that relations (i) and (ii) are respected. Relation (iii) are the bubble relations. To see $\mc{R}$ respects relation (iv), notice that image of the right-hand side is $Q_1 \otimes \til{U}^*$. The top left-most strand of $\til{U}^*$ and $Q_1$ have opposite orientation, so applying the saddle relation, then the inverse of the click relation gives $\til{U}^*\otimes Q_1$. The same idea applies to relation (v). Relation (vi) are exactly the saddle relations. Relation (vii) is taking the trace of $UU^*$ and $U^*U$. For relation (viii), the image of the left-hand side is $\til{U}'\otimes \til{U}^*$. The rightmost strand of $\til{U}'$ and the leftmost strand of $\til{U}^*$ have opposite orientation. The saddle relation then gives an equivalent diagram with a strand of $\til{U}'$ and $\til{U}^*$ connected. Applying this $n$ or $n+1$ times depending on the parity of $m$, we then can apply the $U$ unitary relations to get the image of the lefthand side of the relation. Relation (ix) is done the same way, so $\mc{R}$ is well-defined. \end{proof}

The next lemma is immediate as the two categories $\langle V \rangle$ and $\langle X \rangle$ each have one generating object and the image of $\ryanobj$ is $X$. 

\begin{lem}
    $\mc{R}$ is essentially surjective on objects.
\end{lem}

Corollary 3.6 of \cite{Rey23} gives a basis of $\te{Hom}_{\mc{B}_{m,\zeta}}([k],[\ell])$. Let $\epsilon$ and $\delta$ be two vectors with entries $+$ or $-$ indicating the bottom and top boundary information of a morphism. Let $|\epsilon_+|$ be the number of $+$ labellings of $\epsilon$ and similarly define $|\epsilon_{-}|$, $|\delta_{+}|$, and $|\delta_{-}|$. Define $|\epsilon|$ and $|\delta|$ to be the absolute value of the difference in labeling. In Lemma 3.3, Reynolds proves that $|\epsilon|\cong |\delta|(\te{mod } m)$. Reynolds also proves that any two morphisms with the same labelings are equivalent, so we can call this unique morphism $d_{\epsilon}^\delta$. Then the collection of these morphisms over the possible label sets form a basis of $\te{Hom}_{\mc{B}_{m,\zeta}}([k],[\ell])$. 

We can follow the same process as in \cite{Rey23} by noticing that the $+$ that appear in coordinates of $\delta$ correspond to arrows pointing in the direction of $+$ and in $\epsilon$ correspond to arrows pointing away from $+$. The vice versa is true for $-$. Letting now $|\epsilon_+|$ corresponding to the number of points on the bottom boundary being the sources, $|\epsilon_{-}|$ being the number of sources on the bottom boundary being the sinks, $|\delta_{+}|$ being the number of points on the top boundary that are sinks and $|\delta_{-}|$ being the number of points on the top boundary being the sources, we also have $|\epsilon|\equiv |\delta| (\te{mod } m)$. Further, using the exact same reasoning as his Theorem 3.5, we can conclude that any two diagrams with the same labellings (now labellings being ``source" or ``sink") are equivalent. It is also clear that this unique morphism will be the image of $\mc{R}(d_{\epsilon}^{\delta})$. Thus the set of $\mc{R}(d_{\epsilon}^\delta)$ where the length of $\epsilon$ is $k$ and $\delta$ is $\ell$ form a spanning set of $\te{Hom}_{\langle X \rangle }\left(X^{\otimes k},X^{\otimes \ell}\right)$.

\begin{lem}
    $\mc{R}$ is full and faithful.
\end{lem}

\begin{proof} As $\mc{R}(d_{\epsilon}^\delta)$ is a spanning set for $\te{Hom}_{\langle X \rangle}(X^{\otimes k}, X^{\otimes \ell})$, we know that $\mc{R}$ is full. Further, each $\mc{R}(d_\epsilon^\delta)$ can be split in half horizontally by construction, i.e., $\mc{R}(d_\epsilon^\delta)$ can be written as the product of two morphisms say $fg$. Then $f^*fgg^*$ is a product where all caps are multiplied by cups and all $U$ multiplied by some $U^*$, which all evaluate to 1. So $f^*fgg^*$ can be seen to be a tensor product of $P_1$ and $Q_1$. The trace of this morphism is 1 so $fg=\mc{R}(d_\epsilon^\delta)$ is nonzero. Further, if $\epsilon,\epsilon',\delta, \delta'$ are labellings of boundary data and either $\epsilon\neq \epsilon'$ or $\delta \neq \delta'$ then $\mc{R}(d_\epsilon^\delta)$ and $\mc{R}(d_{\epsilon'}^{\delta'})$ live in two different hom-spaces in the planar algebra, which is graded by the boundary data. Thus the collection of $\mc{R}(d_\epsilon^\delta)$ where $\epsilon$ is boundary data of length $k$ and $\delta$ is boundary data of length $\ell$ form a basis of $\te{Hom}_{\langle X \rangle}\left(X^{\otimes k},X^{\otimes \ell}\right)$. Therefore $\mc{R}$ is faithful.  
\end{proof}

\begin{proof}[Proof of Theorem \ref{thm: ryanequivalence}] $\langle X \rangle$ is monoidally equivalent to $\mc{B}_{m,\zeta}$ which in \cite{Rey23} is shown to be monoidally equivalent to $\langle V \rangle$. Taking the Cauchy completion gives the result. \end{proof}
\section{The Affine \texorpdfstring{$A$}{A} Infinity Subfactor Planar Algebra} 

To conclude the affine $A$ type planar algebra presentations, we consider those with the type $\til{A}_\infty$ Dynkin diagram. Let $\mc{P}$ be a subfactor planar algebra with principal and dual principal graph the $\til{A}_\infty$ Dynkin diagram. Let $\mc{C}_{\mc{P}}$ be the 2-category created from $\mc{P}$. Define $X$, $Y$, $\emptyset_{+}$, and $\emptyset_{-}$ as in section \ref{necessaryshadedsection}. With labels $P_i, P_i', Q_i$, and $Q_i'$ denoting representatives of equivalence classes for $1 \leq i$, the principal and dual principal graphs are:

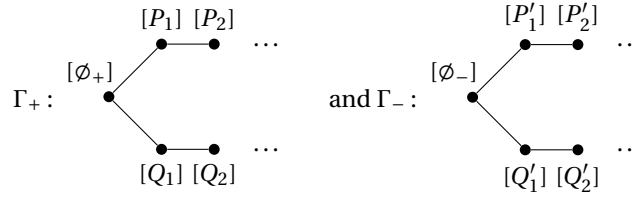
\begin{align*}
   \Gamma_{+}: \arbAinfgraphpos \text{ and } \Gamma_{-}: \arbAinfgraphneg  
\end{align*}

Infinite graphs may correspond to multiple indices so we will take that the corresponding subfactor planar algebras has index 4 as a given.

\begin{lem}\label{lem: Ainftraceofminprojs}
    The trace of any minimal projection in $\mc{C}_{\mc{P}}$ is 1. 
\end{lem}

    \begin{proof} First we prove by induction that for all $k\geq 0$, $\te{tr}(P_k)+\te{tr}(Q_k)=2$, where we define $P_0=Q_0=\emptyset_{+}$. The base cases follow from $\te{tr}(\emptyset_{+})=1$ and from the index being 4 since, $\te{tr}(X)=2=\te{tr}(P_1)+\te{tr}(Q_1)$. Assume the claim is true up to some $k\in \mbb{N}$. The principal graph gives that $2\te{tr}(P_k)=\te{tr}(P_{k-1})+\te{tr}(P_{k+1})$ and $2\te{tr}(Q_k)=\te{tr}(Q_{k-1})+\te{tr}(Q_{k+1})$. Adding these two equations together and using the inductive hypothesis gives $\te{tr}(P_{k+1})+\te{tr}(Q_{k+1})=2$. Therefore, for all $k\geq 0$,  $\te{tr}(P_k)+\te{tr}(Q_k)=2$. The same process shows the analogous result for the dual principal graph. 
    
    Now suppose that $\te{tr}(P_1)\neq 1$, which also implies $\te{tr}(Q_1)\neq 1$. Either $P_1$ or $Q_1$ will have trace less than one by the preceding paragraph, so without loss of generality assume $\te{tr}(P_1)<1$. We then show that $(\te{tr}(P_k))_{k\in \mbb{N}}$ is a strictly decreasing sequence. By the principal graph, $\te{tr}(P_1)=\f{1+\te{tr}(P_2)}{2}$. That is, $\te{tr}(P_1)$ is the arithmetic mean of 1 and $\te{tr}(P_2)$. Since $\te{tr}(P_1)<1$, this means $\te{tr}(P_2)<\te{tr}(P_1)$, as desired. Assume up to some $k$, $\te{tr}(P_{k+1})<\te{tr}(P_k)$. The trace of $P_{k+1}$ is the arithmetic mean of the trace of $P_k$ and the trace of $P_{k+2}$, so its value lies in between the two values. Since $\te{tr}(P_{k+1})<\te{tr}(P_k)$, this gives that $\te{tr}(P_{k+2})<\te{tr}(P_{k+1})$.  Thus $(\te{tr}(P_k))_{k\in \mbb{N}}$ is a strictly decreasing sequence. As the graph is infinite, this means there is some $N \in \mbb{N}$ where $\te{tr}(P_N)<0$. However, since the inner product is positive definite and $P_N$ is a minimal projection, we reach a contradiction. Therefore, $\te{tr}(P_1)=1$ and $\te{tr}(Q_1)=1$. By shifting the starting term of the infinite sequence in the above argument, $\te{tr}(P_k)=\te{tr}(Q_k)=1$. The same argument applies for the dual principal graph. Thus the trace of all minimal projections is 1. \end{proof} 

    Next, we can follow the same process as done for the $\til{A}_{2n-1}$ subfactor planar algebras. As we can just ignore any lemma or part of a proof involving isomorphism $P_n\cong Q_n$, these proofs are easier adaptations of the finite affine $A$ case and will thus be omitted. We then get the following theorems.

\begin{thm}\label{thm: Ainfshadednecessary} (Necessary relations for $\til{A}_{\infty}$)
    If $\mc{P}$ is a subfactor planar algebra with principal and dual principal graph the $\til{A}_\infty$ subfactor planar algebra then $\mc{P}$ has diagrammatic elements $P_1$ and $Q_1$ from Definition \ref{eq: elementsofaffineAfiniteshaded} with relations (i) through (iv) of Definition \ref{relationsofAffineAfinite}. Further, we get the equivalence classes of minimal projections are:
    \begin{equation*}
        \Gamma_{+} \Ainfshadegraphpos \text{ and } \Gamma_{-} \Ainfshadegraphneg
    \end{equation*}
        
\end{thm}

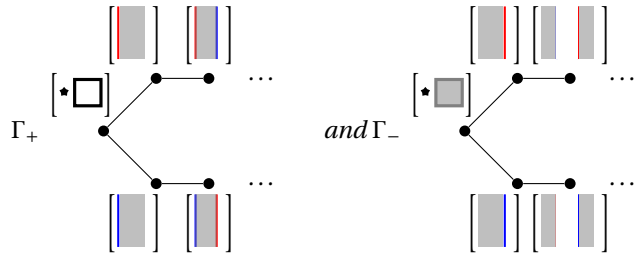
\begin{thm} (Sufficient relations for $\til{A}_\infty$)
    Let $\mc{P}$ be the planar algebra with generator $P_1$ from Definition \ref{eq: elementsofaffineAfiniteshaded} with relations (i) through (iv) given in Definition \ref{relationsofAffineAfinite}. Define $P_1$ and $Q_1$ to be self-adjoint, then extend the * operation anti-linearly and on diagrams. Then $\mc{P}$ is a subfactor planar algebra whose principal graph an dual principal graph are the $\til{A}_\infty$ ones given in Theorem \ref{thm: Ainfshadednecessary}.     
\end{thm}

\begin{thm}
    The elements $P_1$ and $Q_1$ from Theorem \ref{thm: Ainfshadednecessary} generate the planar algebra $\mc{P}$ in that theorem.
\end{thm}

The following theorem, which is a known result by Popa \cite{Pop94}, is then direct from the planar algebra presentation. 

\begin{thm}
    There is exactly one subfactor planar algebra with principal graph $\til{A}_\infty$ up to isomorphism. 
\end{thm}

We now mimic the unshaded $\til{A}_{2n-1}$ planar algebra section for unshaded $\til{A}_\infty$ planar algebras. From now on, let $\mc{P}$ be an unshaded subfactor planar algebra with principal graph the $\til{A}_{\infty}$ Dynkin diagram. Define
\begin{align*}
    X=\strand.
\end{align*}

As before, we take it as fact that $\mc{P}$ has index 4. By Lemma \ref{lem: Ainftraceofminprojs} the trace of any minimal projection in $\mc{C}_{\mc{P}}$ is 1. 

The proofs of the following theorems follow the exact formulation as the finite $\til{A}$ case with less steps due to the infinite graph. In particular, we know that there will be exactly one representative of $P_1$ and $Q_1$ in $\mc{P}_1$ and they will either be dual to themselves or dual to each other. 

\begin{thm}\label{thm: Ainfnecessary} (Necessary relations for $\til{A}_\infty$, unshaded) If $\mc{P}$ is an unshaded subfactor planar algebra with principal graph $\til{A}_\infty$ then one of two cases hold, which we call the arrow case and the color case.

\begin{enumerate}
    \item In the arrow case, $\mc{P}$ has elements $P_1$  given in Definition \ref{Afinarrowgenerators} with relations (i) through (iv) given in Definition \ref{Afinarrowrelations}. Further, we get the equivalence class of minimal projections are \begin{equation*}
            \oAinfgraph
        \end{equation*} 
    \item In the color case, $\mc{P}$ has elements $P_1$ and $Q_1$ given in Definition \ref{Afincolorgenerators} with relations (iv) through (iv) given in Definition \ref{Afincolorrelations}. Further, we get the equivalence classes of minimal projections are
    \begin{equation*}
            \rbAinfgraph.
        \end{equation*}
\end{enumerate}
\end{thm}

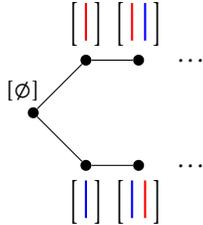
\begin{thm} (Sufficient relations for $\til{A}_\infty$, unshaded: Arrow case)
    Let $\mc{P}$ be the planar algebra generated by $P_1$ from Definition \ref{Afinarrowgenerators} with relations from Definition \ref{Afinarrowrelations}. Define $P_1$ to be self-adjoint and extend * anti-linearly and on diagrams. The $\mc{P}$ is an unshaded subfactor planar algebra who principal graph is given in the arrow case of Theorem \ref{thm: Ainfnecessary}.
\end{thm}

\begin{thm}
    The element $P_1$ from Definition \ref{Afinarrowgenerators} generates the planar algebra $\mc{P}$ from Theorem \ref{thm: Ainfnecessary} in the arrow case. 
\end{thm}

\begin{thm} (Sufficient relations for $\til{A}_\infty$, unshaded: Color Case)
    Let $\mc{P}$ be the planar algebra generated by $P_1$ and $Q_1$ from Definition \ref{Afincolorgenerators} with relations (i) through (iv) from Definition \ref{Afincolorrelations}. Define $P_1$ and $Q_1$ to be self-adjoint and extend * anti-linearly and on diagrams. Then $\mc{P}$ is an unshaded subfactor planar algebra whose principal graph is given in the color case of Theorem \ref{thm: Ainfnecessary}. 
\end{thm}

\begin{thm}
    The elements $P_1$ and $Q_1$ from Definition \ref{Afincolorgenerators} generate the subfactor planar algebra $\mc{P}$ in the color case of Theorem \ref{thm: Ainfnecessary}. 
\end{thm}

\begin{thm}
    There are exactly two nonisomorphic unshaded subfactor planar algebras with principal graph $\til{A}_\infty$.
\end{thm}

\begin{rem}
    By Popa's \cite{Pop94} classification, there are additionally type $\til{D}$ and type $\til{E}$ index 4 subfactor planar algebras. Work forthcoming by the author will display similar presentations of these subfactor planar algebras. In particular, this will complete the Kuperberg program for index 4 subfactors.
\end{rem}

\bibliographystyle{alpha}
\bibliography{Cite}

\begin{thebibliography}{EGNO15}

\bibitem[AMP23]{AMP23}
Narjess Afzaly, Scott Morrison, and David Penneys.
\newblock The classification of subfactors with index at most {$5\frac14$}.
\newblock {\em Mem. Amer. Math. Soc.}, 284(1405):v+81, 2023.

\bibitem[Big10]{Big10}
Stephen Bigelow.
\newblock Skein theory for the {$ADE$} planar algebras.
\newblock {\em J. Pure Appl. Algebra}, 214(5):658--666, 2010.

\bibitem[BPMS12]{BPMS12}
Stephen Bigelow, Emily Peters, Scott Morrison, and Noah Snyder.
\newblock Constructing the extended {H}aagerup planar algebra.
\newblock {\em Acta Math.}, 209(1):29--82, 2012.

\bibitem[Con76]{Con76}
A.~Connes.
\newblock Classification of injective factors. {C}ases {$II\sb{1},$} {$II\sb{\infty },$} {$III\sb{\lambda },$} {$\lambda \not=1$}.
\newblock {\em Ann. of Math. (2)}, 104(1):73--115, 1976.

\bibitem[Con77]{Con77}
A.~Connes.
\newblock Periodic automorphisms of the hyperfinite factor of type {II}1.
\newblock {\em Acta Sci. Math. (Szeged)}, 39(1-2):39--66, 1977.

\bibitem[Cze24]{Cze24}
Agustina Czenky.
\newblock Diagramatics for cyclic pointed fusion categories.
\newblock {\em J. Pure Appl. Algebra}, 228(12):Paper No. 107752, 2024.

\bibitem[EGNO15]{EGNO15}
Pavel Etingof, Shlomo Gelaki, Dmitri Nikshych, and Victor Ostrik.
\newblock {\em Tensor categories}, volume 205 of {\em Mathematical Surveys and Monographs}.
\newblock American Mathematical Society, Providence, RI, 2015.

\bibitem[ENO05]{ENO05}
Pavel Etingof, Dmitri Nikshych, and Viktor Ostrik.
\newblock On fusion categories.
\newblock {\em Ann. of Math. (2)}, 162(2):581--642, 2005.

\bibitem[GdlHJ89]{GHJ89}
Frederick~M. Goodman, Pierre de~la Harpe, and Vaughan F.~R. Jones.
\newblock {\em Coxeter graphs and towers of algebras}, volume~14 of {\em Mathematical Sciences Research Institute Publications}.
\newblock Springer-Verlag, New York, 1989.

\bibitem[GJS10]{GHJ10}
A.~Guionnet, V.~F.~R. Jones, and D.~Shlyakhtenko.
\newblock Random matrices, free probability, planar algebras and subfactors.
\newblock In {\em Quanta of maths}, volume~11 of {\em Clay Math. Proc.}, pages 201--239. Amer. Math. Soc., Providence, RI, 2010.

\bibitem[Han10]{Han10}
Richard Han.
\newblock {\em A {C}onstruction of the ``2221'' {P}lanar {A}lgebra}.
\newblock ProQuest LLC, Ann Arbor, MI, 2010.
\newblock Thesis (Ph.D.)--University of California, Riverside.

\bibitem[JMS14]{JMS14}
Vaughan F.~R. Jones, Scott Morrison, and Noah Snyder.
\newblock The classification of subfactors of index at most 5.
\newblock {\em Bull. Amer. Math. Soc. (N.S.)}, 51(2):277--327, 2014.

\bibitem[Jon80]{Jon80}
Vaughan F.~R. Jones.
\newblock Actions of finite groups on the hyperfinite type {${\rm II}\sb{1}$}\ factor.
\newblock {\em Mem. Amer. Math. Soc.}, 28(237):v+70, 1980.

\bibitem[Jon83]{Jon83}
V.~F.~R. Jones.
\newblock Index for subfactors.
\newblock {\em Invent. Math.}, 72(1):1--25, 1983.

\bibitem[Jon85]{Jon85}
Vaughan F.~R. Jones.
\newblock A polynomial invariant for knots via von {N}eumann algebras.
\newblock {\em Bull. Amer. Math. Soc. (N.S.)}, 12(1):103--111, 1985.

\bibitem[Jon99]{Jon99}
Vaughan F.~R. Jones.
\newblock Planar algebras, i, 1999.

\bibitem[JS97]{JS97}
V.~Jones and V.~S. Sunder.
\newblock {\em Introduction to subfactors}, volume 234 of {\em London Mathematical Society Lecture Note Series}.
\newblock Cambridge University Press, Cambridge, 1997.

\bibitem[Kau87]{Kau87}
Louis~H. Kauffman.
\newblock State models and the {J}ones polynomial.
\newblock {\em Topology}, 26(3):395--407, 1987.

\bibitem[Kle78]{Kle78}
Felix Klein.
\newblock Ueber die {T}ransformation siebenter {O}rdnung der elliptischen {F}unctionen.
\newblock {\em Math. Ann.}, 14(3):428--471, 1878.

\bibitem[McK80]{McK80}
John McKay.
\newblock Graphs, singularities, and finite groups.
\newblock In {\em The {S}anta {C}ruz {C}onference on {F}inite {G}roups ({U}niv. {C}alifornia, {S}anta {C}ruz, {C}alif., 1979)}, volume~37 of {\em Proc. Sympos. Pure Math.}, pages 183--186. Amer. Math. Soc., Providence, RI, 1980.

\bibitem[ML98]{ML98}
Saunders Mac~Lane.
\newblock {\em Categories for the working mathematician}, volume~5 of {\em Graduate Texts in Mathematics}.
\newblock Springer-Verlag, New York, second edition, 1998.

\bibitem[MP15]{MP15}
Scott Morrison and David Penneys.
\newblock Constructing spoke subfactors using the jellyfish algorithm.
\newblock {\em Trans. Amer. Math. Soc.}, 367(5):3257--3298, 2015.

\bibitem[MPS10]{MPS10}
Scott Morrison, Emily Peters, and Noah Snyder.
\newblock Skein theory for the {$D_{2n}$} planar algebras.
\newblock {\em J. Pure Appl. Algebra}, 214(2):117--139, 2010.

\bibitem[NT60a]{NT60a}
Masahiro Nakamura and Zir\^o Takeda.
\newblock A {G}alois theory for finite factors.
\newblock {\em Proc. Japan Acad.}, 36:258--260, 1960.

\bibitem[NT60b]{NT60b}
Masahiro Nakamura and Zir\^o Takeda.
\newblock On the fundamental theorem of the {G}alois theory for finite factors.
\newblock {\em Proc. Japan Acad.}, 36:313--318, 1960.

\bibitem[Ocn88]{Ocn88}
Adrian Ocneanu.
\newblock Quantized groups, string algebras and {G}alois theory for algebras.
\newblock In {\em Operator algebras and applications, {V}ol.\ 2}, volume 136 of {\em London Math. Soc. Lecture Note Ser.}, pages 119--172. Cambridge Univ. Press, Cambridge, 1988.

\bibitem[Pet10]{Pet10}
Emily Peters.
\newblock A planar algebra construction of the {H}aagerup subfactor.
\newblock {\em Internat. J. Math.}, 21(8):987--1045, 2010.

\bibitem[Pop94]{Pop94}
Sorin Popa.
\newblock Classification of amenable subfactors of type {II}.
\newblock {\em Acta Math.}, 172(2):163--255, 1994.

\bibitem[Pop95]{Pop95}
Sorin Popa.
\newblock An axiomatization of the lattice of higher relative commutants of a subfactor.
\newblock {\em Invent. Math.}, 120(3):427--445, 1995.

\bibitem[Rey23]{Rey23}
Ryan Reynolds.
\newblock {\em Diagrammatic Categories which arise from Directed Graphs}.
\newblock Phd thesis, University of Oklahoma, May 2023.
\newblock Available at \url{https://shareok.org/handle/11244/337707}.

\bibitem[TL71]{TL71}
H.~N.~V. Temperley and E.~H. Lieb.
\newblock Relations between the ``percolation'' and ``colouring'' problem and other graph-theoretical problems associated with regular planar lattices: some exact results for the ``percolation'' problem.
\newblock {\em Proc. Roy. Soc. London Ser. A}, 322(1549):251--280, 1971.

\bibitem[vN49]{vN49}
John von Neumann.
\newblock On rings of operators. {R}eduction theory.
\newblock {\em Ann. of Math. (2)}, 50:401--485, 1949.

\bibitem[Wen87]{Wen87}
Hans Wenzl.
\newblock On sequences of projections.
\newblock {\em C. R. Math. Rep. Acad. Sci. Canada}, 9(1):5--9, 1987.

\end{thebibliography}

\end{document}